\documentclass[journal,10pt,draftclsnofoot,onecolumn,onehalfspacing]{IEEEtran}

\usepackage{cite}
\usepackage{amsmath,amssymb,amsfonts}

\usepackage{psfrag}
\usepackage{mathtools, cuted}

\usepackage{graphicx}

\usepackage{textcomp}
\usepackage{xcolor}
\usepackage{cite}
\usepackage[square,numbers]{natbib}
\usepackage{amsmath}
\usepackage{amssymb}
\usepackage{amsfonts}
\usepackage{algpseudocode}
\usepackage{algorithm} 
\usepackage{amsthm}
\usepackage{color}
\usepackage{comment}
\usepackage{amssymb}
\usepackage{graphicx}
\usepackage{lipsum}
\usepackage[colorinlistoftodos]{todonotes}
\usepackage{textcomp}
\usepackage{xcolor}
\usepackage[numbers]{natbib}
\usepackage{multicol}
\newtheorem{theorem}{Theorem}
\newtheorem{proposition}{Proposition}
\newtheorem{corollary}{Corollary}
\newtheorem{fact}{Fact}

\newtheorem{remark}{Remark}

\newtheorem{lemma}{Lemma}

\newtheorem{assumption}{Assumption}
\newtheorem{condition}{Condition}
\usepackage{soul}
\usepackage{hyperref}

\algnewcommand\INPUT{\item[\textbf{Input:}]}
\algnewcommand\OUTPUT{\item[\textbf{Output:}]}

\usepackage{amsmath}

\usepackage{algpseudocode}
\usepackage{algorithm}

\usepackage{amsmath,amssymb,amsfonts}

\usepackage{psfrag}
\usepackage{mathtools, cuted}

\usepackage{graphicx}

\begin{document}

\title{Straggler-Robust Distributed Optimization in Parameter-Server Networks
\thanks{}
}

\author{Elie Atallah, Nazanin Rahnavard~\IEEEmembership{Senior Member,~IEEE}, and Chinwendu Enyioha
        \\
        \IEEEauthorblockA{{Department of Electrical and Computer Engineering}\\
{University of Central Florida, Orlando, FL}\\
Emails: \{elieatallah@knights., nazanin@eecs., cenyioha@\}ucf.edu}
}

\maketitle

\begin{abstract}
Optimization in distributed networks plays a central role in almost all distributed machine learning problems. In principle, the use of distributed task allocation has reduced the computational time, allowing better response rates and higher data reliability. However, for these computational algorithms to run effectively in complex distributed systems, the algorithms ought to compensate for communication asynchrony, network node failures and delays known as stragglers. These issues can change the effective connection topology of the network, which may vary over time, thus hindering the optimization process. In this paper, we propose a new distributed unconstrained optimization algorithm for minimizing a convex function which is adaptable to a parameter server network. In particular, the network worker nodes solve their local optimization problems, allowing the computation of their local coded gradients, which will be sent to different server nodes. Then within this parameter server platform each server node aggregates its communicated local gradients, allowing convergence to the desired optimizer. This algorithm is robust to network's worker node failures, disconnection, or delaying nodes known as stragglers. One way to overcome the straggler problem is to allow coding over the network.  We further extend this coding framework to enhance the convergence of the proposed algorithm under such varying network topologies. By using coding and utilizing evaluations of gradients of uniformly bounded delay we further enhance the proposed algorithm performance. Finally, we implement the proposed scheme in MATLAB and provide comparative results demonstrating the effectiveness of the proposed framework.
\end{abstract}


\section{Introduction}
{M}{any} problems in distributed systems over the cloud, or in wireless ad hoc networks \citep{johansson2008distributed,rabbat2004distributed,ram2009distributed}, are formulated as convex optimization programs in a parallel computing scheme. Depending on the computation architecture of these networks; that is, centralized, decentralized or fully distributed, the optimization techniques are adapted to accommodate such structures. However, the malfunctioning of processors directly impacts the overall performance of parallel computing. This malfunctioning is referred to as the \emph{straggling problem}. Many applications, whether over the cloud or in local distributed networks, have experienced considerable time delays, due in part to this straggling problem. Asynchronous \cite{li2014communication}, \cite{ho2013more} and synchronous algorithms \cite{zaharia2008improving}, \cite{chen2016revisiting} have been proposed to overcome this problem.  
While Lee et al. \cite{lee2016speeding} and Dutta et al. \cite{dutta2016short} describe techniques for mitigating stragglers in different applications, a recent work by Tandon et al. \cite{tandon17a} focused on codes for recovering the batch gradient of a loss function (i.e., synchronous gradient descent). 
Specifically, a coding scheme in \cite{tandon17a} was proposed, enabling a distributed division of tasks into uncoded (naive) and coded parts. This partition alleviates the effect of straggling servers in a trade-off between computational complexity, communication complexity and time delay.
This novel coding scheme solves this problem by providing robustness to partial failure or delay of nodes in a centralized master/workers network. \\ 
For example, in machine learning applications, gradient descent algorithms are employed to optimize parameters of the model. When the size of the training datasets are large, it is not practical to train the model on a single machine. To speed up the used algorithms, gradient computations can be distributed among multiple machines. In a centralized network framework, a parameter server platform with synchronous gradient descent updating, consists of a number of servers which are connected to multiple workers where each worker computes a partial gradient from its local dataset. Then the servers aggregate the partial gradients to obtain a full gradient and solve the model. This aggregation can be done under a full-gradient/ unique-estimate scheme or under distributed partial-gradients/ non-unique-estimate scheme utilizing the stochastic gradient descent approach and achieving consensus. To overcome the effect of delaying nodes, failures, or disconnections, the data is distributed among the workers in a redundant manner, and workers return coded computations to the servers according to the aforementioned coding scheme. \\
To this end, we present a novel algorithm implementable on a parameter server network, with multibus connection between server and worker nodes under a synchronous updating paradigm, which is robust to stragglers by utilizing coding schemes and allowing asynchronous implementation through the use of stale gradients with a uniformly bounded delay. Asynchronous evaluation of gradients with uniform bounded delay is realized with no other statistical assumption. 
Under such relaxed assumption we analyze the convergence of our algorithm and find its convergence rate. 
We also provide numerical simulations to back our theoretical analysis.

\nocite{*}

\section{Problem Setup and Background}\label{sec:problem}

We consider a network of $ n $ server nodes indexed by $ V=\{1,2,\ldots,n\} $ and $ m  $ worker nodes on a parameter server platform using a multi-bus multiprocessor system with shared memory. Thus, we require arbitrary interleaved connections according to availability. The objective is to solve a minimization problem
\begin{equation}\label{eq1}
\begin{split}
\hat{{\bf x}} = arg \min_{{\bf x} \in  \mathbb{R}^{N}}  {f({\bf x})} = \sum_{(\iota)=1}^{p} f^{(\iota)}(x) \\
\end{split}
\end{equation}

Hence, the global function $ f $ to be minimized is divided into $ p $ partitions with arbitrary number of replication for each. After dividing the load into different partitions each partition $ (\iota) \in \{1,\ldots,p\} $ is distributed with an arbitrary redundancy $ \aleph_{(\iota)} $ among the workers. And each replica $ r_{(\iota)} $ of partition $ (\iota) $ consists of $ n_{r_{(\iota)}} $ worker nodes that utilize a gradient coding similar to that in \citep{tandon17a} to enable robustness to an allowed number of stragglers. 

%
To solve Problem \eqref{eq1}, we use a robust gradient-based algorithm which we call the \emph{Straggler-Robust Distributed Optimization} (SRDO) Algorithm.
As its name infers our algorithm has the extra feature of being robust to stragglers. SRDO uses gradient coding to mitigate stragglers and delayed gradient information.

\begin{figure}[H]
\hspace{-0.2cm}
\includegraphics[scale=0.35, bb=0 0 500 500]{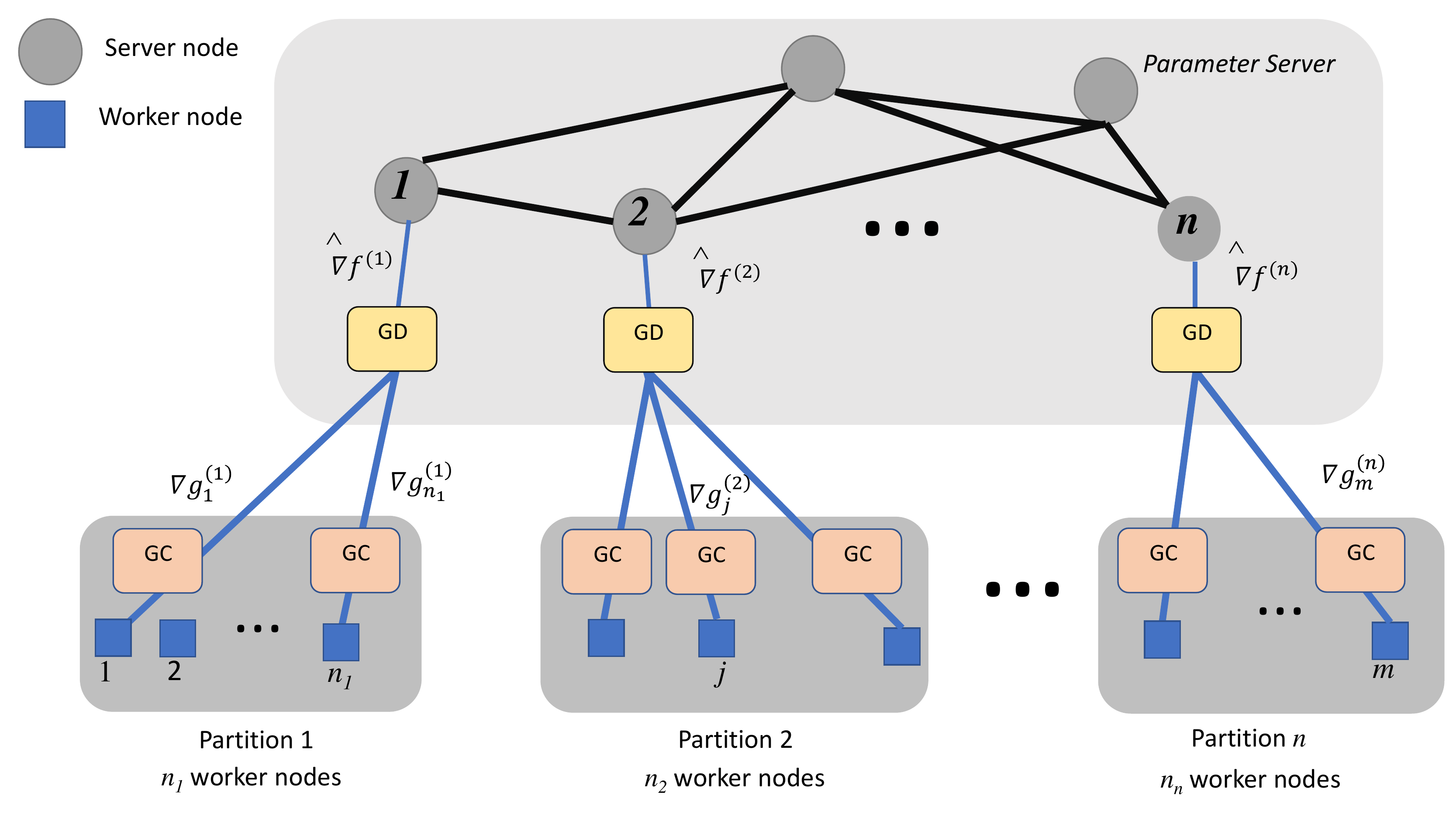}
\label{Parameter Server}
\caption{Parameter server network in the special case with $ n $ server nodes and $ m $ workers nodes. The worker nodes are  divided into $ p = n  $ partitions where partition $ i $ has $ n_{r_{(\iota)}} $ workers and each server is connected to its unique partition at all iterations.}
\label{Fig.1}
\end{figure}

\subsection{Gradient Coding Scheme}
\label{Gradient Coding}

As previously discussed, when solving Problem \eqref{eq1} using a distributed synchronous gradient descent worker nodes may be stragglers either due to failure, being compromised or just communicaiton issues with relaying information to the server nodes.  To address this probem, Tandon et al. \cite{tandon17a} proposed replicating some data across machines via a coding scheme, which allows for the recovery of the overall gradient from the aggregated local gradients of the connected nodes active in the network at a specific time step.

In \cite{tandon17a}, the authors derived a lower bound on the structure of the coding partition scheme that allows the computation of the overall gradient in the presence of $s$ or fewer stragglers; that is, if we have fewer stragglers than the maximum allowed, we can use any $ \nu - s $ combination of the connected nodes where $ \nu $ is the total number of worker nodes. 

To successfully code and decode the overall gradient when the number of stragglers is less than $s$, we require a scheme in which \cite{tandon17a}: 
\begin{equation}
{\bf B} \in \mathbb{R}^{\nu \times d}, \ \  {\bf A} \in \mathbb{R}^{{{\nu}\choose{s}} \times \nu}  \ \ \text{and}  \quad {\bf A}{\bf B}=\boldmath{1}_{{{\nu}\choose{s}} \times d}
\end{equation}
where the decoding matrix ${\bf A}$ and the encoding matrix ${\bf B}$ can be calculated from Algorithms 1 and 2 in \cite{tandon17a}, respectively (see Appendix~H).

We exploit the above coded scheme to compute $ \nabla{f}^{(\iota)} $, where $ \nabla{f}= \sum_{(\iota)=1}^{p}\nabla{f}^{(\iota)} $. Thus, we apply this coding scheme to compute $ \nabla{f}^{(\iota)} = \sum_{\lambda=1}^{n_{r_{(\iota)}}}f^{r_{(\iota)}}_{\lambda} $ (i.e., in the coding scheme, the number of data partitions $ d $ is without loss of generality equal to the number of nodes $ \nu $; therefore, in our partition's replica $ r_{(\iota)} $, $ d = n_{r_{(\iota)}} $, the number of nodes of replica $ r_{(\iota)} $ of partition $ (\iota) $ and the number of maximum allowed number of straggler is $ s_{r_{(\iota)}} $). Thus 
{ \begin{equation}\label{encdecmatrices}
\begin{split}
& {\bf B}^{r_{(\iota)}} \in \mathbb{R}^{\nu \times n_{r_{(\iota)}}}, \ \  {\bf A}^{r_{(\iota)}} \in \mathbb{R}^{{{\nu}\choose{s_{r_{(\iota)}}}} \times \nu},  \ \ \\ & \text{and}  \quad {\bf A}^{r_{(\iota)}}{\bf B}^{r_{(\iota)}}=\boldmath{1}_{{{\nu}\choose{s_{r_{(\iota)}}}} \times n_{r_{(\iota)}}},
\end{split}
\end{equation}}
are the encoding and decoding matrices of the coding scheme employed at replica $ r_{(\iota)} $ of partition $ (\iota) $. 

\section{Main Algorithm: Stragglers Robust Distributed Optimization Algorithm (SRDO)}\label{sec:algorithm}

To solve Problem \eqref{eq1}, we propose a synchronous iterative gradient descent method - the SRDO algorithm, which is robust to stragglers and applicable on a time-varying network topology (i.e., arbitrary connections) with more than the allowed number of stragglers through the use of gradients evaluations with uniformly bounded delay.
As we have mentioned before, an appropriate implementable platform for the algorithm is a multi-bus distributed parameter server shared memory network.
The network is equipped with a universal clock that synchronizes the actions of its server nodes.

\begin{algorithm}
\caption{SRDO Algorithm}
\label{alg:algorithm-srdo}
\begin{algorithmic}[1]
\Statex {\textbf{Given: $ f({\bf x})$} = $ \sum_{(\iota)=1}^{p}f^{(\iota)}(x)$; to compute $x^* \in \mathbb{R}^n$}
\State \textbf{Initialization:} Server $ i $ sets $ v_{i}(0) $ arbitrarily; set $tol$; set ${\bf \epsilon}_{i}= \emph{small number}$
\While{${\bf \epsilon}_{i} > tol $}
\For {server $i = 1,...,n$}
\State Send $ v_{i}(k) $ to workers \Comment{push step}
\State Decode partition gradient $\widehat{\nabla{{f}}}^{(\iota)}(k)$.  (SRDO-1) \Comment{pull step}
\State 
${\bf x}_{i}(k+1)=  {\bf v}_{i}(k)-\alpha_{k}\widehat{\nabla{{f}}}^{(\iota)}(k)$. (SRDO-2)
\State ${\bf v}_{i}(k+1)= \sum_{j=1}^{n}[{\bf W}(k+1)]_{ij}{\bf x}_{j}(k+1)$. (SRDO-3)
\State $ {\bf \epsilon}_{i}= \| {\bf v}_{i}(k+1)-  {\bf v}_{i}(k) \|_{2} $ 
\EndFor
\EndWhile
\Statex {\textbf{Output}}  \ \ $ {\bf x}^{*} = {\bf x}_{i}(k) $.
\end{algorithmic}
\end{algorithm}

\textbf{Distribution of Load:} At the beginning of the algorithm, the total load is divided into $ p $ partitions and each partition $ (\iota) $ is replicated $ \aleph_{(\iota)} $ times. Each replica $ r_{(\iota)} $ of a partition $ (\iota) $ consists of a number of worker nodes $ n_{r_{(\iota)}} $ and utilizes a coding scheme in the distribution of the load upon its workers (see Section~\ref{Gradient Coding}). We note that not only replicas of different partitions employ different coding schemes but also the coding scheme need not be unanimous among replicas of the same partition. 

After the distribution of the load in the distribution step accordingly, we can implement the algorithm as follows:
\newline
\noindent \textbf{Initialization:} Each server node $ q $ at global iteration $ k = 0 $ begins with random weighted average $ {\bf v}_{q}(0) \in \mathbb{R}^{N} $ for $ q \in \{1,2,\dots,n\} $, and sends ${\bf v}_{q}(0)$ to arbitrary number of worker nodes in the \underline{push step}, as shown in Fig. 2.
\newline
\noindent {\bf Push step:} Under a global clock each server $ q $ sends a message containing the weighted average $ {\bf v}_{q}(k) $ to an arbitrary number of workers (i.e., that could be in different partitions' replicas). Using the weighted averages received, worker nodes compute the coded local gradients. 
\newline
\noindent {\bf Pull step:} Under a global clock each server $ i $ gets activated and calls for coded gradients from an arbitrary partition's replica. Without loss of generality, let us assume that the connected replica is $ r_{(\iota)} $. Thus, the worker node $ w $ from replica $ r_{(\iota)} $ of partition $ (\iota) $ will send the coded gradient $ \nabla{g}_{w}^{r_{(\iota)}} $ to the connected server $ i $. After server $i$ receives the coded gradients from the sending worker nodes of the partition, it decodes the gradient of partition $ \nabla{f}^{(\iota)} $. Using the decoded gradient, server $i$ computes the estimate $ {\bf x}_{i}(k+1) $ according to (SRDO-2). 
\newline
\noindent {\bf Consensus step:} Under a global clock each server $ i $ gets activated again and computes its weighted average $ {\bf v}_{i}(k+1) $ from its connected servers estimates $ {\bf x}_{j}(k+1)$ according to (SRDO-3). The iterative process, summarized in Algorithm \ref{alg:algorithm-srdo}, continues until the weighted averages between consecutive estimates are small enough.
\begin{remark}
Here, we seek arbitrary server to worker connection in the push steps and arbitrary partition to server connection in the pull steps. That is, servers are not only connected to the same partition in the pull step all the time.
\end{remark}

On the workers side: \\
{\bf Worker Computation of Coded Gradient:} When a worker $ w $ of a replica $ r_{(\iota)} $ of partition $ (\iota) $ receives a weighted average $ {\bf v}_{q}(k) $ from a server $ q $ it computes its coded gradient $ \nabla{g}_{w}^{r_{(\iota)}} $ evaluated at $ {\bf v}_{q}(k) $ relative to the coding scheme used on replica $ r_{(\iota)} $ of partition $ (\iota) $. Thus, the coded local gradient $(\nabla{g_{w}}^{r_{(\iota)}})^{T}={\bf B}_{w}^{r_{(\iota)}} \overline{\nabla{{f}}^{r_{(\iota)}}} $ evaluated at $ {\bf v}_{q}(k) $ where $ f^{(\iota)} $ corresponds to the partition $ (\iota) $ function where $ f^{(\iota)} = \sum_{\lambda=1}^{n_{r_{(\iota)}}}f^{r_{(\iota)}}_{\lambda} $ 
and  $ \overline{\nabla{{f}}^{r_{(\iota)}}}=[(\nabla{f}^{r_{(\iota)}}_{1})^{T} \; (\nabla{f}^{r_{(\iota)}}_{2})^{T} \; \ldots \; (\nabla{f}^{r_{(\iota)}}_{n_{r_{(\iota)}}})^{T}]^{T} $. The function $f^{r_{(\iota)}}_{\lambda}$ corresponds to $ \lambda \in \{1,2,\dots,n_{r_{(\iota)}}\}$, where $ f^{(\iota)}= \sum_{\lambda=1}^{n_{r_{(\iota)}}}f^{r_{(\iota)}}_{\lambda} $, and $ n_{r_{(\iota)}} $ is the number of worker nodes in replica $ r_{(\iota)} $ of partition $ (\iota) $ (see Section~\ref{Gradient Coding}). Workers can get delayed in their computation of coded gradients and might not send their computed gradients to a connected server to their partition at a subsequent time instant. Here, there is one aspect of asynchronous behavior in the algorithm that is attributed to the computed partition gradient used at the connected server. That is, we do not require that the coded gradients of which the partition's (or connected server) partial gradient is decoded to be of the consecutive previous instant weighted averages evaluations but rather of possibly older weighted averages. That is, ${\bf v}_{q}(k-\Delta{k}(q,w,r_{(\iota)},k))$ where $ 0 \leq \Delta{k}(q,w,r_{(\iota)},k)\leq H $ and  $k-\Delta{k}(q,w,r_{(\iota)},k)$ is the instant at which server $ q $ has sent its weighted average to replica $ r_{(\iota)} $ of partition $ (\iota) $. More precisely, the received coded local gradient $\nabla{g}_{w}^{r_{(\iota)}}({\bf v}_{q}(k-\Delta{k}(q,w,r_{(\iota)},k))$ of worker $ w $ of replica $ r_{(\iota)} $ of partition $ (\iota) $ where $ 0 \leq \Delta{k}(q,w,r_{(\iota)},k)\leq H $, i.e.,  here, $g_{w}^{r_{(\iota)}}$ is a function employed due to the coding scheme used at replica $ r_{(\iota)} $ of partition $ (\iota) $, and $\nabla{g_{w}^{r_{(\iota)}}}$ corresponds to $ g_{w}^{r_{(\iota)}} =\sum_{\lambda=1}^{n_{r_{(\iota)}}}{\bf B}^{r_{(\iota)}}_{w,\lambda}f^{r_{(\iota)}}_{\lambda} $. In other words, we advance asynchronous behavior through allowing the use of coded gradients of uniformly bounded delay in the evaluation of the partitions' partial gradients utilized in the iteration update of each server. 


\begin{remark}
It is worth noting the following about the synchronous behavior of SRDO:
\begin{itemize}
\item If at the push step of iteration $ k $ a worker $ w $ of replica $ r_{(\iota)} $ is unable to receive any of the weighted averages $ {\bf v}_{q}(k-\Delta{k}(q,w,r_{(\iota)},k)) $ for all $ 0 \leq \Delta{k}(q,w,r_{(\iota)},k) \\ 
\leq H $ for all $ q \in \{1,\ldots,n\} $, that allows it to compute its local coded gradient in time  before the pull step for the same instant $ k $, then that worker is considered a straggler. A worker can also become a straggler if its computed coded gradient gets delayed in the network.
\item Each worker has to finish its computation before interacting with a server in the pull step.
In that respect, gradients evaluated at previous weighted averages $ {\bf v}_{q}(k-\Delta{k}(q,w,r_{(\iota)},k)) $ for all $ 0 \leq \Delta{k}(q,w,r_{(\iota)},k) \leq H $ for all $ q \in \{1,\ldots,n\} $, can still be used by a server as long as the worker would send its coded gradient when it finishes computation at the time $ k $ of the server nodes' synchronous update. Another scenario is when these coded gradients from prior instants are kept as stale gradients in the memory of a server and are used in the subsequent updates. These stale gradients can be used at instant $ k $ by the server if they are of weighted averages evaluations $ {\bf v}_{q}(k-\Delta{k}(q,w,r_{(\iota)},k)) $ for all $ 0 \leq \Delta{k}(q,w,r_{(\iota)},k) \leq H $ for all $ q \in \{1,\ldots,n\} $. The benefit of the latter scenario on the prior one is that it mitigates the effect of disconnection at the instant of update.
\item If at instant $ k $ server $i$ does not receive coded gradients of any of the partitions that are evaluated at weighted averages $ {\bf v}_{q}(k-\Delta{k}(q,w,r_{(\iota)},k)) $ for all $ 0 \leq \Delta{k}(q,w,r_{(\iota)},k) \leq H $ for all $ q \in \{1,\ldots,n\} $ and thus is unable to decode the partition's partial gradient, it sets  its estimate $ {\bf x}_{i}(k+1)={\bf v}_{i}(k) $ (i.e., Unanimous full straggling partitions scenario.)
\end{itemize}
\end{remark}

\subsection{Remark on the Computation of the Gradient under Different Scenarios}
We distinguish three scenarios in which the partition's gradient is computed at the pull step:
\begin{enumerate}
\item \textbf{Scenario~1; Allowed Number of Stragglers Gradient Computation Scenario}: In this scenario, the number of all partition's worker nodes disconnected to their respective server node $ i $ at the pull step is less than or equal to the maximum allowed number of stragglers (i.e., $ | \Gamma_{i,r_{(\iota)}}^{\complement}(k) | \leq s_{r_{(\iota)}} $). And the server decodes the partition's inexact gradient $ \widehat{\nabla{f}}^{(\iota)} $ by the brute application of the described coding scheme. 

\item \textbf{Scenario~2; Ignore Stragglers Gradient Computation Scenario}: In this scenario, the number of all partition's worker nodes disconnected (i.e., fail or get delayed) from their respective server node $ i $ at the pull step is greater than the maximum allowed number of stragglers, 
(i.e., $ | \Gamma_{i,r_{(\iota)}}^{\complement}(k) | > s_{r_{(\iota)}} $).
And server node $ i $ uses only the received coded local gradients from its connection set $ \Gamma_{i,r_{(\iota)}}(k) $ to compute the partition's inexact gradient $  \widehat{\nabla{f}}^{(\iota)} $. 

\item \textbf{Scenario~3; Ignore Stragglers-Stale Gradients Gradient Computation Scenario:}
In this scenario, $ | \Gamma_{i,r_{(\iota)}}^{\complement}(k) | > s_{r_{(\iota)}} $
Server node $ i $ uses the received local gradients at instant $ k $ from  its connection set $ \Gamma_{i,r_{(\iota)}}(k) $ of the connected worker nodes of partition $ \iota $ identified with server $ i $ and stale coded local gradients saved at the server at prior instants to compute the partition's inexact gradient  $ \widehat{\nabla{f}}^{(\iota)} $. 
$ 0 \leq \Delta{k}(q,w,r_{(\iota)},k)\leq H $.

\end{enumerate}

\subsection{The General Tractable Updating Step used in the Analysis}

In order to include all possible scenarios, in the SRDO step (SRDO-2) the iterate ${\bf x}_{i}(k+1)$ employing the partition's decoded gradient is calculated by the server node. i.e.,
{ \begin{equation}\label{eq5}
\begin{split}
{\bf x}_{i}(k+1)={\bf v}_{i}(k)-\alpha_{k}\widehat{\nabla{f}}^{(\iota)}(k)
\end{split}
\end{equation}}
and $ { \widehat{\nabla{f}}^{(\iota)}(k)= 
\sum_{w \in \Gamma_{i,r_{(\iota)}}(k)}{\bf A}^{r_{(\iota)}}_{fit,w}\nabla{g}^{r_{(\iota)}}_{w}({\bf v}_{q}(k-\Delta{k}(q,w,r_{(\iota)},k))) } $ \\
where $ 0 \leq \Delta{k}(q,w,r_{(\iota)},k)\leq H  $ and $ q \in V=\{1,\dots,n\} $. $ \widehat{\nabla{f}}^{(\iota)}(k) $ can be further written in a more reduced form corresponding to each index $i \in \{1,\ldots,n\} $ as:
{ \begin{equation}\label{eq5.1}
\begin{split}
& \widehat{\nabla{f}}^{(\iota)}  (k)=  \nabla{f^{(\iota)}({\bf v}_{i}(k))} 
+ \sum_{w \in \Gamma_{fit,r_{(\iota)}}}  {\bf A}^{r_{(\iota)}}_{fit,w}\sum_{\lambda=1}^{n_{r_{(\iota)}}}{\bf B}^{r_{(\iota)}}_{w,\lambda} \times \\
& (\nabla{f}^{r_{(\iota)}}_{\lambda}({\bf v}_{q}(k-\Delta{k}(q,w,r_{(\iota)},k)))-\nabla{f}^{r_{(\iota)}}_{\lambda}({\bf v}_{i}(k))).
\end{split}
\end{equation}}
Where the set of connected worker nodes of replica $ r_{(\iota)} $ of partition $ (\iota) $ at iteration $k$ pull step to server node $ i $ is $\Gamma_{i,r_{(\iota)}}(k)$, and thus, the set of stragglers to server node $ i $ as $\Gamma^{\complement}_{i, r_{(\iota)}}(k) \triangleq \{ 1,2,\dots,n_{r_{(\iota)}} \} \setminus \Gamma_{i,r_{(\iota)}}(k)$. And $ \Gamma_{s} $ is the support's column indices of row $ s $ of matrix $ {\bf A}^{r_{(\iota)}} $ defined in \eqref{encdecmatrices}. While $ fit = argmax_{s}| \Gamma_{i,r_{(\iota)}} \cap \Gamma_{s} | $, the index of the row of matrix $ {\bf A}^{r_{(\iota)}} $ whose support has maximum intersection with the indices identified with the set $ \Gamma_{i,r_{(\iota)}} $ and  $ \Gamma_{fit, r_{(\iota)}} $, and $ {\bf A}^{r_{(\iota)}}_{fit,:} $ are the corresponding column indices of the support of that row and the row itself, respectively.

\begin{remark}
The term $ \nabla{f}^{r_{(\iota)}}_{\lambda}({\bf v}_{q}(k-\Delta{k}(q,w,r_{(\iota)},k))) $ in \eqref{eq5.1} is zero if $ w \in \Gamma_{fit,r_{(\iota)}} \setminus \Gamma_{i,r_{(\iota)}}(k) $ and is as is if $ w \in \Gamma_{i,r_{(\iota)}}(k) $.
\end{remark}

\subsection{Assumptions on the Convex Functions}

\begin{assumption}\label{A3.1}
 We assume:
\begin{itemize}
\item[\textbf{(a)}]{The functions  $\mathit{f^{(\iota)}} : \mathbb{R} ^{N} \rightarrow \mathbb{R} , \ i=1\hdots, n$  are convex, differentiable and have Lipschitz gradients with constants $L_i$ over $\mathbb{R}^N$.
for all ${\bf x}$, ${\bf y} \in \mathbb{R}^{N}$. }
\item[\textbf{(b)}] {$f^{(\iota)}$ in their turn are such that $ f^{(\iota)}=\sum_{\lambda=1}^{n_{r_{(\iota)}}}f^{r_{(\iota)}}_{\lambda} $ and each $f^{r_{(\iota)}}_{\lambda}$ is convex with $ \nabla{f}^{r_{(\iota)}}_{\lambda} $ having Lipschitz constant $ \frac{L_{i}}{c_{r_{(\iota)},\lambda}} $ where $c_{r_{(\iota)},\lambda}$ is dependent on replica $ r_{(\iota)} $ of partition $(\iota)$ subpartition of functions $ f^{r_{(\iota)}}_{\lambda} $, i.e., $ c_{r_{(\iota)},\lambda} > 1 $.}
Let $ L = \max_{i}L_{i} $ then 
{ \begin{equation}\label{Lipshconstpart}
\begin{split}
\|\nabla{f}^{(\iota)}({\bf x})-\nabla{f}^{(\iota)}({\bf y})\| \leq L \|{\bf x}- {\bf y}\|
\end{split}
\end{equation}}
and 
{ \begin{equation}\label{Lipshconstsubpart}
\begin{split}
\|\nabla{f}^{r_{(\iota)}}_{\lambda}({\bf x})-\nabla{f}^{r_{(\iota)}}_{\lambda}({\bf y})\| \leq \frac{L_{i}}{c_{r_{(\iota)},\lambda}} \|{\bf x}- {\bf y}\| \leq L \|{\bf x}- {\bf y}\|
\end{split}
\end{equation}}
since $ c_{r_{(\iota)},\lambda} > 1 $.
Notice, Assumption~1(c) follows from the additivity of Lipshitz constants.

 \item[\textbf{(c)}]{ The solution set of \eqref{eq1} is denoted by $ \mathcal{X}^{*} =\{ {\bf x} | {\bf x} = \arg \min f({\bf x}) \} $. And $ {\bf x}^{*}=\min f({\bf x}) $, ${\bf x}^{(\iota)}= \min f^{(\iota)}({\bf x}) $ and $ f^{\lambda, r_{(\iota)}}=\min f^{r_{(\iota)}}_{\lambda}({\bf x}) $.}
\end{itemize}
\end{assumption}
The assumed structure on $f$ is typical for problems of this kind and enables a detailed convergence analysis.
Next, we make the following assumptions about the server-server edge weights in the consensus step.

\begin{assumption}\label{A3.3}[\textbf{Row Stochastic}]
At each iteration $k$ of Algorithm~\ref{alg:algorithm-srdo}, the matrices $ {\bf W}(k) $ in (SRDO-3) are chosen such that $ [{\bf W}(k)]_{ij} $ depends on the network server to server connection topology in a way that allows the servers to reach consensus: \\
$ \textbf{ (a) }  $ $\sum_{j=1}^{n}[{\bf W}(k)]_{ij} = 1 -\mu $ for all $i \in V$ and $ 0 < \mu < 1 $.\\
$ \textbf{ (b) } $ There exists a scalar $\nu \in (0,1)$ such $[{\bf W}(k)]_{ij} \geq \nu$ if $[{\bf W}(k)]_{ij} > 0$. \\
$ \textbf{ (c) } $ $\sum_{i=1}^{n}[{\bf W}(k)]_{ij} = 1 - \mu $ for all $j \in V$ and $ 0 < \mu < 1 $. \\
$ \textbf{ (d) } $ If server $ i $ is disconnected from server $ j $ at instant $ k $, then $[{\bf W}(k)]_{ij} = 0$.
\end{assumption}

\begin{assumption}\label{A3.4} Bounded Delayed Evaluation and Gradient Computation: 

In the pull step server $ i $ decodes the gradient $ \nabla{f}^{(\iota)}(k) $ of partition $ (\iota) $ at time $ k $ from coded gradients of workers of replica $ r_{(\iota)} $ of partition $ (\iota) $ that are evaluated from weighted averages $ {\bf v}_{q} (k - \Delta{k}(q,w,r_{(\iota)},k)) $ where $  0 \leq \Delta{k}(q,w,r_{(\iota)},k)\leq H $.
We also assume the use of stale gradients saved at the memory of each server $ i $ evaluated from weighted averages of instants $k - \Delta{k}(q,w,r_{(\iota)},k)) $ where $  0 \leq \Delta{k}(q,w,r_{(\iota)},k)\leq H $ of arbitrary server $ q \in V = \{1,2,\ldots,n\} $.
\end{assumption}

\begin{assumption}\label{A3.5} Choice of Partition by Server in Pull Step

Each server $ i $ gets connected to a partition $ (\iota) $ out of the $ p $ partitions with a probability $ \gamma_{(\iota)} $ and gets no connection with any partition with a probability $ \gamma_{(0)} $. Here, $ \gamma_{(\iota)} $ can depend not only on the redundancy $ \aleph_{(\iota)} $ but also on other factors, e.g. the traffic in the network.
\end{assumption}

\begin{assumption}\label{A3.6} Diminishing Coordinated Synchronized Stepsizes

The stepsizes $ \alpha_{i,k} $ of server $ i $ are coordinated and synchronized where $ \alpha_{i,k}=\alpha_{k} > 0 $ and , $  \alpha_{k} \rightarrow 0 $. $ \sum_{k=0}^{\infty} \alpha_{k} = \infty $, $ \sum_{k=0}^{\infty} \alpha_{k}^{2} < \infty $. For example, a unanimous stepsize $ \alpha_{i,k} = \alpha_{k} = \frac{1}{k+1} $ among all servers $ i $ per iteration $ k $.
\end{assumption}

\begin{figure}[H]
\includegraphics[scale=0.35,bb=0 0 500 500]{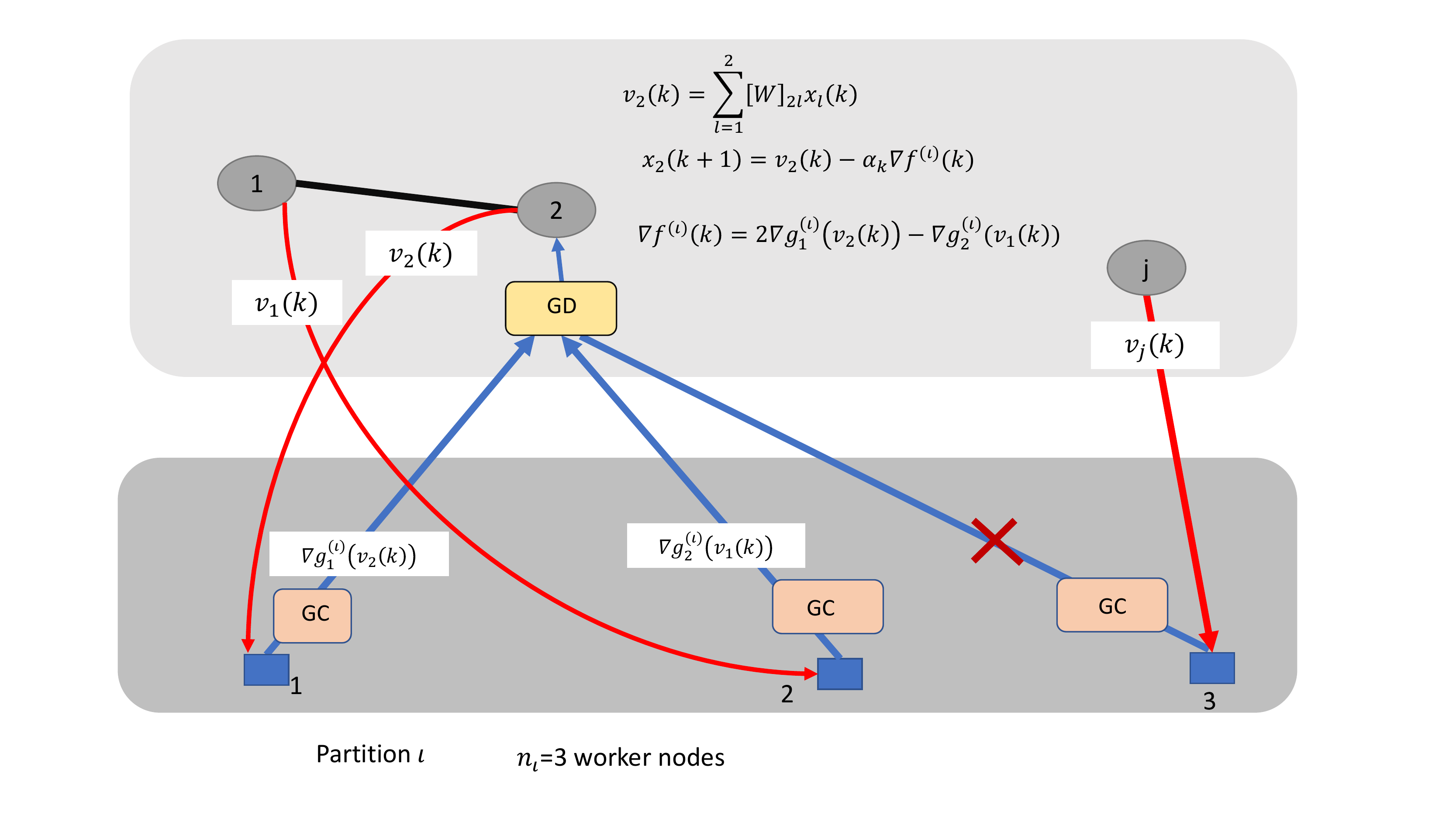}
\caption{Parameter server schematic with coding scheme $ 1 $ of allowed number of stragglers on partition $ 1 $ of $ 3 $ worker nodes connecting to sever $ 2 $. Notice that in this general scenario, the $ \widehat{\nabla}{f}_{2} $ that is calculated at server node $ 2 $ is identified with partition $ 1 $ gradient $ \widehat{\nabla}{f}^{(1)} $.}
\label{Fig.2}
\end{figure}

\section{Convergence Analysis for the Main Algorithm SRDO} \label{sec:convergence}


Let $R_i(k)$ be defined as (see \eqref{eq5} and \eqref{eq5.1})

{ \begin{equation}\label{eq6}
\begin{split}
   {\bf R}_{i,r_{(\iota)}}(k)= & \alpha_{k} {\bf \epsilon}_{i,r_{(\iota)}}(k) \\
     =  -\alpha_{k} \sum_{w \in \Gamma_{fit,r_{(\iota)}}(k)} & {\bf A}^{r_{(\iota)}}_{fit,w}\sum_{\lambda=1}^{n_{r(\iota)}}{\bf B}^{r_{(\iota)}}_{w,\lambda} (\nabla{f}^{r_{(\iota)}}_{\lambda}({\bf v}_{q}(k-\Delta{k}(q,w,r_{(\iota)},k)) - \nabla{f}^{r_{(\iota)}}_{\lambda}({\bf v}_{i}(k)))
\end{split}    
\end{equation}}
\newline
where $ 0 \leq \Delta{k}(q,w,r_{(\iota)},k) \leq H $ and $ q \in \{1,\ldots,n\} $.

Hence, from \eqref{eq5} and \eqref{eq5.1} we have

\begin{equation}\label{eq7}
{\bf x}_{i}(k+1)={\bf v}_{i}(k)-\alpha_{k}\nabla{f^{(\iota)}({\bf v}_{i}(k))}+R_{i}(k),
\end{equation}
or 
\begin{equation}\label{3.43}
 {\bf x}_{i}(k+1)=\overline{{\bf x}}_{i}(k+1) +{\bf R}_{i,r_{(\iota)}}(k), 
 \end{equation}
where $\overline{{\bf x}}_{i}(k+1)={\bf v}_{i}(k)-\alpha_{k}\nabla{f^{(\iota)}({\bf v}_{i}(k))}$.

\begin{condition}\label{Condition-1}
If the sequence $ \{ {\bf v}_{i}(k) \} $ generated by Algorithm~\ref{alg:algorithm-srdo} satisfy \\
$ \max\limits_{(\iota), \lambda}\|{\bf x}^{*}-{\bf x}^{\lambda, r_{(\iota)}}\| \leq \max\limits_{k-H \leq \hat{k} \leq k} \|{\bf v}_{q}(\hat{k})-{\bf x}^{*}\| $ then Algorithm~\ref{alg:algorithm-srdo} is said to satisfy Condition~\ref{Condition-1}. 
\end{condition}

\begin{remark}
Notice this is generally most probably satisfied at the beginning of the algorithm process unless when all $ {\bf x}^{\lambda,r_{(\iota)}} = {\bf x}^{*} $ for all replicas $ r_{(\iota)} $ of partitions $ (\iota) $ where it is satisfied throughout the algorithm process.
\end{remark}

We partition Scenario~1, 2 and 3 into two parts:
\begin{itemize}
\item{{\bf Division~1:} includes Scenario~1, 3 , and Scenario~2 where Scenario~2 satisfies Condition~\ref{Condition-1}.}
\item{{\bf Division~2:} includes Scenario~2 not satisfying Condition~\ref{Condition-1}.}
\end{itemize}

\subsection{Evaluation of $ \sum_{j=1}^{n} \| {\bf R}_{j,r_{(\iota)}}(k) \| ^{2} $}\label{app:evaluateR}

Then for scenarios included in Division~1 we have
{ \begin{align*}
 \| {\bf R}_{j,r_{(\iota)}}(k) \|   \leq \alpha_{k}  \max\limits_{(\iota)}\| {\bf A}^{r_{(\iota)}} \| _{\infty} \|{\bf B}^{r_{(\iota)}}\| _{2,\infty} ( 2 L \max\limits_{k-H \leq \hat{k} \leq k, q \in V} \| {\bf v}_{q}(\hat{k}) -  x^{*}\| ) 
\end{align*}}
%

%
{ \begin{align*}
 \| {\bf R}_{j,r_{(\iota)}}(k) \| ^{2} & \leq \alpha_{k} ^{2} \max\limits_{(\iota)} \| {\bf A}^{r_{(\iota)}} \| ^{2} _{\infty} \|{\bf B}^{r_{(\iota)}}\| _{2,\infty} ^{2}  
 ( 4 L^{2}  \max\limits_{k-H \leq \hat{k} \leq k, q \in V} \| {\bf v}_{q}(\hat{k}) -  x^{*}\|^{2} ) 
\end{align*}}
For scenarios in Division~2, that is, Scenario~2 where Condition~\ref{Condition-1} is not satisfied, we have

{ \begin{align*}
 \| {\bf R}_{j,r_{(\iota)}}(k) \|  & \leq \alpha_{k}  \max\limits_{(\iota)}\| {\bf A}^{r_{(\iota)}} \| _{\infty} \|{\bf B}^{r_{(\iota)}}\| _{2,\infty} \\
 & \times ( 2 L \max\limits_{k-H \leq \hat{k} \leq k} \| {\bf v}_{q}(\hat{k}) -  x^{*}\| + L \max\limits_{(\iota), \lambda} \| {\bf x}^{*} - {\bf x}^{\lambda,r_{(\iota)}} \|) 
\end{align*}}
%
%
{ \begin{align*}
 \| {\bf R}_{j,r_{(\iota)}}(k) \|^{2}  &\leq \alpha_{k}^{2}  \max\limits_{(\iota)}\| {\bf A}^{r_{(\iota)}} \|^{2} _{\infty} \|{\bf B}^{r_{(\iota)}}\|^{2} _{2,\infty} \\
 & \times( 4 L^{2} \max\limits_{k-H \leq \hat{k} \leq k} \| {\bf v}_{q}(\hat{k}) -  x^{*}\| + 2 L^{2} \max\limits_{(\iota), \lambda} \| {\bf x}^{*} - {\bf x}^{\lambda,r_{(\iota)}} \|^{2}) 
\end{align*}}

And we can get the upper bounds for $ \| {\bf \epsilon}_{j,r_{(\iota)}}(k) \| $ and $ \| {\bf \epsilon}_{j,r_{(\iota)}}(k) \| ^{2} $ for both divisions by substituting $ {\bf \epsilon}_{j,r_{(\iota)}}(k) = \frac{1}{\alpha_{k}} \| {\bf R}_{j,r_{(\iota)}}(k) \| $, respectively.
For all scenarios irrespective of which part they belong, we have
{ \begin{align*}
 \| {\bf R}_{j,r_{(\iota)}}(k) \|  & \leq \alpha_{k}  \max\limits_{(\iota)}\| {\bf A}^{r_{(\iota)}} \| _{\infty} \|{\bf B}^{r_{(\iota)}}\| _{2,\infty} \\
 & \times ( 2 L \max\limits_{k-H \leq \hat{k} \leq k} \| {\bf v}_{q}(\hat{k}) -  x^{*}\| + L \max\limits_{(\iota), \lambda} \| {\bf x}^{*} - {\bf x}^{\lambda,r_{(\iota)}} \|) 
\end{align*}}
%
%
{ \begin{align*}
 \| {\bf R}_{j,r_{(\iota)}}(k) \|^{2}  & \leq \alpha_{k}^{2}  \max\limits_{(\iota)}\| {\bf A}^{r_{(\iota)}} \|^{2} _{\infty} \|{\bf B}^{r_{(\iota)}}\|^{2} _{2,\infty} \\
 & \hspace{-0.5cm} \times ( 4 L^{2} \max\limits_{k-H \leq \hat{k} \leq k} \| {\bf v}_{q}(\hat{k}) -  x^{*}\| + 2 L^{2} \max\limits_{(\iota), \lambda} \| {\bf x}^{*} - {\bf x}^{\lambda,r_{(\iota)}} \|^{2}) 
\end{align*}}
%


\begin{proof}
See Appendix~A for the results in this subsection.
\end{proof}

%
%

And then
\begin{equation}\label{3.67a}
\begin{split}
\sum_{j=1}^{n} & \| {\bf \epsilon}_{j,r_{(\iota)}} (k) \|  \leq  \max\limits_{(\iota)} \|{\bf A}^{r_{(\iota)}} \| _{\infty} \|{\bf B}^{r_{(\iota)}}\| _{2,\infty} \\
& \times ( 2 L \max\limits_{k-H \leq \hat{k} \leq k} \sum_{j=1}^{n}\| {\bf v}_{q}(\hat{k}) -  x^{*}\| + n L \max\limits_{(\iota), \lambda} \| {\bf x}^{*} - {\bf x}^{\lambda,r_{(\iota)}} \| ) \end{split}
\end{equation}
And squaring both sides and using $ 2ab \leq a^{2} + b^{2} $, we have
\begin{equation}\label{3.67a}
\begin{split}
\sum_{j=1}^{n} & \| {\bf \epsilon}_{j,r_{(\iota)}} (k) \|^{2}  \leq  \max\limits_{(\iota)} \|{\bf A}^{r_{(\iota)}} \|^{2} _{\infty} \|{\bf B}^{r_{(\iota)}}\|^{2} _{2,\infty}  \\
& \times ( 4 L^{2}\max\limits_{k-H \leq \hat{k} \leq k} \sum_{j=1}^{n}\| {\bf v}_{q}(\hat{k}) -  x^{*}\| + 2 n L^{2} \max\limits_{(\iota), \lambda} \| {\bf x}^{*} - {\bf x}^{\lambda,r_{(\iota)}} \| ) \end{split}
\end{equation}

\begin{lemma}\label{general_error_bound}
Let Assumptions \ref{A3.1}, \ref{A3.3}, \ref{A3.5} and \ref{A3.6} hold. Let the sequences $ \{{\bf x}_{i}(k)\} $ and $ \{ {\bf v}_{i}(k) \} $ be generated by Algorithm~\ref{alg:algorithm-srdo}. Then we have 
{\small \begin{equation}\label{ineq_general_error_bound}
\begin{split}
    & \sum_{l=1}^{n}\|  {\bf v}_{l}(k+1)-  {\bf x}^{*}\|^{2}  \leq (1-\mu)\sum_{j=1}^{n}\| {\bf v}_{j}(k)- {\bf x}^{*}\|^{2} \\
    & + 2\alpha_{k}\sum_{j=1}^{n}\sum_{(\iota)=1}^{p}\gamma_{(\iota)}\langle \nabla{f}^{(\iota)}( {\bf v}_{j}(k)), {\bf x}^{*}- {\bf v}_{j}(k)\rangle 
    + 2 \alpha_{k}\sum_{j=1}^{n} \sum_{(\iota)=1}^{p}\gamma_{(\iota)} \langle {\bf \epsilon}_{j,r_{(\iota)}}(k), {\bf v}_{j}(k)- {\bf x}^{*}\rangle  \\
    & + 2 (1 - \gamma_{(0)})\alpha_{k}^{2} \sum_{j=1}^{n}\sum_{(\iota)=1}^{p}\gamma_{(\iota)}[\|\nabla{f}^{(\iota)}( {\bf v}_{j}(k))\|^{2} + \| {\bf \epsilon}_{j,r_{(\iota)}}(k)\|^{2}]
\end{split}
\end{equation}}
\end{lemma}

\begin{proof}
See Appendix~B for proof. 
\end{proof}

\subsection{Convergence Results for SRDO in Scenario~1, 3 and in Scenario 2 satisfying Condition~\ref{Condition-1}}

\begin{lemma}\label{type1_error_bound}
Let Assumptions \ref{A3.1}, \ref{A3.3}, \ref{A3.5} and \ref{A3.6} hold. Let the functions in Assumption~\ref{A3.1} satisfy  $ f^{(\iota)}( {\bf x}^{*}) = f^{(\iota)}( {\bf x}^{(\iota)}) $ for all $ (\iota) $. Let the sequences $ \{{\bf x}_{i}(k)\} $ and $ \{ {\bf v}_{i}(k) \} $, $ i \in V $ be generated by Algorithm~\ref{alg:algorithm-srdo}. Then we have 
{ \begin{equation}\label{ineq_type1_error_bound}
\begin{split}
    & \sum_{l=1}^{n}\|  {\bf v}_{l}(k+1) - {\bf x}^{*}\|^{2}  \leq (1-\mu) \sum_{j=1}^{n}\|{\bf v}_{j}(k)- {\bf x}^{*}\|^{2} \\
   & + 2 \alpha_{k} \sum_{j=1}^{n}\sum_{(\iota)=1}^{p}\gamma_{(\iota)} \langle {\bf \epsilon}_{j,r_{(\iota)}}(k), {\bf v}_{j}(k) - {\bf x}^{*} \rangle  
   + 2 (1 -\gamma_{(0)}) \alpha_{k}^{2} \sum_{j=1}^{n}\sum_{(\iota)=1}^{p}\gamma_{(\iota)}\|{\bf \epsilon}_{j,r_{(\iota)}}(k)\|^{2} \\
   & - 2 \alpha_{k} \sum_{j=1}^{n}\sum_{(\iota)=1}^{p}\gamma_{(\iota)} a_{v_{j}(k),x^{*}}(b_{v_{j}(k),x^{*}} - (1 -\gamma_{(0)}) \alpha_{k} a_{v_{j}(k),x^{*}} ) \| {\bf v}_{j}(k) - {\bf x}^{*} \| ^{2}  
\end{split}
\end{equation}}
\end{lemma}

\begin{proof}
See Appendix~C for proof. 
\end{proof}

\begin{proposition}\label{Convergence_type1}
Let Assumptions \ref{A3.1}-\ref{A3.6} hold. Let the functions in Assumption~\ref{A3.1} be strongly convex and satisfy  $ f^{(\iota)}( {\bf x}^{*}) = f^{(\iota)}( {\bf x}^{(\iota)}) $ for all $ (\iota) $. Let the sequences $ \{ {\bf x}_{i}(k) \} $ and $ \{ {\bf v}_{i}(k) \} $, $ i \in V  $ be generated by Algorithm~\ref{alg:algorithm-srdo} under scenarios of Division~1 with stepsizes and errors as given in Assumptions \ref{A3.4} and \ref{A3.6}. Assume that problem \eqref{eq1} has a non-empty optimal solution set $ \mathcal{X}^{*} $ as given in Assumption~\ref{A3.1}. Then, the sequences $ \{ {\bf x}_{i}(k) \} $ and $ \{ {\bf v}_{i}(k) \} $, $ i \in V  $ converge to the same random point in $ \mathcal{X}^{*} $ with probability 1.
\end{proposition}

\begin{proof}

With errors as in Assumption~\ref{A3.4} we have $ \|R_{i}(k)\|$ and $\|{\bf \epsilon}_{j,r_{(\iota)}}(k)\|$ as given in Appendix or Section~\ref{app:evaluateR} for Division~1 scenarios.
Then having $ f^{(\iota)}( {\bf x}^{*}) = f^{(\iota)}( {\bf x}^{(\iota)}) $ for all $ (\iota) $ also then we have Lemma~\ref{type1_error_bound} satisfied. Then we can use  the resulting inequality \eqref{ineq_type1_error_bound} with the substitution of $ \| {\bf \epsilon}_{j,r_{(\iota)}}(k) \| $ from Appendix to get 
{ \begin{equation}\label{ineq_type1_error_bound_subst}
\begin{split}
& \sum_{l=1}^{n}\|  {\bf v}_{l}(k+1) -  {\bf x}^{*}\|^{2}  \leq (1-\mu)\sum_{j=1}^{n}\| {\bf v}_{j}(k) - {\bf x}^{*}\|^{2}  \\
& + 4 p L \alpha_{k} \sum_{j=1}^{n} \max_{(\iota)}
\| {\bf A}^{r_{(\iota)}} \| _{\infty} \|{\bf B}^{r_{(\iota)}}\| _{2,\infty}  \max_{k-H \leq \hat{k} \leq k ; q \in V} \| {\bf v}_{q}(\hat{k}) -  {\bf x}^{*}\|^{2}   \\
& + 8 p (1 -\gamma_{(0)}) L^{2} \alpha_{k}^{2} \sum_{j=1}^{n} \max_{(\iota)} \| {\bf A}^{r_{(\iota)}} \| ^{2} _{\infty} \|{\bf B}^{r_{(\iota)}}\|^{2} _{2,\infty} \max_{k-H \leq \hat{k} \leq k ; q \in V } \| {\bf v}_{q}(\hat{k}) -  {\bf x}^{*}\|^{2} \\
& - 2 \alpha_{k} \sum_{j=1}^{n}\sum_{(\iota)=1}^{p}\gamma_{(\iota)} a_{v_{j}(k),x^{*}}(b_{v_{j}(k),x^{*}} -  (1 -\gamma_{(0)}) \alpha_{k} a_{v_{j}(k),x^{*}} ) \| {\bf v}_{j}(k) - {\bf x}^{*} \| ^{2}    
\end{split}
\end{equation}}
To prove convergence we employ Lemma~\ref{Lemma6a}. But in order to be able to use Lemma~\ref{Lemma6a} the last term in \eqref{ineq_type1_error_bound_subst} should be negative so that it can be deleted from the inequality. Which means $ b_{v_{j}(k),x^{*}} - (1-\gamma_{(0)}) \alpha_{k} a_{v_{j}(k),x^{*}}  \geq 0 $ where $ b_{v_{j}(k),x^{*}} = \langle \overrightarrow{u},\overrightarrow{v} \rangle $. And $ \nabla{f}^{(\iota)}({\bf v}_{j}(k)) - \nabla{f}^{(\iota)}({\bf x}^{*}) = a_{v_{j}(k),x^{*}} \| {\bf v}_{j}(k) - {\bf x}^{*} \| \overrightarrow{u} $ where $ a_{v_{j}(k),x^{*}} \leq L $.
However, $ f^{(\iota)} $ is strongly convex for every $ (\iota) $, then $ \langle \nabla{f}^{(\iota)}({\bf v}_{j}(k)) - \nabla{f}^{(\iota)}({\bf x}^{*}), {\bf v}_{j}(k) - {\bf x}^{*} \rangle  = a_{v_{j}(k),x^{*}} \| {\bf v}_{j}(k) - {\bf x}^{*} \| ^{2} \langle \overrightarrow{u},\overrightarrow{v} \rangle  \geq \sigma_{(\iota)} ^{2} \| {\bf v}_{j}(k) - {\bf x}^{*} \| ^{2} $, that is $ \langle \overrightarrow{u},\overrightarrow{v} \rangle \geq \frac{\sigma_{(\iota)}}{a_{v_{j}(k),x^{*}}} $. 
Therefore, a sufficient condition is 
\begin{equation}\label{suffcond1}
\begin{split}
 (1-\gamma_{(0)}) \alpha_{k} a_{v_{j}(k),(\iota)} \leq  (1-\gamma_{(0)}) \alpha_{k} L  \leq \frac{\sigma_{(\iota)}}{a_{v_{j}(k),(\iota)}} = \langle \overrightarrow{u},\overrightarrow{v} \rangle 
\end{split}
\end{equation}
The sufficient condition in \eqref{suffcond1} is satisfied for $ k \geq k^{c}_{1} $ since $ \alpha_{k} \rightarrow 0 $.
Then by deleting the last term of \eqref{ineq_type1_error_bound_subst} it will be similar to the martingale inequality \eqref{eqw} of Lemma~\ref{Lemma6a} for $ k \geq k^{*}_{1}= \max(k^{c}_{1},k^{+}_{1},H+1) $ where $ k^{+}_{1} $ is the iteration at which $ 0 < (1-\mu)  + 4 L \alpha_{k} \max\limits_{(\iota)} \| {\bf A}^{r_{(\iota)}} \| _{\infty} \|{\bf B}^{r_{(\iota)}}\| _{2,\infty}  + 8 (1-\gamma_{(0)}) L^{2} \alpha_{k}^{2} \max\limits_{(\iota)} \| {\bf A}^{r_{(\iota)}} \| ^{2} _{\infty} \|{\bf B}^{r_{(\iota)}}\|^{2} _{2,\infty}  < 1 $ for the first time.

 By the result of Lemma~\ref{Lemma6a} we have for $ v_{k}= \sum_{i=1}^{n} \| {\bf v}_{i}(k) - {\bf x}^{*} \|^{2} $ that
\begin{equation}
\begin{split}
    \sum_{i=1}^{n} \| {\bf v}_{i}(k) - {\bf x}^{*} \|^{2} \leq \rho_{1} ^{k}  V^{'}_{0} 
\end{split}    
\end{equation}
for $ k \geq \bar{k}_{1}+H+1 $ where $ \rho_{1}=\rho $, $ V^{'}_{0}=V_{0} $, $ H = B $ and $ \bar{k}_{1}=\bar{k}$ and $ k^{*}_{1}= k^{*} $ are as in the lemma, i.e., take $ \bar{k}_{1}=k^{*}_{1} - 1 $. Therefore, as $ k \rightarrow \infty $ we have $ \sum_{i=1}^{n} \| {\bf v}_{i}(k) - {\bf x}^{*} \|^{2} \rightarrow 0 $. That is, $  \| {\bf v}_{i}(k) - {\bf x}^{*} \| \rightarrow 0 $ for all $ i \in V $. 
Then in view of (SRDO-2) where $ {\bf x}_{i}(k+1) = {\bf v}_{i}(k) -\alpha_{k}\nabla{f}^{(\iota)}({\bf v}_{i}(k)) + {\bf R}_{i,r_{(\iota)}}(k) $ and since ${\bf R}_{i,r_{(\iota)}}(k) \rightarrow 0 $ because  $  \| {\bf v}_{i}(k) - {\bf x}^{*} \| \rightarrow 0 $, $ \alpha_{k} \rightarrow 0 $ and $ \alpha_{k} \nabla{f}^{(\iota)}({\bf v}_{i}(k))\rightarrow 0 $ since $ \alpha_{k} \rightarrow 0 $   where $ \nabla{f}^{(\iota)}({\bf v}_{i}(k)) \leq G_{f} $ since $ \nabla{f}^{(\iota)}({\bf v}_{i}(k))=  \nabla{f}^{(\iota)}({\bf v}_{i}(k))-  \nabla{f}^{(\iota)}({\bf x}^{(\iota)}) \leq L \| {\bf v}_{i}(k)- {\bf x}^{*} \| + L \| {\bf x}^{*} - {\bf x}^{(\iota)} \| $, $\| {\bf x}^{*} - {\bf x}^{(\iota)} \| =const. $ and ${\bf v}_{i}(k) \rightarrow {\bf x}^{*} $. Therefore, $ {\bf x}_{i}(k+1) = {\bf v}_{i}(k) = {\bf x}^{*} $ as $ k \rightarrow \infty $. 

\end{proof}

Proposition~\ref{Convergence_type1} follows from Proposition~\ref{Convergence_type1a} We explicitly present Proposition~\ref{Convergence_type1} to delineate the different convergence rate for the case with strongly convex functions.

\begin{proposition}\label{Convergence_type1a}
Let Assumptions \ref{A3.1}-\ref{A3.6} hold. Let the functions in Assumption~\ref{A3.1}  satisfy  $ f^{(\iota)}( {\bf x}^{*}) = f^{(\iota)}( {\bf x}^{(\iota)}) $ for all $ (\iota) $. Let the sequences $ \{ {\bf x}_{i}(k) \} $ and $ \{ {\bf v}_{i}(k) \} $, $ i \in V  $ be generated by Algorithm~\ref{alg:algorithm-srdo} under scenarios of Division~1 with stepsizes and errors as given in Assumptions \ref{A3.4} and \ref{A3.6}. Assume that problem \eqref{eq1} has a non-empty optimal solution set $ \mathcal{X}^{*} $ as given in Assumption \ref{A3.1}. Then, the sequences $ \{ {\bf x}_{i}(k) \} $ and $ \{ {\bf v}_{i}(k) \} $, $ i \in V  $ converge to the same random point in $ \mathcal{X}^{*} $ with probability 1.
\end{proposition}

\begin{proof}
With errors as in Assumption~\ref{A3.4} we have $ \|R_{i}(k)\|$ and $\|{\bf \epsilon}_{j,r_{(\iota)}}(k)\|$ as given in Appendix or Section~\ref{app:evaluateR} for Division~1 scenarios.
Then having $ f^{(\iota)}( {\bf x}^{*}) = f^{(\iota)}( {\bf x}^{(\iota)}) $ for all $ (\iota) $ also satisfied we have Lemma~\ref{type1_error_bound} satisfied. Then we can use  the resulting inequality \eqref{ineq_type1_error_bound} with the substitution of $ \| {\bf \epsilon}_{j,r_{(\iota)}}(k) \| $ from Appendix to get 
{ \begin{equation}\label{ineq_s1}
\begin{split}
& \sum_{l=1}^{n}\|  {\bf v}_{l}(k+1) -  {\bf x}^{*}\|^{2}  \leq (1-\mu)\sum_{j=1}^{n}\| {\bf v}_{j}(k) - {\bf x}^{*}\|^{2}  \\
& + 4 p L \alpha_{k} \sum_{j=1}^{n} \max_{(\iota)}
\| {\bf A}^{r_{(\iota)}} \| _{\infty} \|{\bf B}^{r_{(\iota)}}\| _{2,\infty}  \max_{k-H \leq \hat{k} \leq k ; q \in V} \| {\bf v}_{q}(\hat{k}) -  {\bf x}^{*}\|^{2}   \\
& + 8 p (1-\gamma_{(0)}) L^{2} \alpha_{k}^{2} \sum_{j=1}^{n} \max_{(\iota)} \| {\bf A}^{r_{(\iota)}} \| ^{2} _{\infty} \|{\bf B}^{r_{(\iota)}}\|^{2} _{2,\infty} \max_{k-H \leq \hat{k} \leq k ; q \in V } \| {\bf v}_{q}(\hat{k}) -  {\bf x}^{*}\|^{2} \\
& - 2 a_{v_{j}(k),x^{*}} \alpha_{k} \sum_{j=1}^{n}\sum_{(\iota)=1}^{p}\gamma_{(\iota)} (b_{v_{j}(k),x^{*}} - (1-\gamma_{(0)})\alpha_{k} a_{v_{j}(k),x^{*}} \| {\bf v}_{j}(k) - {\bf x}^{*} \| ^{2}   
\end{split}
\end{equation}}
But in order to be able to use Lemma~\ref{Lemma6a} on \eqref{ineq_s1} with the use of the convexity of $ f^{(\iota)} $ only, we must take the negative part of the last term of the RHS of the above inequality so that we have

{ \begin{equation}\label{ineq_s2}
\begin{split}
& \sum_{l=1}^{n}\|  {\bf v}_{l}(k+1) -  {\bf x}^{*}\|^{2}  \leq (1-\mu)\sum_{j=1}^{n}\| {\bf v}_{j}(k) - {\bf x}^{*}\|^{2}  \\
& + 4 L \alpha_{k} \sum_{j=1}^{n} \max_{(\iota)}
\| {\bf A}^{r_{(\iota)}} \| _{\infty} \|{\bf B}^{r_{(\iota)}}\| _{2,\infty}  \max_{k-H \leq \hat{k} \leq k ; q \in V} \| {\bf v}_{q}(\hat{k}) -  {\bf x}^{*}\|^{2}   \\
&+ 8 p (1-\gamma_{(0)}) L^{2} \alpha_{k}^{2} \sum_{j=1}^{n} \max_{(\iota)} \| {\bf A}^{r_{(\iota)}} \| ^{2} _{\infty} \|{\bf B}^{r_{(\iota)}}\|^{2} _{2,\infty} \max_{k-H \leq \hat{k} \leq k ; q \in V } \| {\bf v}_{q}(\hat{k}) -  {\bf x}^{*}\|^{2} \\
&+ 2 (1-\gamma_{(0)})\sum_{j=1}^{n}\sum_{(\iota)=1}^{p}\gamma_{(\iota)} \alpha_{k}^{2} a^{2}_{v_{j}(k),x^{*}}\max_{k-H \leq \hat{k} \leq k ; q \in V }\| {\bf v}_{q}(\hat{k}) - {\bf x}^{*} \| ^{2}   
\end{split}
\end{equation}}

Then \eqref{ineq_s2} is similar to the martingale inequality \eqref{eqw} of Lemma~\ref{Lemma6a} for $ k \geq k_{2}^{*} $ where $ k^{*}_{2} $ is the iteration at which $ 
   0 < (1-\mu)+ 4 L \alpha_{k} \max_{(\iota)}
\| {\bf A}^{r_{(\iota)}} \| _{\infty} \|{\bf B}^{r_{(\iota)}}\| _{2,\infty}  + 8 (1-\gamma_{(0)}) L^{2} \alpha_{k}^{2} \max_{(\iota)} \| {\bf A}^{r_{(\iota)}} \| ^{2} _{\infty} \|{\bf B}^{r_{(\iota)}}\|^{2} _{2,\infty} + 2 (1-\gamma_{(0)}) \alpha_{k}^{2}L^{2} < 1 $ for the first time (i.e., since $ a^{2}_{v_{j}(k),x^{*}} \leq L^{2} $).

By the result of Lemma~\ref{Lemma6a} we have for $ v_{k}= \sum_{i=1}^{n} \| {\bf v}_{i}(k) - {\bf x}^{*} \|^{2} $ that
\begin{equation}
\begin{split}
    \sum_{i=1}^{n} \| {\bf v}_{i}(k) - {\bf x}^{*} \|^{2} \leq \rho_{2} ^{k}  V_{0}^{''} 
\end{split}    
\end{equation}
for $ k \geq \bar{k}_{2} + H + 1 $ where $ \rho_{2}=\rho $, $ V_{0}^{''}=V_{0} $, $ H = B $ and $ \bar{k}_{2}=\bar{k} $ and $ k^{*}_{2}=k^{*} $ are as in the lemma, i.e., take $ \bar{k}_{2}=\max(k^{*}_{2}-1,B) $. Therefore, as $ k \rightarrow \infty $ we have $ \sum_{i=1}^{n} \| {\bf v}_{i}(k) - {\bf x}^{*} \|^{2} \rightarrow 0 $. That is, $  \| {\bf v}_{i}(k) - {\bf x}^{*} \| \rightarrow 0 $ for all $ i \in V $. 
Then in view of (SRDO-2) where $ {\bf x}_{i}(k+1) = {\bf v}_{i}(k) -\alpha_{k}\nabla{f}^{(\iota)}({\bf v}_{i}(k)) + {\bf R}_{i,r_{(\iota)}}(k) $ and since ${\bf R}_{i,r_{(\iota)}}(k) \rightarrow 0 $ because  $  \| {\bf v}_{i}(k) - {\bf x}^{*} \| \rightarrow 0 $, $ \alpha_{k} \rightarrow 0 $ and $ \alpha_{k} \nabla{f}^{(\iota)}({\bf v}_{i}(k))\rightarrow 0 $ since $ \alpha_{k} \rightarrow 0 $   where $ \nabla{f}^{(\iota)}({\bf v}_{i}(k)) \leq G_{f} $ since $ \nabla{f}^{(\iota)}({\bf v}_{i}(k))=  \nabla{f}^{(\iota)}({\bf v}_{i}(k))-  \nabla{f}^{(\iota)}({\bf x}^{(\iota)}) \leq L \| {\bf v}_{i}(k)- {\bf x}^{*} \| + L \| {\bf x}^{*} - {\bf x}^{(\iota)} \| $, $\| {\bf x}^{*} - {\bf x}^{(\iota)} \| =const. $ and ${\bf v}_{i}(k) \rightarrow {\bf x}^{*} $. Therefore, $ {\bf x}_{i}(k+1) = {\bf v}_{i}(k) = {\bf x}^{*} $ as $ k \rightarrow \infty $.

\end{proof}

\begin{fact}
Note that Proposition~\ref{Convergence_type1} and \ref{Convergence_type1a} are valid under Scenarios 1 or 3. However, under Scenario~2 they are valid if Condition~\ref{Condition-1} is satisfied. That is, under Scenario~2 usually at the beginning of the algorithm process or always if $ {\bf x}^{\lambda,r_{(\iota)}} = {\bf x}^{*} $ for all replica $ r_{(\iota)} $ of all partitions $ (\iota) $. Therefore, we have the following two corollaries.
\end{fact}

\begin{corollary}\label{Convergence_type1}
Let Assumptions \ref{A3.1}-\ref{A3.6} hold. Let the functions in Assumption~\ref{A3.1} be strongly convex and satisfy $ {\bf x}^{\lambda,r_{(\iota)}} = {\bf x}^{*} $ for all replica $ r_{(\iota)} $ of all partitions $ (\iota) $. Let the sequences $ \{ {\bf x}_{i}(k) \} $ and $ \{ {\bf v}_{i}(k) \} $, $ i \in V  $ be generated by Algorithm~\ref{alg:algorithm-srdo} under Scenario~2 with stepsizes and errors as given in Assumptions \ref{A3.4} and \ref{A3.6}. Assume that problem \eqref{eq1} has a non-empty optimal solution set $ \mathcal{X}^{*} $ as given in Assumption~\ref{A3.1}. Then, the sequences $ \{ {\bf x}_{i}(k) \} $ and $ \{ {\bf v}_{i}(k) \} $, $ i \in V  $ converge to the same random point in $ \mathcal{X}^{*} $ with probability 1.
\end{corollary}

\begin{proof}
Follows from Proposition~\ref{Convergence_type1} and condition $ {\bf x}^{\lambda,r_{(\iota)}} = {\bf x}^{*} $. 
\end{proof}

\begin{corollary}\label{Convergence_type1a}
Let Assumptions \ref{A3.1}-\ref{A3.6} hold. Let the functions in Assumption~\ref{A3.1}  satisfy $ {\bf x}^{\lambda,r_{(\iota)}} = {\bf x}^{*} $ for all replica $ r_{(\iota)} $ of all partitions $ (\iota) $. Let the sequences $ \{ {\bf x}_{i}(k) \} $ and $ \{ {\bf v}_{i}(k) \} $, $ i \in V  $ be generated by Algorithm~\ref{alg:algorithm-srdo} under Scenario~2 with stepsizes and errors as given in Assumptions \ref{A3.4} and \ref{A3.6}. Assume that problem \eqref{eq1} has a non-empty optimal solution set $ \mathcal{X}^{*} $ as given in Assumption~\ref{A3.1}. Then, the sequences $ \{ {\bf x}_{i}(k) \} $ and $ \{ {\bf v}_{i}(k) \} $, $ i \in V  $ converge to the same random point in $ \mathcal{X}^{*} $ with probability 1.
\end{corollary}

\begin{proof}
Follows from Proposition~\ref{Convergence_type1a} and condition $ {\bf x}^{\lambda,r_{(\iota)}} = {\bf x}^{*} $. 
\end{proof}

\subsection{Convergence Results for SRDO in any scenario}

\begin{lemma}\label{type2_error_bound}
Let Assumptions 1, 2, 4 and 5 hold. And let the functions that satisfy  $ f^{(\iota)}( {\bf x}^{*}) = f^{(\iota)}( {\bf x}^{(\iota)}) $ for any $ (\iota) $ be strongly convex. Let the sequences $ \{{\bf x}_{i}(k)\} $ and $ \{ {\bf v}_{i}(k) \} $, $ i \in V $ be generated by Algorithm~\ref{alg:algorithm-srdo}.
Then we have
{ \begin{equation}\label{ineq_type2_error_bound}
\begin{split}
&  \sum_{l=1}^{n}\|  {\bf v}_{l}(k+1) - {\bf x}^{*}\|^{2}  \leq (1-\mu) \sum_{j=1}^{n}\|{\bf v}_{j}(k)- {\bf x}^{*}\|^{2}  \\
&  + 2 \alpha_{k}\sum_{j=1}^{n}\sum_{(\iota)=1}^{p}\gamma_{(\iota)} \langle {\bf \epsilon}_{j,r_{(\iota)}}(k), {\bf v}_{j} - {\bf x}^{*} \rangle 
 + 2 (1 - \gamma_{(0)} )\alpha_{k}^{2} \sum_{j=1}^{n}\sum_{(\iota)=1}^{p}\gamma_{(\iota)} \|{\bf \epsilon}_{j,r_{(\iota)}}(k)\|^{2} \\
& + 2 \alpha_{k} \sum_{j=1}^{n}|I|\gamma_{max} L \| {\bf x}^{*} - \bar{\bf x}^{(\iota)} \| ^{2} + 2 \alpha_{k} \sum_{j=1}^{n}|I| \gamma_{max} L \| {\bf v}_{j}(k) - \bar{\bar{\bf x}}^{(\iota)} \| ^{2} \\
& - 2 \alpha_{k}  \sum_{j=1}^{n} \sum_{(\iota)=1}^{p} \gamma_{(\iota)} a_{v_{j}(k),(\iota)}(b_{v_{j}(k),(\iota)} - (1 -\gamma_{(0)}) \alpha_{k} a_{v_{j}(k),(\iota)} ) \| {\bf v}_{j}(k) - {\bf x}^{(\iota)} \| ^{2}    
\end{split}
\end{equation}}
\end{lemma}

\begin{proof}
See Appendix~D for proof. 
\end{proof}

\begin{proposition}\label{sum_v_i-x^*_type2}
Let Assumptions \ref{A3.1}-\ref{A3.6} hold. Let the functions that satisfy $ f^{(\iota)}({\bf x}^{*}) = f^{(\iota)}({\bf x}^{(\iota)}) $ be strongly convex. Let the sequences $ \{ {\bf x}_{i}(k) \} $ and $ \{ {\bf v}_{i}(k) \} $, $ i \in V  $ be generated by Algorithm~\ref{alg:algorithm-srdo} with stepsizes and errors as given in Assumptions \ref{A3.4} and \ref{A3.6}. Assume that problem \eqref{eq1} has a non-empty optimal solution set $ \mathcal{X}^{*} $ as given in Assumption~\ref{A3.1}. Then, the sequences $ \{ {\bf x}_{i}(k) \} $ and $ \{ {\bf v}_{i}(k) \} $, $ i \in V  $ converge to the same random point in $ \mathcal{X}^{*} $ with probability 1.
\end{proposition}

\begin{proof}
With errors as in Assumption~\ref{A3.4} we have $ \|{\bf R}_{i,(\iota)}(k)\|$ and $\|{\bf \epsilon}_{j,r_{(\iota)}}(k)\|$ as given in Appendix or Section~\ref{app:evaluateR} for any algorithm scenario.
Then having $ f^{(\iota)}({\bf x}^{*}) > f^{(\iota)}({\bf x}^{(\iota)}) $ for at least one $ (\iota) $ and the functions that satisfy  $ f^{(\iota)}( {\bf x}^{*}) = f^{(\iota)}( {\bf x}^{(\iota)}) $ for any $ (\iota) $ strongly convex then we have Lemma~\ref{type2_error_bound} satisfied. Then we can use  the resulting inequality \eqref{ineq_type2_error_bound} with the substitution of $ \| {\bf \epsilon}_{j,r_{(\iota)}}(k) \| $ and $ \| {\bf v}_{j}(k) -\bar{\bar{\bf x}}^{(\iota)} \|^{2}  \leq 2 \| {\bf v}_{j}(k) -{\bf x}^{*} \|^{2} + 2 \| {\bf x}^{*} -\bar{\bar{\bf x}}^{(\iota)} \|^{2} $, $ \| {\bf v}_{j}(k) -{\bf x}^{*} \|^{2} \leq max_{k-H \leq \hat{k} \leq k ; q \in V} \| {\bf v}_{q}(\hat{k}) -  {\bf x}^{*}\|^{2} $ and $ \| {\bf x}^{*} - \bar{\bar{\bf x}}^{(\iota)} \| ^{2}  \leq  \| {\bf x}^{*} - \bar{\bf x}^{(\iota)} \| = \max_{(\iota)} \| {\bf x}^{*}- {\bf x}^{(\iota)} \| ^{2} $ to get 
{ \begin{equation}\label{ineq_type2_error_bound_subst}
\begin{split}
& \sum_{l=1}^{n}\|  {\bf v}_{l}(k+1)-  {\bf x}^{*}\|^{2}  \leq (1-\mu)\sum_{j=1}^{n}\| {\bf v}_{j}(k)- {\bf x}^{*}\|^{2}  \\
& + 6 \alpha_{k} n |I| L  \max_{(\iota)} \| {\bf x}^{*}- {\bf x}^{(\iota)} \| ^{2} + 4 \alpha_{k} |I|  L \max_{k-H \leq \hat{k} \leq k ; q \in V} \sum_{j=1}^{n}\| {\bf v}_{q}(\hat{k}) -  {\bf x}^{*}\|^{2} \\
& + 4 p L \alpha_{k} \max\limits_{(\iota)} \| {\bf A}^{r_{(\iota)}} \| _{\infty} \|{\bf B}^{r_{(\iota)}}\| _{2,\infty}  \max_{k-H \leq \hat{k} \leq k ; q \in V} \sum_{j=1}^{n} \| {\bf v}_{q}(\hat{k}) -  {\bf x}^{*}\|^{2} \\ & + 8 p ( 1 - \gamma_{(0)} ) L^{2} \alpha_{k}^{2} \max_{(\iota)} \| {\bf A}^{r_{(\iota)}} \| ^{2} _{\infty} \|{\bf B}^{r_{(\iota)}}\| ^{2} _{2,\infty} \max_{k-H \leq \hat{k} \leq k ; q \in V} \sum_{j=1}^{n} \| {\bf v}_{q}(\hat{k}) -  {\bf x}^{*}\|^{2} \\
& + \alpha_{k}n (1 + 4 (1-\gamma_{(0)}) \alpha_{k} L \max\limits_{(\iota)} \| {\bf A}^{r_{(\iota)}} \| _{\infty} \|{\bf B}^{r_{(\iota)}}\| _{2,\infty}  ) \\
& \hspace{2cm} \times p L \max\limits_{(\iota)} \| {\bf A}^{r_{(\iota)}} \| _{\infty} \|{\bf B}^{r_{(\iota)}}\| _{2,\infty} \max_{(\iota)} \| {\bf x}^{*}- {\bf x}^{\lambda,r_{(\iota)}} \| ^{2} \\
& - 2 \alpha_{k} \sum_{j=1}^{n}\sum_{(\iota)=1}^{p}\gamma_{(\iota)} a_{v_{j}(k),(\iota)}(b_{v_{j}(k),(\iota)} - (1 -\gamma_{(0)}) \alpha_{k} a_{v_{j}(k),(\iota)} ) \| {\bf v}_{j}(k) - {\bf x}^{(\iota)} \| ^{2}  
\end{split}
\end{equation}}
But in order to be able to use Lemma~\ref{Lemma6a1} the last term in \eqref{ineq_type2_error_bound_subst} should be negative  so that it can be deleted from the inequality. Which means $ b_{v_{j}(k),(\iota)} - 2 (1-\gamma_{(0)}) \alpha_{k} a_{v_{j}(k),(\iota)}  \geq 0 $ where $ b_{v_{j}(k),(\iota)} = \langle \overrightarrow{u},\overrightarrow{v} \rangle $. And $ \nabla{f}^{(\iota)}({\bf v}_{j}(k)) - \nabla{f}^{(\iota)}({\bf x}^{(\iota)}) = a_{v_{j}(k),(\iota)} \| {\bf v}_{j}(k) - {\bf x}^{(\iota)} \| \overrightarrow{u} $ where $ a_{v_{j}(k),(\iota)} \leq L $.
However, $ f^{(\iota)} $ is strongly convex for every $ (\iota) $, then $ \langle \nabla{f}^{(\iota)}({\bf v}_{j}(k)) - \nabla{f}^{(\iota)}({\bf x}^{(\iota)}), {\bf v}_{j}(k) - {\bf x}^{(\iota)} \rangle  = a_{v_{j}(k),(\iota)} \| {\bf v}_{j}(k) - {\bf x}^{(\iota)} \| ^{2} \langle \overrightarrow{u},\overrightarrow{v} \rangle  \geq \sigma_{(\iota)} ^{2} \| {\bf v}_{j}(k) - {\bf x}^{(\iota)} \| ^{2} $, that is $ \langle \overrightarrow{u},\overrightarrow{v} \rangle \geq \frac{\sigma_{(\iota)}}{a_{v_{j}(k),(\iota)}} $. 
Therefore, a sufficient condition is 
\begin{equation}\label{suffcond2}
\begin{split}
 (1-\gamma_{(0)}) \alpha_{k} a_{v_{j}(k),(\iota)} \leq   (1-\gamma_{(0)}) \alpha_{k} L \leq \frac{\sigma_{(\iota)}}{a_{v_{j}(k),(\iota)}} = \langle \overrightarrow{u},\overrightarrow{v} \rangle 
\end{split}
\end{equation}
The sufficient condition in \eqref{suffcond2} is satisfied for $ k \geq k^{c}_{3} $ since $ \alpha_{k} \rightarrow 0 $. \\

Let us choose $ \alpha_{k} = \frac{1}{{k+a}^{\theta}} $ where $ a \geq 0 $ and $ \theta \in (0,1] $.
Let $ k_{3,1} $ be the iteration at which $ 0 < (1-\mu) + 4 \alpha_{k} |I|  L + 4 p L \alpha_{k} \max\limits_{(\iota)} \| {\bf A}^{r_{(\iota)}} \| _{\infty} \|{\bf B}^{r_{(\iota)}}\| _{2,\infty} $ \\ $ + 8 p (1-\gamma_{(0)}) L^{2} \alpha_{k}^{2} \max_{(\iota)} \| {\bf A}^{r_{(\iota)}} \| ^{2} _{\infty} \|{\bf B}^{r_{(\iota)}}\| ^{2} _{2,\infty} < 1 $ for the first time. And let $ k_{3,2} $ be the first iteration at which $ 4 \alpha_{k} |I|  L + 4 p L \alpha_{k} \max\limits_{(\iota)} \| {\bf A}^{r_{(\iota)}} \| _{\infty} \|{\bf B}^{r_{(\iota)}}\| _{2,\infty} + 8 p (1-\gamma_{(0)}) L^{2} \alpha_{k}^{2} \max_{(\iota)} \| {\bf A}^{r_{(\iota)}} \| ^{2} _{\infty} \|{\bf B}^{r_{(\iota)}}\| ^{2} _{2,\infty} < \frac{1-\frac{1}{l}}{(B+2)^{\theta}} + \mu - 1 $ for the first time. Then one choice is $ b_{k} = \alpha_{k} $ and an $ l \geq 1 $ such that we can find a feasible $ k_{3,2} $ depending on the value of $ \mu $.

Since $ \max_{(\iota)} \| {\bf x}^{*}- {\bf x}^{(\iota)} \| ^{2} $ and $ \max_{(\iota)} \| {\bf x}^{*}- {\bf x}^{\lambda,r_{(\iota)}} \| ^{2} $ are fixed independent of $ k $ then by deleting the last term of \eqref{ineq_type2_error_bound_subst} it will be similar to the martingale inequality \eqref{eqw1} of Lemma~\ref{Lemma6a1} for $ k \geq k^{*}_{3} $ where $ k^{+}_{3} , k_{3,1},k_{3,2}) $.

By the result of Lemma~\ref{Lemma6a1} we have for $ v_{k}= \sum_{i=1}^{n} \| {\bf v}_{i}(k) - {\bf x}^{*} \|^{2} $ that
{ \begin{equation}
\begin{split}
    \sum_{i=1}^{n} \| {\bf v}_{i}(k) - {\bf x}^{*} \|^{2} \leq \rho_{3} ^{k}  V_{0}^{'''} + b_{k} \eta_{3}
\end{split}    
\end{equation}}
for $ k \geq \bar{k}_{3} + H + 1 $ where $ \rho_{3} = \rho $, $ V_{0}^{'''} = V_{0} $, $ H = B $ and $ \bar{k}_{3}=\bar{k} $ and $ k^{*}_{3} $, i.e., take $ \bar{k}_{3}=k^{*}_{3} - 1 $ are as in the lemma and where $ \eta_{3}=\eta > 0 $ as substituted from the inequality \eqref{ineq_type2_error_bound_subst} by using the lemma. Therefore, as $ k \rightarrow \infty $ we have $ \sum_{i=1}^{n} \| {\bf v}_{i}(k) - {\bf x}^{*} \|^{2} \rightarrow 0 $ since $ b_{k} \rightarrow 0 $. That is, $  \| {\bf v}_{i}(k) - {\bf x}^{*} \| \rightarrow 0 $ for all $ i \in V $. 
Then in view of (SRDO-2) where $ {\bf x}_{i}(k+1) = {\bf v}_{i}(k) -\alpha_{k}\nabla{f}^{(\iota)}({\bf v}_{i}(k)) + {\bf R}_{i,r_{(\iota)}}(k) $ and since ${\bf R}_{i,r_{(\iota)}}(k) \rightarrow 0 $ because  $  \| {\bf v}_{i}(k) - {\bf x}^{*} \| \rightarrow 0 $, $ \alpha_{k} \rightarrow 0 $ and $ \alpha_{k} \nabla{f}^{(\iota)}({\bf v}_{i}(k))\rightarrow 0 $ since $ \alpha_{k} \rightarrow 0 $   where $ \nabla{f}^{(\iota)}({\bf v}_{i}(k)) \leq G_{f} $ since $ \nabla{f}^{(\iota)}({\bf v}_{i}(k))=  \nabla{f}^{(\iota)}({\bf v}_{i}(k))-  \nabla{f}^{(\iota)}({\bf x}^{(\iota)}) \leq L \| {\bf v}_{i}(k)- {\bf x}^{*} \| + L \| {\bf x}^{*} - {\bf x}^{(\iota)} \| $, $\| {\bf x}^{*} - {\bf x}^{(\iota)} \| =const. $ and ${\bf v}_{i}(k) \rightarrow {\bf x}^{*} $. Therefore, $ {\bf x}_{i}(k+1) = {\bf v}_{i}(k) = {\bf x}^{*} $ as $ k \rightarrow \infty $.

\end{proof}

\begin{theorem}\label{sum_v_i-x^*_type2a}
Let Assumptions \ref{A3.1}-\ref{A3.6} hold. Let the sequences $ \{ {\bf x}_{i}(k) \} $ and $ \{ {\bf v}_{i}(k) \} $, $ i \in V  $ be generated by Algorithm~\ref{alg:algorithm-srdo} with stepsizes and errors as given in Assumptions \ref{A3.4} and \ref{A3.6}. Assume that problem \eqref{eq1} has a non-empty optimal solution set $ \mathcal{X}^{*} $ as given in Assumption~\ref{A3.1}. Then, the sequences $ \{ {\bf x}_{i}(k) \} $ and $ \{ {\bf v}_{i}(k) \} $, $ i \in V  $ converge to the same random point in $ \mathcal{X}^{*} $ with probability 1.
\end{theorem}

\begin{proof}
With errors as in Assumption~\ref{A3.4} we have $ \|R_{i}(k)\|$ and $\|\epsilon_{j,r_{(\iota)}}(k)\|$ as given in Appendix Section~\ref{app:evaluateR} for any algorithm scenario.
Then having Assumptions 1-5 holding satisfied we have \eqref{ineq_general_error_bound} inequality of Lemma~\ref{general_error_bound} satisfied. Then we can use   \eqref{ineq_general_error_bound} with the substitution of $ \| {\bf \epsilon}_{j,r_{(\iota)}}(k) \| $ and \eqref{ineq_subs3}, \eqref{res2} and $ \| {\bf v}_{j}(k) - {\bf x}^{(\iota)} \|^{2}  \leq \| {\bf v}_{j}(k) -\bar{\bar{\bf x}}^{(\iota)} \|^{2}  \leq 2 \| {\bf v}_{j}(k) -{\bf x}^{*} \|^{2} + 2 \| {\bf x}^{*} -\bar{\bar{\bf x}}^{(\iota)} \|^{2} $, $ \| {\bf v}_{j}(k) -{\bf x}^{*} \|^{2} \leq max_{k-H \leq \hat{k} \leq k ; q \in V} \| {\bf v}_{q}(\hat{k}) -  {\bf x}^{*}\|^{2} $ and $ \| {\bf x}^{*} - \bar{\bar{\bf x}}^{(\iota)} \| ^{2}  \leq  \| {\bf x}^{*} - \bar{\bf x}^{(\iota)} \| = \max_{(\iota)} \| {\bf x}^{*}- {\bf x}^{(\iota)} \| ^{2} $ to get 
{ \begin{equation}\label{ineq_type2_error_bound_substa}
\begin{split}
& \sum_{l=1}^{n}\|  {\bf v}_{l}(k+1)-  {\bf x}^{*}\|^{2}  \leq (1-\mu)\sum_{j=1}^{n}\| {\bf v}_{j}(k)- {\bf x}^{*}\|^{2}  \\
& + 4 p L \alpha_{k} \max\limits_{(\iota)} \| {\bf A}^{r_{(\iota)}} \| _{\infty} \|{\bf B}^{r_{(\iota)}}\| _{2,\infty}  \max_{k-H \leq \hat{k} \leq k ; q \in V} \sum_{j=1}^{n}\| {\bf v}_{q}(\hat{k}) -  {\bf x}^{*}\|^{2} \\
& + 8 p ( 1 - \gamma_{(0)} ) L^{2} \alpha_{k}^{2} \max_{(\iota)} \| {\bf A}^{r_{(\iota)}} \| ^{2} _{\infty} \|{\bf B}^{r_{(\iota)}}\| ^{2} _{2,\infty} 
\max_{k-H \leq \hat{k} \leq k ; q \in V} \sum_{j=1}^{n} \| {\bf v}_{q}(\hat{k}) -  {\bf x}^{*}\|^{2} \\
& + \alpha_{k}( \max\limits_{(\iota)} \| {\bf A}^{r_{(\iota)}} \| _{\infty} \|{\bf B}^{r_{(\iota)}}\| _{2,\infty}  + 4 ( 1 - \gamma_{(0)}) \alpha_{k} L)
p L \sum_{j=1}^{n}\max_{k-H \leq \hat{k} \leq k ; q \in V} \| {\bf v}_{q}(\hat{k}) -  {\bf x}^{*}\|^{2} \\
& \hspace{2cm} + \alpha_{k}n (2 + 4 (1-\gamma_{(0)}) \alpha_{k} L  ) p L \max_{(\iota)} \| {\bf x}^{*}- {\bf x}^{(\iota)} \| ^{2} \\
&  + \alpha_{k}n (1 + 4 (1-\gamma_{(0)}) \alpha_{k} L \max\limits_{(\iota)} \| {\bf A}^{r_{(\iota)}} \| _{\infty} \|{\bf B}^{r_{(\iota)}}\| _{2,\infty} ) \\
& \hspace{2cm} \times p L \max\limits_{(\iota)} \| {\bf A}^{r_{(\iota)}} \| _{\infty} \|{\bf B}^{r_{(\iota)}}\| _{2,\infty} \max_{(\iota)} \| {\bf x}^{*}- {\bf x}^{\lambda,r_{(\iota)}} \| ^{2} 
\end{split}
\end{equation}}

Let us choose $ \alpha_{k} = \frac{1}{{k+a}^{\theta}} $ where $ a \geq 0 $ and $ \theta \in (0,1] $.
Let $ k_{4,1} $ be the iteration at which $ 0 < (1-\mu) + 4 p L \alpha_{k} \max\limits_{(\iota)} \| {\bf A}^{r_{(\iota)}} \| _{\infty} \|{\bf B}^{r_{(\iota)}}\| _{2,\infty} \\ $ 
\\ $ + 8 p (1-\gamma_{(0)}) L^{2} \alpha_{k}^{2}  \max_{(\iota)} \| {\bf A}^{r_{(\iota)}} \| ^{2} _{\infty} \|{\bf B}^{r_{(\iota)}}\| ^{2} _{2,\infty} + 8 (1-\gamma_{(0)}) \alpha_{k}^{2}  p \gamma_{max}L ^{2} + \alpha_{k}( 1 + 4 ( 1 - \gamma_{(0)}) \alpha_{k} L)p L < 1 $ for the first time. And let $ k_{4,2} $ be the first iteration at which $ 4 p L \alpha_{k} \max\limits_{(\iota)} \| {\bf A}^{r_{(\iota)}} \| _{\infty} \|{\bf B}^{r_{(\iota)}}\| _{2,\infty} + 8 p (1-\gamma_{(0)}) L^{2} \alpha_{k}^{2} \max_{(\iota)} \| {\bf A}^{r_{(\iota)}} \| ^{2} _{\infty} \|{\bf B}^{r_{(\iota)}}\| ^{2} _{2,\infty} $ \\
$ + 8 (1-\gamma_{(0)}) \alpha_{k}^{2}  p \gamma_{max}L ^{2} + \alpha_{k}( \max\limits_{(\iota)} \| {\bf A}^{r_{(\iota)}} \| _{\infty} \|{\bf B}^{r_{(\iota)}}\| _{2,\infty}  + 4 ( 1 - \gamma_{(0)}) \alpha_{k} L)p L < \frac{1-\frac{1}{l}}{(B+2)^{\theta}}+\mu -1 $. Then one choice is $ b_{k} = \alpha_{k} $ and an $ l \geq 1 $ so that we can find a feasible $ k_{4,2} $ depending on the value of $ \mu $.

Since $ \max_{(\iota)} \| {\bf x}^{*}- {\bf x}^{(\iota)} \| ^{2} $ and $ \max_{(\iota)} \| {\bf x}^{*}- {\bf x}^{\lambda,r_{(\iota)}} \| ^{2} $ are fixed independent of $ k $ then \eqref{ineq_type2_error_bound_substa} is similar to the martingale inequality \eqref{eqw1} of Lemma~\ref{Lemma6a1} for $ k \geq k_{4}^{*}  $ where $ k_{4}^{*}=\max(k_{4,1},k_{4,2}) $.

By the result of Lemma~\ref{Lemma6a1} we have for $ v_{k}= \sum_{i=1}^{n} \| {\bf v}_{i}(k) - {\bf x}^{*} \|^{2} $ that
{ \begin{equation}
\begin{split}
    \sum_{i=1}^{n} \| {\bf v}_{i}(k) - {\bf x}^{*} \|^{2} \leq \rho_{4} ^{k}  V_{0}^{''''} + b_{k} \eta_{4}
\end{split}    
\end{equation}}
for $ k \geq \bar{k}_{4} + H + 1 $ where $ \rho_{4}=\rho $, $ V_{0}^{''''}=V_{0} $, $ H = B $ and $ \bar{k}_{4}=\bar{k} $ and $ k_{4}^{*}=k^{*} $ are as in the lemma, i.e., take $ \bar{k}_{4}=\max(k_{4}^{*}-1,H) $ and where $ \eta_{4}=\eta > 0 $ as substituted from the inequality \eqref{ineq_type2_error_bound_substa} by using the lemma. Therefore, as $ k \rightarrow \infty $ we have $ \sum_{i=1}^{n} \| {\bf v}_{i}(k) - {\bf x}^{*} \|^{2} \rightarrow 0 $ since $  b_{k} \rightarrow 0 $. That is, $  \| {\bf v}_{i}(k) - {\bf x}^{*} \| \rightarrow 0 $ for all $ i \in V $. 
Then in view of (SRDO-2) where $ {\bf x}_{i}(k+1) = {\bf v}_{i}(k) -\alpha_{k}\nabla{f}^{(\iota)}({\bf v}_{i}(k)) + {\bf R}_{i,r_{(\iota)}}(k) $ and since ${\bf R}_{i,r_{(\iota)}}(k) \rightarrow 0 $ because  $  \| {\bf v}_{i}(k) - {\bf x}^{*} \| \rightarrow 0 $, $ \alpha_{k} \rightarrow 0 $ and $ \alpha_{k} \nabla{f}^{(\iota)}({\bf v}_{i}(k))\rightarrow 0 $ since $ \alpha_{k} \rightarrow 0 $   where $ \nabla{f}^{(\iota)}({\bf v}_{i}(k)) \leq G_{f} $ since $ \nabla{f}^{(\iota)}({\bf v}_{i}(k))=  \nabla{f}^{(\iota)}({\bf v}_{i}(k))-  \nabla{f}^{(\iota)}({\bf x}^{(\iota)}) \leq L \| {\bf v}_{i}(k)- {\bf x}^{*} \| + L \| {\bf x}^{*} - {\bf x}^{(\iota)} \| $, $\| {\bf x}^{*} - {\bf x}^{(\iota)} \| =const. $ and ${\bf v}_{i}(k) \rightarrow {\bf x}^{*} $. Therefore, $ {\bf x}_{i}(k+1) = {\bf v}_{i}(k) = {\bf x}^{*} $ as $ k \rightarrow \infty $.

\end{proof}

\begin{corollary}
If Theorem~\ref{sum_v_i-x^*_type2a} holds then Proposition \ref{Convergence_type1}, \ref{Convergence_type1a}, \ref{sum_v_i-x^*_type2} also hold.
\end{corollary}

\begin{proof}
Premises of propositions are also satisfied from premise of Theorem~\ref{sum_v_i-x^*_type2a}. 
\end{proof}



\subsection{Martingale 1}

\begin{lemma}\label{Lemma6a}
Assume the following inequality holds a.s. for all $ k \geq k^{*} $,

{ \begin{equation}\label{eqw}
    v_{k+1} \leq a_{1}v_{k}+ a_{2,k}\max_{k-B \leq \hat{k} \leq k} v_{\hat{k}}
\end{equation}}
$ v_{k}$, $ a_{1} $ and $ a_{2,k} $ are non-negative random variables where $ a_{1} + a_{2,k} \leq 1 $ and $ \{ a_{2,k} \} $ is a decreasing sequence. Then if for $ \rho = (a_{1}+a_{2,\bar{k}})^{\frac{1}{B+1}} $ where $ \bar{k} \geq k^{*} - 1 $ and $ \bar{k} \geq B $ (i.e., we index from $ k =0 $) we have
{ \begin{equation}
     v_{\bar{k}+n} \leq \rho^{\bar{k}+B+1}V_{0} \ \ \ a.s.
 \end{equation}}
for $ 1 \leq n \leq B + 1$
and
{ \begin{equation}\label{rec1}
     v_{k} \leq \rho^{k}V_{0} \ \ \ a.s.
 \end{equation}}
for all $ k \geq \bar{k}+B+1 $ where $ V_{0} > 0 $ as in proof and $ \rho $ as before.
\end{lemma}

\begin{proof}
See Appendix~E for proof. 
\end{proof}

\subsection{Martingale 2}

\begin{lemma}\label{Lemma6a1}
Assume the following inequality holds a.s. for all $ k \geq k^{*} $,

{ \begin{equation}\label{eqw1}
    v_{k+1} \leq a_{1}v_{k}+ a_{2,k}\max_{k-B \leq \hat{k} \leq k} v_{\hat{k}}+a_{3,k}
\end{equation}}
$ v_{k}$, $ a_{1} $, $ a_{2,k} $  and $ a_{3} $ are non-negative random variables where $ a_{1} + a_{2,k} \leq 1 $ and $ \{ a_{2,k} \} $ is a decreasing sequence. Then if for $ \rho = (a_{1}+a_{2,\bar{k}})^{\frac{1}{B+1}} $ where $ \bar{k} \geq k^{*} - 1 $ and $ \bar{k} \geq B $ (i.e., we index from $ k =0 $), $ a_{1} \leq 1 -\mu $, $ a_{2,\bar{k}} \leq \frac{1-\frac{1}{l}}{(B+2)^{\theta}} + \mu - 1 $  and $ a_{3,k} \leq \frac{b_{k+1}}{l}\eta$ where $ b_{k}=\frac{1}{(k+a)^{\theta}} $, $ l \geq 1 $, $ \theta \in (0,1] $ and $ \eta =\frac{a_{3}}{1-a_{1}-a_{2,\bar{k}}} $ we have
{ \begin{equation}
     v_{\bar{k}+n} \leq \rho^{\bar{k}+B+1}V_{0} + b_{\bar{k}+n} \eta \ \ \ a.s.
 \end{equation}}
for $ 1 \leq n \leq B + 1$
and
{ \begin{equation}\label{rec11}
     v_{k} \leq \rho^{k}V_{0} + b_{k} \eta \ \ \ a.s.
 \end{equation}}
for all $ k \geq \bar{k}+B+1 $ where $ V_{0} > 0 $ as in proof and $ \rho $ and $ \eta $ as before.
\end{lemma}

\begin{proof}
See Appendix~F for proof. 
\end{proof}

\section{Convergence Rate}\label{sec-convergence_rate}

In this subsection are going to find the expected convergence rate of SRDO under any scenario where the function $ f $ to be minimized is a strongly convex formed of $ p $ functions $ f^{(\iota)} $ that are strongly convex.

Then by elaborating upon Lemma~\ref{general_error_bound} for the case above we have 

{ \begin{equation}
\begin{split}
& \sum_{l=1}^{n}\|  {\bf x}_{l}(k+1)-  {\bf x}^{*}\|^{2}  \leq (1-\mu)\sum_{j=1}^{n}\| {\bf x}_{j}(k)- {\bf x}^{*}\|^{2}  \\
& + \frac{1}{1-\mu} 6 \alpha_{k} n |I| \gamma_{max} L  \max_{(\iota)} \| {\bf x}^{*}- {\bf x}^{(\iota)} \| ^{2}  \\
& + \frac{1}{1-\mu} 4 \alpha_{k}|I| \sum_{j=1}^{n}p \gamma_{max} L \max_{k-H \leq \hat{k} \leq k ; q \in V} \| {\bf v}_{q}(\hat{k}) -  {\bf x}^{*}\|^{2} \\
& + \frac{1}{1-\mu} 4 p L \alpha_{k}\sum_{j=1}^{n}\| {\bf A}^{r_{(\iota)}} \| _{\infty} \|{\bf B}^{r_{(\iota)}}\| _{2,\infty}  \max_{k-H \leq \hat{k} \leq k ; q \in V} \| {\bf v}_{q}(\hat{k}) -  {\bf x}^{*}\|^{2} \\
& + \frac{1}{1-\mu} 8 p ( 1 - \gamma_{(0)} ) L^{2} \alpha_{k}^{2} \sum_{j=1}^{n}\| {\bf A}^{r_{(\iota)}} \| ^{2} _{\infty} \|{\bf B}^{r_{(\iota)}}\| ^{2} _{2,\infty} \max_{k-H \leq \hat{k} \leq k ; q \in V} \| {\bf v}_{q}(\hat{k}) -  {\bf x}^{*}\|^{2} \\
& + \frac{1}{1-\mu} \alpha_{k}n (1 + 4 (1-\gamma_{(0)}) \alpha_{k} L \| {\bf A}^{r_{(\iota)}} \| _{\infty} \|{\bf B}^{r_{(\iota)}}\| _{2,\infty}  ) \\
& \hspace{2cm} \times p L \| {\bf A}^{r_{(\iota)}} \| _{\infty} \|{\bf B}^{r_{(\iota)}}\| _{2,\infty} \max_{(\iota)} \| {\bf x}^{*}- {\bf x}^{\lambda,r_{(\iota)}} \| ^{2} 
\end{split}
\end{equation}}
since $ \sum_{l=1}^{n}\|  {\bf v}_{l}(k)-  {\bf x}^{*}\|^{2} \leq (1-\mu) \sum_{l=1}^{n}\|  {\bf x}_{l}(k)-  {\bf x}^{*}\|^{2} $.

\begin{lemma}\label{convergence-lemma}
Let $ \{d_{k} \} $ and $ \{ u_{k} \} $ be scalar sequences such that $ d_{k} \leq c d_{k-1} + u_{k-1} $ for all $ k \geq 1$ and some scalar $ c \in (0,1) $. Then, $ \lim \sup_{k \rightarrow \infty} d_{k} \leq \frac{1}{1-c} \lim \sup _{ k \rightarrow \infty} u_{k} $. 
\end{lemma}

But using Proposition~\ref{sum_v_i-x^*_type2} we have inequality \eqref{ineq_type2_error_bound_subst} for $ k \geq k_{3}^{*} $ in the form \eqref{eqw1} is valid. Thus, using Lemma~\ref{Lemma6a1} we get $ \sum_{i=1}^{n} \| {\bf v}_{i}(k) - {\bf x}^{*} \|^{2} \leq \rho_{3}^{k}V_{0}^{'''} + b_{k} \eta_{3} $ where $ b_{k}= \alpha_{k} $, $ \rho_{3} $ and $ \eta_{3} $ (see Appendix~G) for $ k \geq \bar{k}_{3} + H + 1 $. 

Using Lemma~\ref{convergence-lemma} we get
{ \begin{equation}
\begin{split}
& \sum_{j=1}^{n}\| {\bf x}_{j}(k)- {\bf x}^{*}\|^{2}  \leq
  \frac{6n|I|p\gamma_{max}L}{\mu(1-\mu)} \alpha_{k} \max_{(\iota)} \| {\bf x}^{*}- {\bf x}^{(\iota)} \| ^{2} \\
& + \frac{1}{\mu(1-\mu)} \alpha_{k}n (1 + 4 (1-\gamma_{(0)}) \alpha_{k} L \| {\bf A}^{r_{(\iota)}} \| _{\infty} \|{\bf B}^{r_{(\iota)}}\| _{2,\infty}  ) \\
& \hspace{2cm} \times p L \| {\bf A}^{r_{(\iota)}} \| _{\infty} \|{\bf B}^{r_{(\iota)}}\| _{2,\infty} \max_{(\iota)} \| {\bf x}^{*}- {\bf x}^{\lambda,r_{(\iota)}} \| ^{2} \\
& + \frac{\alpha_{k}} {\mu(1-\mu)} ( 4 p L \alpha_{\bar{k}_{3}} \| {\bf A}^{r_{(\iota)}} \| _{\infty} \|{\bf B}^{r_{(\iota)}}\| _{2,\infty} + 8 p (1-\gamma_{(0)}) L^{2} \alpha_{\bar{k}_{3}}^{2}  \| {\bf A}^{r_{(\iota)}} \| ^{2} _{\infty} \|{\bf B}^{r_{(\iota)}}\| ^{2} _{2,\infty} \\
& + 8 (1-\gamma_{(0)}) \alpha_{\bar{k}_{3}}^{2}  p \gamma_{max}L ^{2} + \alpha_{k}( 1 + 4 ( 1 - \gamma_{(0)}) \alpha_{k} L)p L) \eta_{3}
\end{split}
\end{equation}}
\subsection{Convergence Rate for Strongly Convex Function with $ f^{(\iota)}({\bf x}^{*})=f^{(\iota)}({\bf x}^{(\iota)}) $ for all $ (\iota) $ under scenarios of Division~1}

The convergence rate of SRDO in minimizing strongly convex function $ f $ formed of $ p $ strongly convex functions  with $ f^{(\iota)}({\bf x}^{*})=f^{(\iota)}({\bf x}^{i}) $ for all $ (\iota) $ under scenarios of Division~1 can be deduced by applying Proposition~\ref{Convergence_type1} and is 
{\small \begin{equation}\label{convergence_rate_strongly_convex_Martingale_1}
\begin{split}
 \mathbb{E}  [ \sum_{i=1}^{n} & \| {\bf v}_{i}  (k)- {\bf x}^{*} \| ^{2}  ] \leq \\
 & ( 1 - \mu + 4 (1-\gamma_{(0)}) L \alpha_{\bar{k}_{1}} \| {\bf A}^{r_{(\iota)}} \| _{\infty} \|{\bf B}^{r_{(\iota)}}\| _{2,\infty}  ( 1 + 2 L \alpha_{\bar{k}_{1}}  \| {\bf A}^{r_{(\iota)}} \|  _{\infty} \|{\bf B}^{r_{(\iota)}}\|  _{2,\infty} )^{\frac{k}{H+1}}V_{0}  
\end{split}
\end{equation}}
for $ k \geq \bar{k}_{1} + H + 1 $.

\section{Numerical Simulation}

Our aim in the simulation is to verify the convergence of the proposed algorithm while showing its convergence rate for different algorithm's scenarios.

In this section, we restrict the optimization problem to the following unconstrained convex optimization problem on a parameter server network  
\begin{equation}\label{objective.simulation}
 \arg\min_{x \in \mathbf{R}^{N}} \|  {\bf G}{\bf x}- {\bf y} \|_{2}^{2},\end{equation}
where the network contains $ n $ server nodes, $ G $ is a random matrix of size $M\times N$ whose entries are independent and identically distributed standard normal random variables, and
\begin{equation} y= {\bf G} {\bf x}_o\in \mathbb{R}^{M}\end{equation}
has entries of ${\bf x}_o$ that are identically independent random variables sampled from the uniform bounded random distribution  between $ -1 $ and $ 1 $.
The solution ${\bf x}^{*}$ of the optimization problem above  is the least squares solution of the overdetermined system  
$ y  = {\bf G} {\bf x}_o, \ x_o\in \mathbb{R}^N$. 
We demonstrate the  performance of SRDO to solve the convex optimization problem \eqref{eq1}, and match it with the calculated convergence rates.

Accordingly, the random measurement matrix $ {\bf G} $,  the measurement data $ y$,  and the objective function  $ f({\bf x}):= \|  {\bf G} {\bf x}- {\bf y}  \|_{2}^{2} $ in \eqref{eq1} as follows:
 $f({\bf x})= \sum_{i=1}^{p} f^{(\iota)}({\bf x}):= \sum_{i=1}^{p} \| {\bf G}_i {\bf x} - {\bf y}_i\|_2^2.  $
 
In the simulations, without a loss of generality, we assume the number of workers' partitions equals the number of server nodes; i.e., $ p = n $. We also require that we have one replica per partition and that the repartitioned parts have the same size; i.e., the number of rows in $ {\bf G}_{i} $ and the lengths of vectors $ {\bf y}_i, 1\le i\le n$ are the same, respectively and obviously equal. Then we can partition the network
around $ m $ worker nodes, where the number of workers per  replica of partition $ (\iota) $, $ n_{r_{(\iota)}} = n_{(\iota)} = \emph{const} $ and $ s_{r_{(\iota)}}=s_{(\iota)} = \emph{const} $, the maximum number of allowed stragglers per replica of partition $ (\iota) $.
We further require that $ \gamma_{(0)} $ is small (i.e., $ \gamma_{(0)}= 0.05 $) and $=\gamma_{(\iota)} = \frac{1}{p} (1-\gamma_{(0)}) $, that is all partitions are connected to a server at the pull step with the same probability and the probability of disconnection of a server from all partitions at each pull step is relatively small.

For each simulation, we ran the experiment 100 times and average the results. Thus, we present the simulation for $ 100 $ samples of parameter server networks of $ p =5 $ equal sized partitions $ (\iota) $, with unanimous $ n_{r_{(\iota)}} = n_{(\iota)} = 3 $, $ s_{r_{(\iota)}} = s_{(\iota)} = 1 $ and $ n_{(\iota)} = 5 $, $ s_{(\iota)} = 2 $ for $ 1 \leq (\iota) \leq p $, respectively (i.e., $ m_{(\iota)} = 300$ and $ 500 $ where $ 1 \leq (\iota) \leq p $, and $ M= 1500 $ and $ 2500 $ and $ \bar{m} = 100 $, $ N =100 $, respectively). Here, $ \bar{m} $ stands for the number of rows in a partition sub-partition, which is assumed equal all over the network and $ m_{(\iota)} $ is the number of rows used by partition $ (\iota) $. That is, $ \bar{m} $ corresponds to the functions $ f^{r_{(\iota)}}_{\lambda} $, where $ \lambda $ corresponds to repartition $ 1 \leq \lambda \leq n_{r_{(\iota)}} $ in an arbitrary replica $ r_{(\iota)} $ of partition $ (\iota) $. Each worker node finds its local coded gradient through a combination of uncoded local gradients computed through local optimization problems of  overdetermined linear systems of equations. That is, for partition $ (\iota) $ where $ n_{(\iota)}=3 $, each worker takes a total of rows which is at least $ \bar{m} $ to calculate its coded gradient according to the used coding scheme in \citep{tandon17a}.

The stepsizes $\alpha_{k}$ are chosen such that
$ \alpha_{k} = \frac{1}{(k+a)^{\theta}}$ where $ a \geq 0 $ and $ \theta \in (0,1] $. 

We define the absolute error
${\rm AE}:= \max_{1\le i\le n} \frac{\|{\bf x}_i(k)-{\bf x}_o\|_2}{\|{\bf x}_0\|_2}$ 
and  consensus error
$ {\rm CE}:= \max_{1\le i\le n} \frac{\|{\bf x}_i(k)-\bar{\bf x}(k)\|_2}{\|{\bf x}_o\|_2} $
which are used to measure  the performance of SRDO. \\

Moreover, in this simulation we have used a fixed coding scheme in each experiment. We could have adapted other coding schemes that can be adjusted to improve the convergence rate as the performance of SRDO is dependent on the used coding scheme. To that end, we can effectively improve the performance of our algorithm by adjusting the coding scheme in a manner dependent on the factors that govern this performance such as independence and probability of stragglers, partitions' connections' probabilities, full disconnection probability, prioritization of stale gradients, the delay uniform bound $ H $ of the allowed delayed coded, respectively uncoded gradients used in the gradient computation schemes. If done efficiently, by the use of a learning algorithm for example we can aim to outperform the Centralized-SGD.


For our chosen coding scheme, we compare SRDO in gradient computation scenario~1 to Centralized–SGD with full connection in Fig~\ref{Fig.3}. We see that both have almost the same performance, this is due to the fact that they both compute the full batch inexact gradient although SRDO is mitigating the effect of an allowed number of stragglers. That is, Centralized-SGD has an equivalent estimate of the gradient at the expense of a higher communication cost.

\vspace{-1cm}
\begin{figure}[H]
\includegraphics[bb=0 0 800 800,scale=0.3]{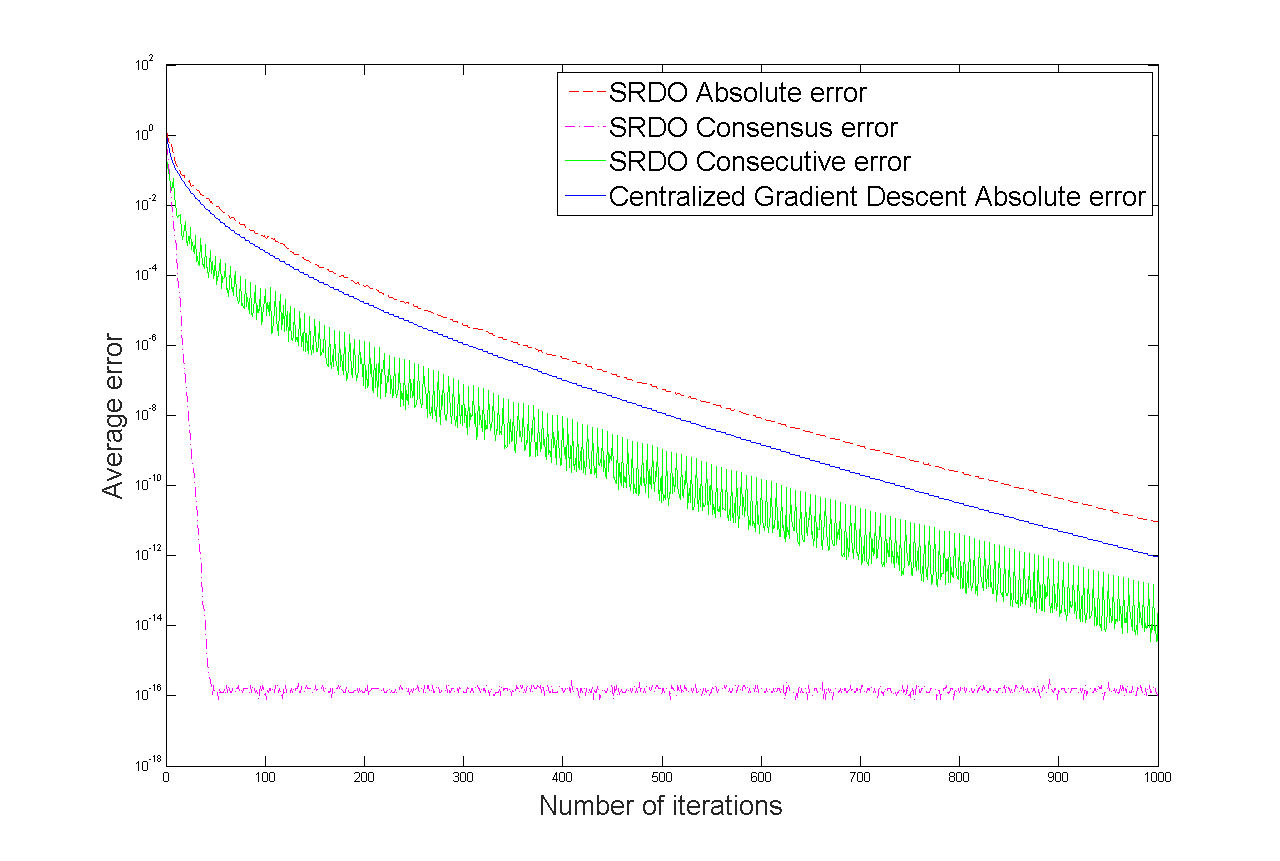}
\caption{Allowed number of stragglers connection for SRDO algorithm network ($ n_{(\iota)}=5 $, $ s_{(\iota)}=2 $) with $ \frac{1}{(k+300)^{0.55}} $ and Centralized-SGD with full connection and no stragglers ($ n_{(\iota)}=5 $, $ s_{(\iota)}=0 $), for $ M= 2500, N =100 $.}
\label{Fig.3}
\end{figure}
A lower bound on the convergence rate performance of SRDO is in scenario~2 which converges slower than Centralized-SGD (with no failures) as shown in Fig.~\ref{Fig.4}. 


We apprehend that if $ H > 0 $ then the behavior of SRDO in scenario~2 would be much worse than SRDO scenario~2 with $ H = 0 $ since the partial inexact gradient computed in the first is of delayed evaluations.

Meanwhile, SRDO in scenario~3 would perform better than SRDO in scenario~2 since the stale delayed gradients are added in an attempt to allow the servers to form an overall inexact gradient and thus its performance might match that of Centralized-SGD with no failures depending on the used coding scheme and the adequate stepsize calibration as we are going to discuss later in Subsection~\ref{Subsection-theta-discussion}. 
Although SRDO in Scenario~3 has better convergence rate than Centralized-SGD with the same type of failures depending on the value of the delay uniform bound $ H $, where smaller $ H $ favors a better performance, and this is due to the leverage allowed by the utilized coding scheme.


The above analysis is clearly seen in Fig.~\ref{Fig.5}, \ref{Fig.6} and \ref{Fig.7} where an increase in $ H $ degrades the convergence rate and it is up to the stepsize calibration, as we are going to show later in Subsection~\ref{Subsection-theta-discussion}, to provide suitable performance accomodation


\vspace{-2cm}
\begin{figure}[H]
\includegraphics[bb=0 0 800 800,scale=0.3]{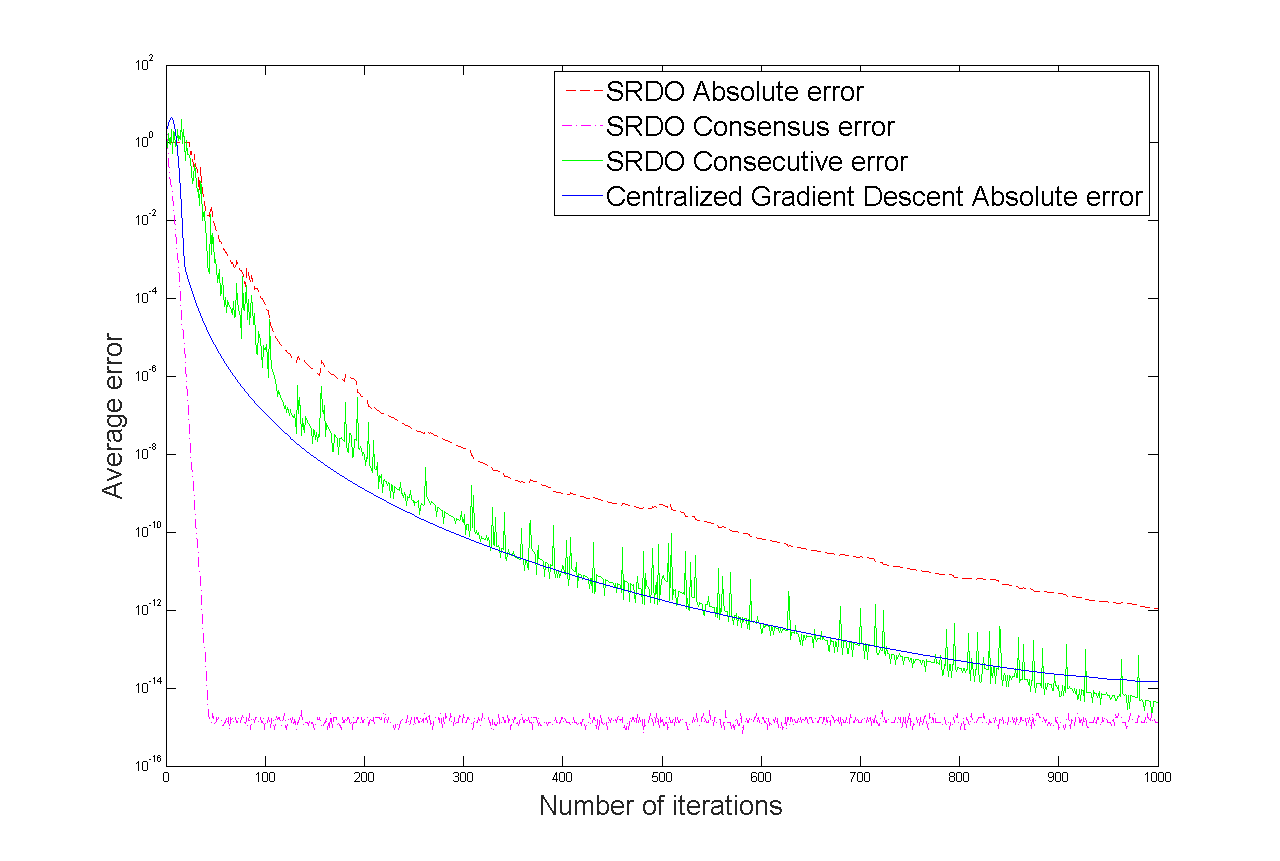}
\caption{More than the allowed number of stragglers connection for SRDO algorithm network ($ n_{(\iota)}=5 $, $ s_{(\iota)}=3 $) where $ H=20 $ for $ \alpha_{k}= \frac{1}{(k+300)^{0.95}} $ using gradient computation scenario $ 2 $ and Centralized-SGD with full connection and no stragglers ($ n_{(\iota)}=5 $, $ s_{(\iota)}=0 $), for $ M= 1500, N =100 $.}
\label{Fig.4}
\end{figure}

\vspace{-3cm}
\begin{figure}[H]
\includegraphics[bb=0 0 800 800,scale=0.3]{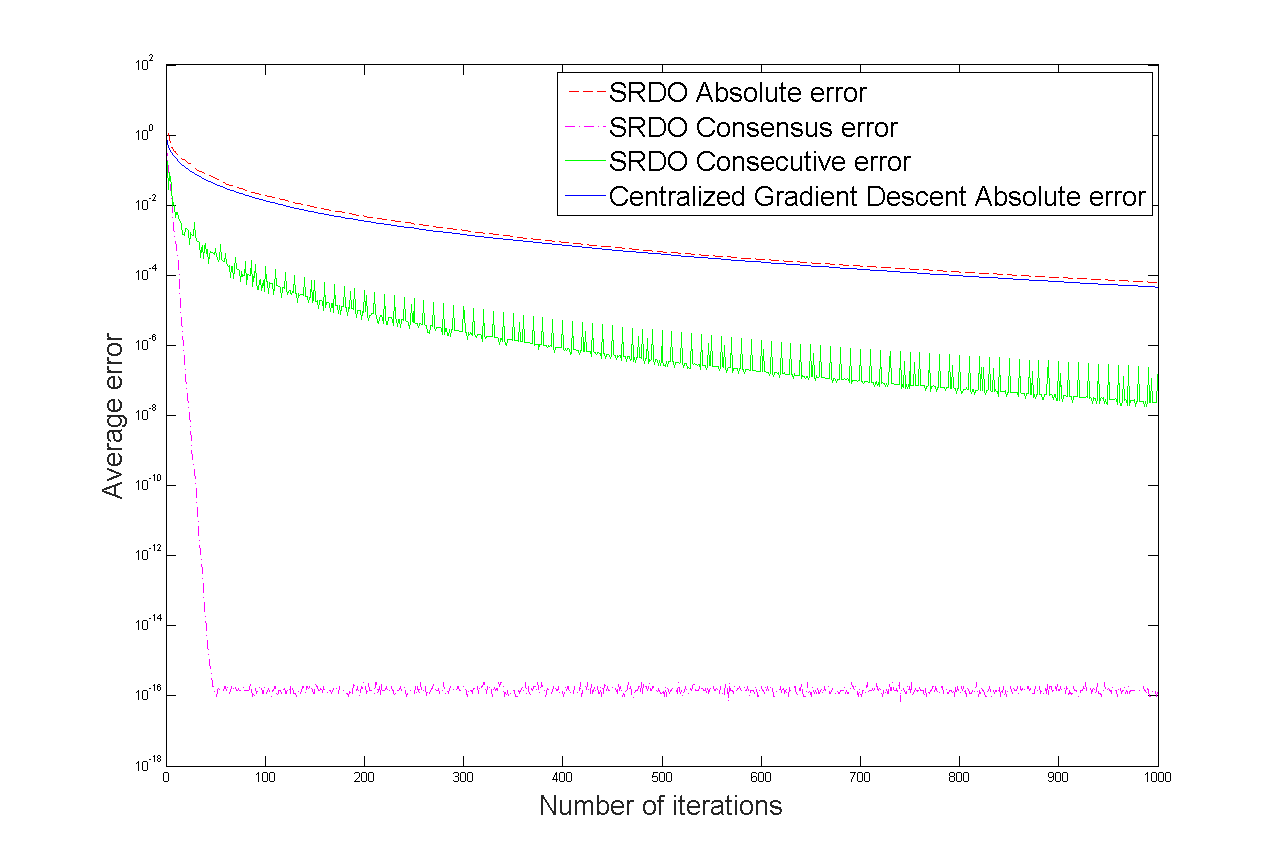}
\caption{More than the allowed number of stragglers connection for SRDO algorithm network ($ n_{(\iota)}=3 $, $ s_{(\iota)}=2 $) where $ H=5 $ for $ \alpha_{k}= \frac{1}{(k+300)^{0.35}} $ using gradient computation scenario $ 3 $ and Centralized-SGD with same type of failures ($ n_{(\iota)}=3 $, $ s_{(\iota)}=2 $), for $ M= 1500, N =100 $.}
\label{Fig.5}
\end{figure}

\vspace{-1cm}
\begin{figure}[H]
\includegraphics[bb=0 0 800 800,scale=0.3]{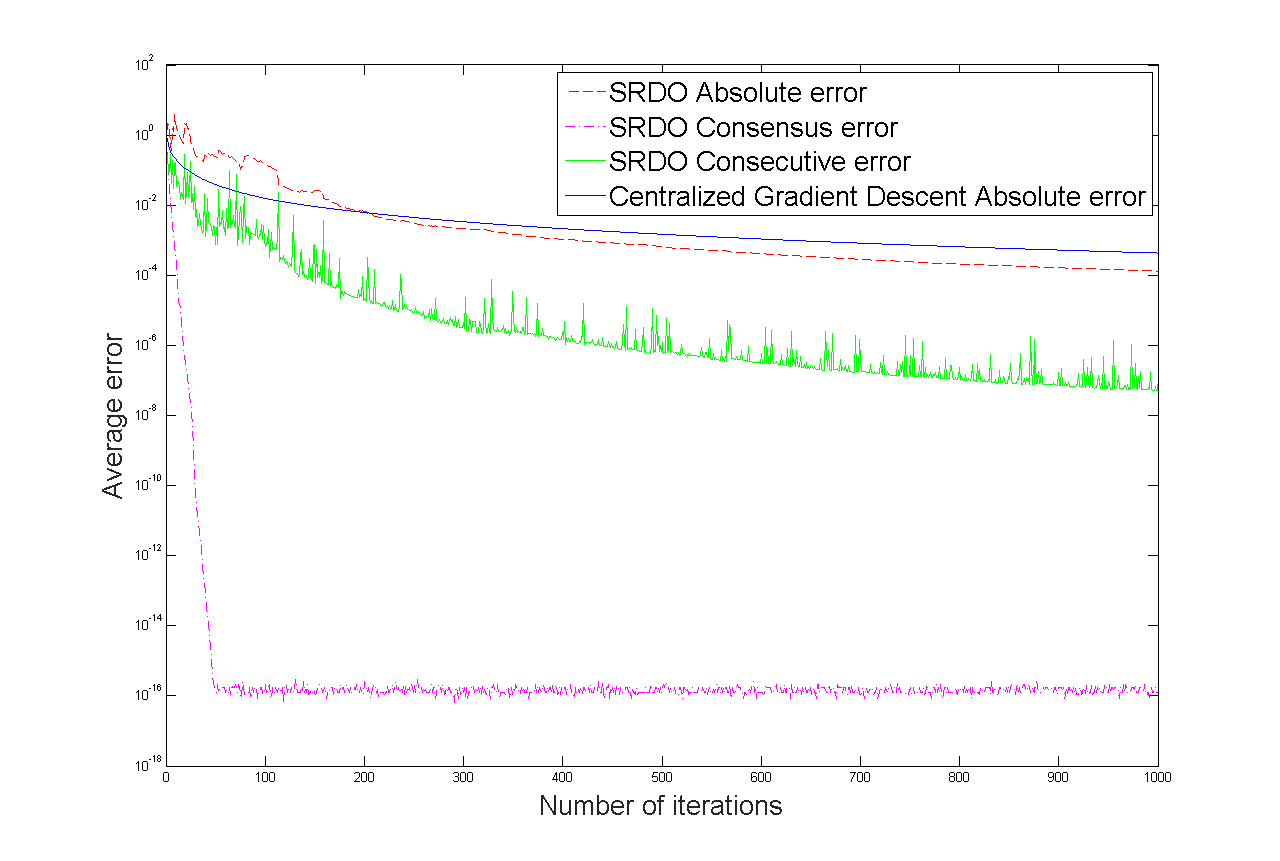}
\caption{More than the allowed number of stragglers connection for SRDO algorithm network ($ n_{(\iota)}=3 $, $ s_{(\iota)}=2 $) where $ H=10 $ for $ \alpha_{k}= \frac{1}{(k+300)^{0.55}} $ using gradient computation scenario $ 3 $  and Centralized-SGD with same type of failures ($ n_{(\iota)}=3 $, $ s_{(\iota)}=2 $), for $ M= 1500, N =100 $.}
\label{Fig.6}
\end{figure}

\vspace{-2cm}
\begin{figure}[H]
\includegraphics[bb=0 0 800 800,scale=0.3]{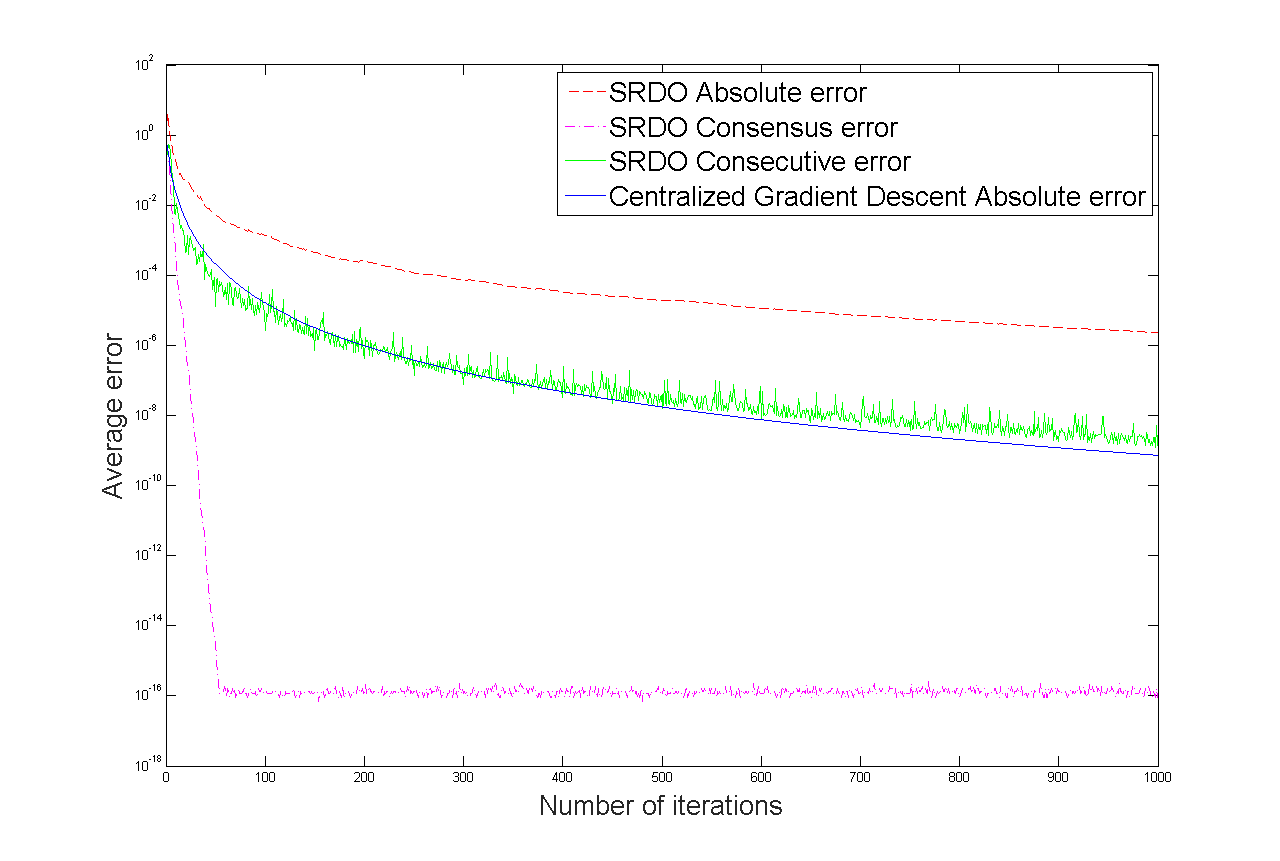}
\caption{More than the allowed number of stragglers connection for SRDO algorithm network ($ n_{(\iota)}=3 $, $ s_{(\iota)}=2 $) where $ H=20 $ for $ \alpha_{k}= \frac{1}{(k+300)^{0.75}} $ using gradient computation scenario $ 3 $  and Centralized-SGD with same type of failures ($ n_{(\iota)}=3 $, $ s_{(\iota)}=2 $), for $ M= 1500, N =100 $.}
\label{Fig.7}
\end{figure}


\subsection{Discussion relative to stepsize and delay uniform bound}\label{Subsection-theta-discussion}

The convergence can be faster or slower depending on the condition number of the coded matrices in relation to the uncoded matrices at each node; i.e., this is related to the respective Lipschitz constants.
The fluctuation of the average consensus  error for the SRDO has a larger variation, 
while the proposed algorithm corresponding absolute error behaves more smoothly.
We can conceive from the simulations that SRDO has faster convergence for a smaller exponent $\theta \in (0 , 1] $, which confirms its convergence rate estimate in Section~\ref{sec-convergence_rate}.
However, the simulations also indicate that decreasing more the exponent $\theta$ moves SRDO into the divergence phase, which could be directly related to the complexity of the network.

It is  worth noting that we can adequately calibrate this convergence/divergence trade-off by increasing the value of $ a $ in our illustrated examples for a fixed exponent $ \theta $.
We can see that for a fixed value of $ H $, the delay uniform bound, and a fixed condition number, i.e., fixed Lipschitz constant,(more specifically for a fixed matrix $ {\bf G}_{i} $), a considerable decrease in the exponent $ \theta $ (i.e., the increase in the stepsize) allows the algorithm to enter the divergence instability region. Then an increase in $ \theta $ (i.e., a decrease in the stepsize) will make it converge the fastest where then any other increase in $ \theta $ (i.e., other decrease in the stepsize) will ultimately degrade the algorithm to a slower convergence.
Similarly, if we fix $ H $ and the stepsize, the behavior of the convergence of the algorithm relative to the change in the condition number is the same as that relative to the stepsize in the previous scenario.
Moreover, we see that when $ H $ increases, and the allowed number of stragglers of the same instant connection becomes less frequent, then the convergence under the same stepsize and fixed coding matrices is replaced by an anticipated divergence. Then, for that $ H $, we can reenter the convergence region of the algorithm by increasing the exponent $ \theta $ for a fixed optimization problem.  Convergence is also achieved for problems with matrices of higher condition numbers when the stepsize and $ H $ are fixed.


For scenario~3 of SRDO, we can also see in Fig. \ref{Fig.5} that the value of $ \theta = 0.35 $ allowed a comparable convergence rate of the SRDO for $ H = 5 $ as that of the Centralized-SGD algorithm with same type of failures.
In Figs. \ref{Fig.6} and \ref{Fig.7}, we realize that the lower value of $ \theta=0.35 $ is not permissible, because the SRDO algorithm will considerably enter the instability region, while a higher value of $ \theta=0.55 $ favors a better convergence rate for $ H = 10 $, and the highest value of $ \theta=0.75 $ a better convergence rate for $ H = 20 $. And this better convergence rate is relative to Centralized-SGD with the same type of failures although the overall performance here of the first is much degraded relative to the latter. We could have also acquired a better convergence rate for $ H=10 $ and $ H=20 $ relative to Centralized-SGD with same type of failures or even no failures if we adequately calibrated $ \theta $ (i.e., increase $ \theta $) so that the algorithm is the fast convergence region as described at the beginning of this subsection.
Moreover, as we mentioned earlier, in the simulation we have used a definite coding scheme introduced in \citep{tandon17a} using the encoding algorithm~\ref{alg:algorithm-B} and decoding algorithm~\ref{alg:algorithm-A} described in Subsection~\ref{Gradient Coding}. We could have adapted other coding schemes that can be adjusted to improve the convergence rate.



\section{Conclusion}

We have considered in this paper a parameter server network algorithm, SRDO, for minimizing a convex function that consists of a number of component functions. The parameter server updates estimates synchronously with the possibility of asynchronous use of computed gradients' evaluations and straggler workers mitigation. Computed gradients can be of a delayed time with uniform bound on that delay and no other statistical assumptions. We restricted the simulation for the case of a quadratic function which corresponds to solving an overdetermined system of linear equations. A convergence proof for this algorithm in its general form (not necessarily a quadratic function) was provided in the case of network topologies where the number of stragglers is under the allowed threshold and when the number of stragglers exceeding the quantity allowed. In Section~\ref{sec-convergence_rate} we describe the convergence rates. Furthermore, the simulation showed optimistic results for the algorithm convergence rate. The metrics matched the centralized gradient descent method metrics with the bonus of robustness to an allowed number of stragglers. Furthermore, we analytically showed that the convergence rate can be considerably enhanced through applying an adequate coding scheme as shown from the dependency of the convergence rate on the coding matrices.

\vspace{1cm}

\section{Appendix}

\subsection{Evaluation of $ \sum_{j=1}^{n} \| {\bf R}_{j,r_{(\iota)}}(k) \| ^{2} $}\label{app:evaluateR}

Then by Cauchy-Schwartz inequality we have \\
$ \| {\bf R}_{j,r_{(\iota)}}(k) \|  \leq \alpha_{k}  \|{\bf A}^{r_{(\iota)}} \|  _{\infty} \|{\bf B}^{r_{(\iota)}}\| _{2,\infty} $ \\ $ \times  \max\limits_{(\iota), \lambda, 0 \leq \Delta{k}(q,w,r_{(\iota)},k) \leq H} \| \nabla{f}^{r_{(\iota)}}_{\lambda}({\bf v}_{q}(k-\Delta{k}(q,w,r_{(\iota)},k)) - \nabla{f}^{r_{(\iota)}}_{\lambda}({\bf v}_{i}(k)) \|  $
and 
\\
$ \| {\bf R}_{j,r_{(\iota)}}(k) \| ^{2} \leq \alpha_{k} ^{2} \|{\bf A}^{r_{(\iota)}} \| ^{2} _{\infty} \|{\bf B}^{r_{(\iota)}}\| _{2,\infty} ^{2} $ \\ $ \times \max\limits_{(\iota), \lambda, 0 \leq \Delta{k}(q,w,r_{(\iota)},k) \leq H} \| \nabla{f}^{r_{(\iota)}}_{\lambda}({\bf v}_{q}(k-\Delta{k}(q,w,r_{(\iota)},k)) - \nabla{f}^{r_{(\iota)}}_{\lambda}({\bf v}_{i}(k)) \| ^{2} $



For scenario~1  and 3 of Division~1, we have the contribution relative to each $ w \in \Gamma_{fit, r_{(\iota)}} $ of $ (\nabla{f}^{r_{(\iota)}}_{\lambda}({\bf v}_{q}(k-\Delta{k}(q,w,r_{(\iota)},k)) - \nabla{f}^{r_{(\iota)}}_{\lambda}({\bf v}_{i}(k))) $ satisfying the inequality 
{\begin{align*}
  &  \| \nabla{f}^{r_{(\iota)}}_{\lambda}({\bf v}_{q}(k-\Delta{k}(q,w,r_{(\iota)},k)) - \nabla{f}^{r_{(\iota)}}_{\lambda}({\bf v}_{i}(k)) \|   \leq \\
  & \| \nabla{f}^{r_{(\iota)}}_{\lambda}({\bf v}_{q}(k-\Delta{k}(q,w,r_{(\iota)},k)) - \nabla{f}^{r_{(\iota)}}_{\lambda}(x^{*}) \| 
  + \| \nabla{f}^{r_{(\iota)}}_{\lambda}({\bf v}_{i}(k))-\nabla{f}^{r_{(\iota)}}_{\lambda}(x^{*}) \| \leq \\
& L\| {\bf v}_{q}(k-\Delta{k}(q,w,r_{(\iota)},k) - x^{*} \| 
+ L \| {\bf v}_{i}(k) - x^{*} \|.
\end{align*}}
where we used the Lipschitz assumption   on the gradients for the last inequality. 
Thus, we have
{ \begin{align*}
  & \max\limits_{\lambda, 0 \leq \Delta{k}(q,w,r_{(\iota)},k) \leq H, q \in V} \| \nabla{f}^{r_{(\iota)}}_{\lambda}({\bf v}_{q}(k-\Delta{k}(q,w,r_{(\iota)},k)) - \nabla{f}^{r_{(\iota)}}_{\lambda}({\bf v}_{i}(k)) \|   \leq \\
  &  \max\limits_{\lambda, 0 \leq \Delta{k}(q,w,r_{(\iota)},k) \leq H, q \in V} (\| \nabla{f}^{r_{(\iota)}}_{\lambda}({\bf v}_{q}(k-\Delta{k}(q,w,r_{(\iota)},k)) - \nabla{f}^{r_{(\iota)}}_{\lambda}(x^{*}) \| \\
 & \hspace{2cm} + \| \nabla{f}^{r_{(\iota)}}_{\lambda}({\bf v}_{i}(k))-\nabla{f}^{r_{(\iota)}}_{\lambda}(x^{*}) \|) \leq \\
 & \max\limits_{(\iota), \lambda, 0 \leq \Delta{k}(q,w,r_{(\iota)},k) \leq H, q \in V} (L\| {\bf v}_{q}(k-\Delta{k}(q,w,r_{(\iota)},k) - x^{*} \| 
 + L \| {\bf v}_{i}(k) - x^{*} \|) .
\end{align*}}

But we have

{ \begin{equation}\label{3.64a}
\begin{split}
 & \max\limits_{\lambda, 0 \leq \Delta{k}(q,w,r_{(\iota)},k) \leq H, q} \| {\bf v}_{q}(k-\Delta{k}(q,w,r_{(\iota)},k) - x^{*} \|   \leq
 \max\limits_{k-H \leq \hat{k} \leq k, q} \| {\bf v}_{q}(\hat{k}) -  x^{*} \|
\end{split}
\end{equation}}
where $ q \in \{1,\ldots,q\} $.

{\small \begin{equation}\label{3.64b}
\begin{split}
\max\limits_{\lambda, 0 \leq \Delta{k}(q,w,r_{(\iota)},k) \leq H, q} \| {\bf v}_{i}(k) - x^{*} \|  \leq \max\limits_{k-H \leq \hat{k} \leq k, q} \| {\bf v}_{q}(\hat{k}) -  x^{*}\| 
\end{split}
\end{equation}}

where $ q \in \{1,\ldots,q\} $.

Then we have for scenario~1 and 3 
{ \begin{align*}
  & \max\limits_{\lambda, 0 \leq \Delta{k}(q,w,r_{(\iota)},k) \leq H, q \in V} \| \nabla{f}^{r_{(\iota)}}_{\lambda}({\bf v}_{q}(k-\Delta{k}(q,w,r_{(\iota)},k)) - \nabla{f}^{r_{(\iota)}}_{\lambda}({\bf v}_{i}(k)) \|   \leq \\
  & \hspace{2cm} 2 L \max\limits_{k-H \leq \hat{k} \leq k, q \in V} \| {\bf v}_{q}(\hat{k}) -  x^{*}\|.
\end{align*}}

 Therefore, for scenario~1 and 3 we have
{ \begin{align*}
 \| {\bf R}_{j,r_{(\iota)}}(k) \|   \leq \alpha_{k}   {\bf A}^{r_{(\iota)}} \| _{\infty} \|{\bf B}^{r_{(\iota)}}\| _{2,\infty} ( 2 L \max\limits_{k-H \leq \hat{k} \leq k, q \in V} \| {\bf v}_{q}(\hat{k}) -  x^{*}\| ) 
\end{align*}}
And squaring both sides and using $ 2ab \leq a^{2} + b^{2} $, we have
{ \begin{align*}
 \| {\bf R}_{j,r_{(\iota)}}(k) \| ^{2} & \leq \alpha_{k} ^{2}  {\bf A}^{r_{(\iota)}} \| ^{2} _{\infty} \|{\bf B}^{r_{(\iota)}}\| _{2,\infty} ^{2} 
 ( 4 L^{2}  \max\limits_{k-H \leq \hat{k} \leq k, q \in V} \| {\bf v}_{q}(\hat{k}) -  x^{*}\|^{2} ) 
\end{align*}}

For scenario~2, we have the contribution relative to each $ w \in \Gamma_{fit}, r_{(\iota)} \cap \Gamma_{(\iota),r_{(\iota)}} $ of  $(\nabla{f}^{r_{(\iota)}}_{\lambda}({\bf v}_{q}(k-\Delta{k}(q,w,r_{(\iota)},k)) - \nabla{f}^{r_{(\iota)}}_{\lambda}({\bf v}_{i}(k))) $ satisfying the inequality 
{\begin{align*}
  &  \| \nabla{f}^{r_{(\iota)}}_{\lambda}({\bf v}_{q}(k-\Delta{k}(q,w,r_{(\iota)},k)) - \nabla{f}^{r_{(\iota)}}_{\lambda}({\bf v}_{i}(k)) \|   \leq \\
  & \| \nabla{f}^{r_{(\iota)}}_{\lambda}({\bf v}_{q}(k-\Delta{k}(q,w,r_{(\iota)},k)) - \nabla{f}^{r_{(\iota)}}_{\lambda}(x^{*}) \| 
  +  \| \nabla{f}^{r_{(\iota)}}_{\lambda}({\bf v}_{i}(k))-\nabla{f}^{r_{(\iota)}}_{\lambda}(x^{*}) \| \leq \\
& L\| {\bf v}_{q}(k-\Delta{k}(q,w,r_{(\iota)},k) - x^{*} \| 
+ L \| {\bf v}_{i}(k) - x^{*} \|.
\end{align*}}
And the contribution relative to each $ w \in \Gamma_{fit, r_{(\iota)}} \ \Gamma_{i, r_{(\iota)}} $ \\ 
of  $(\nabla{f}^{r_{(\iota)}}_{\lambda}({\bf v}_{q}(k-\Delta{k}(q,w,r_{(\iota)},k)) - \nabla{f}^{r_{(\iota)}}_{\lambda}({\bf v}_{i}(k))) $ satisfying the inequality 
{\begin{align*}
  &  \| \nabla{f}^{r_{(\iota)}}_{\lambda}({\bf v}_{q}(k-\Delta{k}(q,w,r_{(\iota)},k)) - \nabla{f}^{r_{(\iota)}}_{\lambda}({\bf v}_{i}(k)) \|   = \\
  &  \| \nabla{f}^{r_{(\iota)}}_{\lambda}({\bf v}_{i}(k)) \|   \leq  \| \nabla{f}^{r_{(\iota)}}_{\lambda}({\bf v}_{i}(k)) - \nabla{f}^{r_{(\iota)}}_{\lambda}({\bf x}^{*}) \| 
  +
  \| \nabla{f}^{r_{(\iota)}}_{\lambda}(x^{*}) \| \leq \\
& L\| {\bf v}_{i}(k) - {\bf x}^{*} \| 
+ L \| {\bf x}^{*} - {\bf x}^{\lambda,r_{(\iota)}} \|.
\end{align*}}
where we used the Lipschitz assumption   and $ \nabla{f}^{r_{(\iota)}}_{\lambda}({\bf x}^{\lambda,r_{(\iota)}}) = 0 $ in the last inequality.

Thus, we have for scenario~2 

{\begin{align*}
   & \max\limits_{\lambda, 0 \leq \Delta{k} (q,w,r_{(\iota)},k) \leq H, q}  \| \nabla{f}^{r_{(\iota)}}_{\lambda}({\bf v}_{q}(k-\Delta{k}(q,w,r_{(\iota)},k)) - \nabla{f}^{r_{(\iota)}}_{\lambda}({\bf v}_{i}(k)) \|   \leq \\
 & \max(\max\limits_{\lambda, 0 \leq \Delta{k}(q,w,r_{(\iota)},k) \leq H} (L\| {\bf v}_{q}(k-\Delta{k}(q,w,r_{(\iota)},k) - x^{*} \| \\
& + L \| {\bf v}_{i}(k) - x^{*} \|) ,  \max\limits_{\lambda, 0 \leq \Delta{k}(q,w,r_{(\iota)},k) \leq H} (L\| {\bf v}_{i}(k) - {\bf x}^{*} \| + L \| {\bf x}^{*} - {\bf x}^{\lambda,r_{(\iota)}} \|) 
\end{align*}}

Therefore, for scenario~2 in Division~1, that is, where Condition~4.1 is satisfied we have

{\begin{align*}
   & \max\limits_{\lambda, 0 \leq \Delta{k} (q,w,r_{(\iota)},k) \leq H, q}  \| \nabla{f}^{r_{(\iota)}}_{\lambda}({\bf v}_{q}(k-\Delta{k}(q,w,r_{(\iota)},k)) - \nabla{f}^{r_{(\iota)}}_{\lambda}({\bf v}_{i}(k)) \|   \leq \\
 & \max(\max\limits_{\lambda, 0 \leq \Delta{k}(q,w,r_{(\iota)},k) \leq H} (L\| {\bf v}_{q}(k-\Delta{k}(q,w,r_{(\iota)},k) - x^{*} \| \\
& + L \| {\bf v}_{i}(k) - x^{*} \|) ,  \max\limits_{\lambda, 0 \leq \Delta{k}(q,w,r_{(\iota)},k) \leq H} (L\| {\bf v}_{i}(k) - {\bf x}^{*} \| + L \| {\bf x}^{*} - {\bf x}^{\lambda,r_{(\iota)}} \|) \\
\leq & \max( 2 L \max\limits_{k-H \leq \hat{k} \leq k, q} \| {\bf v}_{q}(\hat{k}) -  x^{*}\| , 2 L \max\limits_{k-H \leq \hat{k} \leq k, q} \| {\bf v}_{q}(\hat{k}) -  x^{*}\| )
\\
\leq &  2 L \max\limits_{k-H \leq \hat{k} \leq k, q \in V} \| {\bf v}_{q}(\hat{k}) -  x^{*}\| .
\end{align*}}

where we used \eqref{3.64a} and \eqref{3.64b} in the second inequality.

Therefore, for scenario~2 in Division~2, that is, where Condition~6 is not satisfied we have

{\begin{align*}
   & \max\limits_{\lambda, 0 \leq \Delta{k} (q,w,r_{(\iota)},k) \leq H, q}  \| \nabla{f}^{r_{(\iota)}}_{\lambda}({\bf v}_{q}(k-\Delta{k}(q,w,r_{(\iota)},k)) - \nabla{f}^{r_{(\iota)}}_{\lambda}({\bf v}_{i}(k)) \|   \leq \\
 & \max(\max\limits_{\lambda, 0 \leq \Delta{k}(q,w,r_{(\iota)},k) \leq H} (L\| {\bf v}_{q}(k-\Delta{k}(q,w,r_{(\iota)},k) - x^{*} \| \\
& + L \| {\bf v}_{i}(k) - x^{*} \|) ,  \max\limits_{\lambda, 0 \leq \Delta{k}(q,w,r_{(\iota)},k) \leq H} (L\| {\bf v}_{i}(k) - {\bf x}^{*} \| + L \| {\bf x}^{*} - {\bf x}^{\lambda,r_{(\iota)}} \|) \\
\leq & \max( 2 L \max\limits_{k-H \leq \hat{k} \leq k, q} \| {\bf v}_{q}(\hat{k}) -  x^{*}\| ,  L \max\limits_{k-H \leq \hat{k} \leq k, q} \| {\bf v}_{q}(\hat{k}) -  x^{*}\| + L \| {\bf x}^{*} - {\bf x}^{\lambda,r_{(\iota)}} \|) 
\\
\leq &  2 L \max\limits_{k-H \leq \hat{k} \leq k, q} \| {\bf v}_{q}(\hat{k}) -  x^{*}\|  + L \max\limits_{\lambda} \| {\bf x}^{*} - {\bf x}^{\lambda,r_{(\iota)}} \| .
\end{align*}}

where we used \eqref{3.64b} in the second inequality.

Then for scenarios included in Division~1 we have
{ \begin{align*}
 \| {\bf R}_{j,r_{(\iota)}}(k) \|   \leq \alpha_{k}  \max\limits_{(\iota)}\| {\bf A}^{r_{(\iota)}} \| _{\infty} \|{\bf B}^{r_{(\iota)}}\| _{2,\infty} ( 2 L \max\limits_{k-H \leq \hat{k} \leq k, q \in V} \| {\bf v}_{q}(\hat{k}) -  x^{*}\| ) 
\end{align*}}

And squaring both sides and using $ 2ab \leq a^{2} + b^{2} $ we have
{ \begin{align*}
 \| {\bf R}_{j,r_{(\iota)}}(k) \| ^{2} & \leq \alpha_{k} ^{2} \max\limits_{(\iota)} \| {\bf A}^{r_{(\iota)}} \| ^{2} _{\infty} \|{\bf B}^{r_{(\iota)}}\| _{2,\infty} ^{2}  
 ( 4 L^{2}  \max\limits_{k-H \leq \hat{k} \leq k, q \in V} \| {\bf v}_{q}(\hat{k}) -  x^{*}\|^{2} ) 
\end{align*}}
For scenarios in Division~2, that is, Scenario~2 where Condition~4.1 is not satisfied, we have

{ \begin{align*}
 \| {\bf R}_{j,r_{(\iota)}}(k) \|  & \leq \alpha_{k}  \max\limits_{(\iota)}\| {\bf A}^{r_{(\iota)}} \| _{\infty} \|{\bf B}^{r_{(\iota)}}\| _{2,\infty} \\
 & \times ( 2 L \max\limits_{k-H \leq \hat{k} \leq k} \| {\bf v}_{q}(\hat{k}) -  x^{*}\| + L \max\limits_{(\iota), \lambda} \| {\bf x}^{*} - {\bf x}^{\lambda,r_{(\iota)}} \|) 
\end{align*}}
And squaring both sides and using $ 2ab \leq a^{2} + b^{2} $, we have
{ \begin{align*}
 \| {\bf R}_{j,r_{(\iota)}}(k) \|^{2}  &\leq \alpha_{k}^{2}  \max\limits_{(\iota)}\| {\bf A}^{r_{(\iota)}} \|^{2} _{\infty} \|{\bf B}^{r_{(\iota)}}\|^{2} _{2,\infty} \\
 & \times( 4 L^{2} \max\limits_{k-H \leq \hat{k} \leq k} \| {\bf v}_{q}(\hat{k}) -  x^{*}\| + 2 L^{2} \max\limits_{(\iota), \lambda} \| {\bf x}^{*} - {\bf x}^{\lambda,r_{(\iota)}} \|^{2}) 
\end{align*}}

And we can get the upper bounds for $ \| {\bf \epsilon}_{j,r_{(\iota)}}(k) \| $ and $ \| {\bf \epsilon}_{j,r_{(\iota)}}(k) \| ^{2} $ for both divisions by substituting $ {\bf \epsilon}_{j,r_{(\iota)}}(k) = \frac{1}{\alpha_{k}} \| {\bf R}_{j,r_{(\iota)}}(k) \| $, respectively.
For all scenarios irrespective of which part they belong, we can upper bound by the bound which is a maximum for both parts, that is the bound of Division~2, so we have
{ \begin{align*}
 \| {\bf R}_{j,r_{(\iota)}}(k) \|  & \leq \alpha_{k}  \max\limits_{(\iota)}\| {\bf A}^{r_{(\iota)}} \| _{\infty} \|{\bf B}^{r_{(\iota)}}\| _{2,\infty} \\
 & \times ( 2 L \max\limits_{k-H \leq \hat{k} \leq k} \| {\bf v}_{q}(\hat{k}) -  x^{*}\| + L \max\limits_{(\iota), \lambda} \| {\bf x}^{*} - {\bf x}^{\lambda,r_{(\iota)}} \|) 
\end{align*}}
And squaring both sides and using $ 2ab \leq a^{2} + b^{2} $, we have
{ \begin{align*}
 \| {\bf R}_{j,r_{(\iota)}}(k) \|^{2}  & \leq \alpha_{k}^{2}  \max\limits_{(\iota)}\| {\bf A}^{r_{(\iota)}} \|^{2} _{\infty} \|{\bf B}^{r_{(\iota)}}\|^{2} _{2,\infty} \\
 & \hspace{-0.5cm} \times ( 4 L^{2} \max\limits_{k-H \leq \hat{k} \leq k} \| {\bf v}_{q}(\hat{k}) -  x^{*}\| + 2 L^{2} \max\limits_{(\iota), \lambda} \| {\bf x}^{*} - {\bf x}^{\lambda,r_{(\iota)}} \|^{2}) 
\end{align*}}
%


%
%

And then
{ \begin{equation}\label{3.67a}
\begin{split}
\sum_{j=1}^{n} & \| {\bf \epsilon}_{j,r_{(\iota)}} (k) \|  \leq  \max\limits_{(\iota)} \|{\bf A}^{r_{(\iota)}} \| _{\infty} \|{\bf B}^{r_{(\iota)}}\| _{2,\infty} \\
& \times ( 2 L \max\limits_{k-H \leq \hat{k} \leq k} \sum_{j=1}^{n}\| {\bf v}_{q}(\hat{k}) -  x^{*}\| + n L \max\limits_{(\iota), \lambda} \| {\bf x}^{*} - {\bf x}^{\lambda,r_{(\iota)}} \| ) \end{split}
\end{equation}}
And squaring both sides and using $ 2ab \leq a^{2} + b^{2} $, we have
{ \begin{equation}\label{3.67a}
\begin{split}
\sum_{j=1}^{n} & \| {\bf \epsilon}_{j,r_{(\iota)}} (k) \|^{2}  \leq  \max\limits_{(\iota)} \|{\bf A}^{r_{(\iota)}} \|^{2} _{\infty} \|{\bf B}^{r_{(\iota)}}\|^{2} _{2,\infty}  \\
& \times ( 4 L^{2}\max\limits_{k-H \leq \hat{k} \leq k} \sum_{j=1}^{n}\| {\bf v}_{q}(\hat{k}) -  x^{*}\| + 2 n L^{2} \max\limits_{(\iota), \lambda} \| {\bf x}^{*} - {\bf x}^{\lambda,r_{(\iota)}} \| ) \end{split}
\end{equation}}

\subsection{Proof of Lemma~1}

\begin{align*}
     \| & {\bf v}_{l}(k+1) -  {\bf x}^{*}\|^{2} 
      \leq \sum_{j=1}^{n} [{\bf W}(k+1)]_{l,j} \| {\bf x}_{j}(k+1)- {\bf x}^{*}\|^{2} \\
     & \leq \sum_{j=1}^{n}{[\bf W}(k+1)]_{l,j} \| \sum_{(\iota)=1}^{p} \gamma_{(\iota)} [{\bf v}_{j}(k) -\alpha_{k}\widehat{\nabla{f}}^{(\iota)}({\bf v}_{j}(k))]+\gamma_{(0)} {\bf v}_{j}(k) - {\bf x}^{*} \|^{2} \\
     & \leq \sum_{j=1}^{n}[{\bf W}(k+1)]_{l,j} \| {\bf v}_{j}(k) -\alpha_{k} \sum_{(\iota)=1}^{p}\widehat{\nabla{f}}^{(\iota)}({\bf v}_{j}(k))- {\bf x}^{*} \|^{2} \\
     & \leq \sum_{j=1}^{n}[{\bf W}(k+1)]_{l,j} [\| {\bf v}_{j}(k) - {\bf x}^{*} \|^{2} - 2 \alpha_{k}\sum_{(\iota)=1}^{p} \gamma_{(\iota)}\langle \widehat{\nabla{f}}^{(\iota)}({\bf v}_{j}(k)), {\bf v}_{j}(k) - {\bf x}^{*} \rangle \\
     & \hspace{2cm} + \alpha_{k}^{2} \|\sum_{(\iota)=1}^{p}\gamma_{(\iota)}\widehat{\nabla{f}}^{(\iota)}({\bf v}_{j}(k))\|^{2}] \\
\end{align*}
     
\begin{align*}
     & \leq \sum_{j=1}^{n}[{\bf W}(k+1)]_{l,j} [\|  {\bf v}_{j}(k) - {\bf x}^{*} \|^{2} + 2 \alpha_{k} \sum_{(\iota)=1}^{p} \gamma_{(\iota)} \langle \nabla{f}^{(\iota)}({\bf v}_{j}(k)), {\bf x}^{*} - {\bf v}_{j}(k) \rangle \\
     & \hspace{2cm} + 2 \alpha_{k} \sum_{(\iota)=1}^{p} \gamma_{(\iota)}\langle {\bf \epsilon}_{j,r_{(\iota)}}(k), {\bf v}_{j}(k) - {\bf x}^{*} \rangle \\
     & \hspace{2cm} + \alpha_{k}^{2} \|\sum_{(\iota)=1}^{p}\gamma_{(\iota)}(\nabla{f}^{(\iota)}({\bf v}_{j}(k)) - {\bf \epsilon}_{j,r_{(\iota)}}(k))\|^{2}] \\
     & \leq \sum_{j=1}^{n}[{\bf W}(k+1)]_{l,j} [ \|  {\bf v}_{j}(k) - {\bf x}^{*} \|^{2} + 2 \alpha_{k} \sum_{(\iota)=1}^{p} \gamma_{(\iota)} \langle \nabla{f}^{(\iota)}({\bf v}_{j}(k)), {\bf x}^{*} - {\bf v}_{j}(k) \rangle \\ 
     & \hspace{2cm} + 2 \alpha_{k} \sum_{(\iota)=1}^{p} \gamma_{(\iota)}\langle {\bf \epsilon}_{j,r_{(\iota)}}(k), {\bf v}_{j}(k) - {\bf x}^{*} \rangle \\
     & \hspace{2cm} + \alpha_{k}^{2} ( 1- \gamma_{(0)} ) \sum_{(\iota)=1}^{p}\gamma_{(\iota)} \|\nabla{f}^{(\iota)}({\bf v}_{j}(k)) - {\bf \epsilon}_{j,r_{(\iota)}}(k)\|^{2}
\end{align*}

where the first inequality follows from Jensen's inequality and convexity of $ \| x - b  \|^{2} $ and where we used Jensen's inequality for the last inequality.
Then summing from $i=1$ to $n$ and knowing that the sum of each column is less than or equal to $ 1 -\mu $ then the lemma follows.

\subsection{ Proof of Lemma~2}

Having Lemma~1 then we have
 \begin{align*}
     \sum_{l=1}^{n} & \|  {\bf v}_{l}(k+1) - {\bf x}^{*}\|^{2}  \leq (1-\mu) \sum_{j=1}^{n}\|{\bf v}_{j}(k)- {\bf x}^{*}\|^{2} 
     \\ & + 2 \alpha_{k}\sum_{j=1}^{n}\sum_{(\iota)=1}^{p}\gamma_{(\iota)} \langle {\bf \epsilon}_{j,r_{(\iota)}}(k), {\bf v}_{j} - {\bf x}^{*} \rangle + 2 (1-\gamma_{(0)}) \alpha_{k}^{2}  \sum_{j=1}^{n}\sum_{(\iota)=1}^{p}\gamma_{(\iota)} \|{\bf \epsilon}_{j,r_{(\iota)}}(k)\|^{2} \\
   & + 2 \alpha_{k} \sum_{j=1}^{n}\sum_{(\iota)=1}^{p}\gamma_{(\iota)} \langle \nabla{f}^{(\iota)}({\bf v}_{j}(k)), {\bf x}^{*} - {\bf v}_{j}(k) \rangle \\
  & \hspace{2cm} + 2 (1-\gamma_{(0)}) \alpha_{k}^{2} \sum_{j=1}^{n}\sum_{(\iota)=1}^{p}\gamma_{(\iota)} \| \nabla{f}^{(\iota)}({\bf v}_{j}(k)) \| ^{2}
\end{align*}

But
{ \begin{equation}\label{des_exp}
\begin{split}
\langle \nabla{f}^{(\iota)}({\bf v}_{j}(k)), {\bf x}^{*} - & {\bf v}_{j}(k) \rangle   = - \langle \nabla{f}^{(\iota)}({\bf v}_{j}(k)), {\bf v}_{j}(k) -{\bf x}^{*} \rangle \\
& = - \langle \nabla{f}^{(\iota)}({\bf v}_{j}(k)) - \nabla{f}^{(\iota)}({\bf x}^{*}), {\bf v}_{j}(k) -{\bf x}^{*} \rangle 
\end{split}
\end{equation}}
Notice that in the second equality we used $ {\bf x}^{*} $ to be a minimizer of $ f^{(\iota)} $ that is $ \nabla{f}^{(\iota)}({\bf x}^{*})=0 $ since $ f^{(\iota)}({\bf x}^{*})=f^{(\iota)}({\bf x}^{(\iota)}) $ for all $(\iota)$.
While 
{ \begin{equation}
\begin{split}
 \nabla{f}^{(\iota)}({\bf v}_{j}(k)) - \nabla{f}^{(\iota)}({\bf x}^{*}) = a_{v_{j}(k), x^{*}} \| {\bf v}_{j}(k) - {\bf x}^{*} \| \overrightarrow{u}
\end{split}
\end{equation}}
where $ \| \overrightarrow{u} \| = 1 $ and $  0 \leq a_{v_{j}(k), x^{*}} \leq L $,
and 
{ \begin{equation}
\begin{split}
& {\bf v}_{j}(k) -{\bf x}^{*} = \| {\bf v}_{j}(k) -{\bf x}^{*} \| \overrightarrow{v}
\end{split}
\end{equation}}
where $ \| \overrightarrow{v} \| = 1 $.
Using what preceded we have the expression in \eqref{des_exp} equal to
 \begin{align*}
\langle \nabla{f}^{(\iota)}({\bf v}_{j}(k)), {\bf x}^{*} - & {\bf v}_{j}(k) \rangle   = - \langle \nabla{f}^{(\iota)}({\bf v}_{j}(k)) - \nabla{f}^{(\iota)}({\bf x}^{*}), {\bf v}_{j}(k) -{\bf x}^{*} \rangle \\
 & = -a_{v_{j}(k), x^{*}} \| {\bf v}_{j}(k) - {\bf x}^{*} \| ^{2}  \langle \overrightarrow{u},\overrightarrow{v} \rangle
\end{align*}
But since $  \langle \overrightarrow{u},\overrightarrow{v} \rangle \geq 0 $ due to the monotonicity of the gradient we have $  0 \leq  b_{v_{j}(k), x^{*}} = \langle \overrightarrow{u},\overrightarrow{v} \rangle \leq  1 $.
Then
 \begin{equation}\label{res1}
\begin{split}
 \langle \nabla{f}^{(\iota)}({\bf v}_{j}(k)), {\bf x}^{*} - {\bf v}_{j}(k) \rangle   = -a_{v_{j}(k), x^{*}} \| {\bf v}_{j}(k) - {\bf x}^{*} \| ^{2}  b_{v_{j}(k), x^{*}}
\end{split}
\end{equation}
where $ 0 \leq b_{v_{j}(k), x^{*}} \leq 1 $.


Similarly, 
 \begin{equation}\label{res2}
\begin{split}
\| \nabla{f}^{(\iota)}  ({\bf v}_{j}(k)) \|^{2} & =  \langle \nabla{f}^{(\iota)}({\bf v}_{j}(k)) - \nabla{f}^{(\iota)}({\bf x}^{*}),\nabla{f}^{(\iota)}({\bf v}_{j}(k)) - \nabla{f}^{(\iota)}({\bf x}^{*}) \rangle \\
& = a_{v_{j}(k), x^{*}}^{2} \| {\bf v}_{j}(k) -{\bf x}^{*} \| ^{2} \langle \overrightarrow{u},\overrightarrow{u} \rangle  
= a_{v_{j}(k), x ^{*}}^{2} \| {\bf v}_{j}(k) -{\bf x}^{*} \| ^{2}
\end{split}
\end{equation}
Then substituting \eqref{res1} and \eqref{res2} in \eqref{type1_exp} we have
 \begin{equation}\label{type1_exp}
\begin{split}
     \sum_{l=1}^{n}\|  {\bf v}_{l}(k+1) & - {\bf x}^{*}\|^{2}  \leq  (1-\mu) \sum_{j=1}^{n}\|{\bf v}_{j}(k)- {\bf x}^{*}\|^{2} 
  \\ &  + 2 \alpha_{k}\sum_{j=1}^{n}\sum_{(\iota)=1}^{p}\gamma_{(\iota)} \langle {\bf \epsilon}_{j,r_{(\iota)}}(k), {\bf v}_{j} - {\bf x}^{*} \rangle \\
    & + 2 (1-\gamma_{(0)}) \alpha_{k}^{2} \sum_{j=1}^{n}\sum_{(\iota)=1}^{p}\gamma_{(\iota)}\|{\bf \epsilon}_{j,r_{(\iota)}}(k)\|^{2} \\
    & - 2 \alpha_{k} \sum_{j=1}^{n}\sum_{(\iota)=1}^{p}\gamma_{(\iota)} a_{v_{j}(k), x^{*}} (b_{v_{j}(k), x^{*}} \\
    & \hspace{2cm} - (1-\gamma_{(0)}) \alpha_{k} a_{v_{j}(k), x^{*}} ) \| {\bf v}_{j}(k) - {\bf x}^{*} \| ^{2}   
\end{split}
\end{equation}

\section{Proof of Lemma~3}

From Lemma~1, we have
 \begin{equation}
\begin{split}
     \sum_{l=1}^{n}\| & {\bf v}_{l}(k+1) -  {\bf x}^{*}\|^{2} \leq (1-\mu) \sum_{j=1}^{n}\|{\bf v}_{j}(k)- {\bf x}^{*}\|^{2}  
 \\
 &   + 2 \alpha_{k}\sum_{j=1}^{n}\sum_{(\iota)=1}^{p}\gamma_{(\iota)} \langle {\bf \epsilon}_{j,r_{(\iota)}}(k), {\bf v}_{j} - {\bf x}^{*} \rangle \\
   & + 2 (1-\gamma_{(0)}) \alpha_{k}^{2} \sum_{j=1}^{n}\sum_{(\iota)=1}^{p}\gamma_{(\iota)} \|{\bf \epsilon}_{j,r_{(\iota)}}(k)\|^{2} \\
 &  + 2 \alpha_{k} \sum_{j=1}^{n}\sum_{(\iota)\in I }\gamma_{(\iota)} \langle \nabla{f}^{(\iota)}({\bf v}_{j}(k)), {\bf x}^{*} - {\bf v}_{j}(k) \rangle \\ 
   & + 2 \alpha_{k} \sum_{j=1}^{n}\sum_{(\iota)\in I^{\complement} }\gamma_{(\iota)} \langle \nabla{f}^{(\iota)}({\bf v}_{j}(k)), {\bf x}^{*} - {\bf v}_{j}(k) \rangle \\
   & + 2 (1-\gamma_{(0)}) \alpha_{k}^{2} \sum_{j=1}^{n}\sum_{(\iota)=1}^{p}\gamma_{(\iota)} \| \nabla{f}^{(\iota)}({\bf v}_{j}(k)) \| ^{2}
\end{split}
\end{equation}

But for $(\iota) \in I$ we have $ f^{(\iota)}({\bf x}^{(\iota)}) < f^{(\iota)}({\bf x}^{*})$ and for $(\iota) \in I^{\complement}$ we have $ f^{(\iota)}({\bf x}^{*})= f^{(\iota)}({\bf x}^{(\iota)})$, then the above inequality becomes
 \begin{align*}
     & \sum_{l=1}^{n}\| {\bf v}_{l}(k+1)  - {\bf x}^{*}\|^{2}   \leq 
     (1-\mu) \sum_{j=1}^{n}\|{\bf v}_{j}(k)- {\bf x}^{*}\|^{2} \\ & + 2 \alpha_{k}\sum_{j=1}^{n}\sum_{(\iota)=1}^{p}\gamma_{(\iota)} \langle {\bf \epsilon}_{j,r_{(\iota)}}(k), {\bf v}_{j} - {\bf x}^{*} \rangle \\
     & + 2 (1-\gamma_{(0)}) \alpha_{k}^{2} \sum_{j=1}^{n}\sum_{(\iota)=1}^{p}\gamma_{(\iota)}\|{\bf \epsilon}_{j,r_{(\iota)}}(k)\|^{2} \\
   &  + 2 \alpha_{k} \sum_{j=1}^{n}\sum_{(\iota)\in I }\gamma_{(\iota)} \langle \nabla{f}^{(\iota)}({\bf v}_{j}(k)), {\bf x}^{*} - {\bf x}^{(\iota)} - {\bf x}^{(\iota)} - {\bf v}_{j}(k) \rangle \\
     & + 2 \alpha_{k} \sum_{j=1}^{n}\sum_{(\iota)\in I^{\complement} }\gamma_{(\iota)} \langle \nabla{f}^{(\iota)}({\bf v}_{j}(k)), {\bf x}^{(\iota)} - {\bf v}_{j}(k) \rangle \\ 
   &  + 2 (1-\gamma_{(0)}) \alpha_{k}^{2} \sum_{j=1}^{n}\sum_{(\iota)=1}^{p}\gamma_{(\iota)} \| \nabla{f}^{(\iota)}({\bf v}_{j}(k)) \| ^{2}
\end{align*}
Notice that in the fifth term of RHS we used the strong convexity of $ f^{(\iota)}(x) $ for $ (\iota) \in I^{\complement} $. That is, since $ f^{(\iota)}(x^{*})=f^{(\iota)}(x^{(\iota)}) $ for $ (\iota) \in I^{\complement} $ and $ f^{(\iota)}(x) $ strongly convex for $ (\iota) \in I ^{\complement} $ we have a unique minimizer and therefore $ x^{*}=x^{(\iota)} $ on $ I ^{\complement} $.

Then the above becomes
 \begin{equation}\label{type2_error_bound_no_subst}
\begin{split}
    &  \sum_{l=1}^{n}\|  {\bf v}_{l}(k+1) - {\bf x}^{*}\|^{2}  \leq (1-\mu) \sum_{j=1}^{n}\|{\bf v}_{j}(k)- {\bf x}^{*}\|^{2} 
 \\ &   + 2 \alpha_{k}\sum_{j=1}^{n}\sum_{(\iota)=1}^{p}\gamma_{(\iota)} \langle {\bf \epsilon}_{j,r_{(\iota)}}(k), {\bf v}_{j} - {\bf x}^{*} \rangle \\
    & + 2 (1-\gamma_{(0)}) \alpha_{k}^{2} \sum_{j=1}^{n}\sum_{(\iota)=1}^{p}\gamma_{(\iota)} \|{\bf \epsilon}_{j,r_{(\iota)}}(k)\|^{2} \\
   & + 2 \alpha_{k} \sum_{j=1}^{n}\sum_{(\iota)\in I }\gamma_{(\iota)} \langle \nabla{f}^{(\iota)}({\bf v}_{j}(k)), {\bf x}^{*} - {\bf x}^{(\iota)} \rangle \\ 
    & + 2 \alpha_{k} \sum_{j=1}^{n}\sum_{(\iota)=1}^{p}\gamma_{(\iota)} \langle \nabla{f}^{(\iota)}({\bf v}_{j}(k)), {\bf x}^{(\iota)} - {\bf v}_{j}(k) \rangle  \\
    & + 2 (1-\gamma_{(0)}) \alpha_{k}^{2} \sum_{j=1}^{n}\sum_{(\iota)=1}^{p}\gamma_{(\iota)} \| \nabla{f}^{(\iota)}({\bf v}_{j}(k)) \| ^{2}
\end{split}
\end{equation}

Meanwhile, $  \langle \nabla{f}^{(\iota)}({\bf v}_{j}(k)), {\bf x}^{*} - {\bf x}^{(\iota)} \rangle \leq  \langle \nabla{f}^{(\iota)}({\bf v}_{j}(k)), {\bf x}^{*} - {\bf v}_{j}(k) \rangle  \\ +  \langle \nabla{f}^{(\iota)}({\bf v}_{j}(k)), {\bf v}_{j}(k) - {\bf x}^{(\iota)} \rangle $. 
And 
\begin{equation}\label{ineq_subs2}
\begin{split}
 \langle \nabla{f}^{(\iota)}({\bf v}_{j}(k)), {\bf x}^{*} - {\bf v}_{j}(k) \rangle & \leq f^{(\iota)}({\bf x}^{*}) - f^{(\iota)}({\bf v}_{j}(k)) \\
 & \leq  f^{(\iota)}({\bf x}^{*}) - f^{(\iota)}({\bf x}^{(\iota)})
\end{split}
\end{equation}
since $ {\bf x}^{(\iota)} $ is the minimizer of $ f^{(\iota)}(x) $.
But $ f^{(\iota)}({\bf x}^{*}) - f^{(\iota)}({\bf x}^{(\iota)}) \leq \langle \nabla{f}^{(\iota)}({\bf x}^{*}) - \nabla{f}^{(\iota)}({\bf x}^{(\iota)}) , {\bf x}^{*} - {\bf x}^{(\iota)} \rangle $ where we used $ \nabla{f}^{(\iota)}({\bf x}^{(\iota)})= 0 $ in the inequality.

Let $ {\bf x}^{*} - {\bf x}^{(\iota)} = \| {\bf x}^{*} - {\bf x}^{(\iota)} \| \overrightarrow{v}^{'} $ where $ \| \overrightarrow{v}^{'} \| = 1 $.
Using what preceded we have the expression in \eqref{ineq_subs2} less than or equal to
 \begin{equation}\label{ineq_subs3}
\begin{split}
\langle \nabla{f}^{(\iota)}({\bf v}_{j}(k)), {\bf x}^{*} - {\bf v}_{j}(k) \rangle & \leq f^{(\iota)}({\bf x}^{*}) - f^{(\iota)}({\bf v}_{j}(k))  \\
& \leq  f^{(\iota)}({\bf x}^{*}) - f^{(\iota)}({\bf x}^{(\iota)}) \\
& \leq \langle \nabla{f}^{(\iota)}({\bf x}^{*}) - \nabla{f}^{(\iota)}({\bf x}^{(\iota)}), {\bf x}^{*} -{\bf x}^{(\iota)} \rangle \\
& = a_{x^{*},(\iota)} \| {\bf x}^{*} - {\bf x}^{(\iota)} \| ^{2}  \langle \overrightarrow{u}^{'},\overrightarrow{v}^{'} \rangle
\end{split}
\end{equation}
But since $  \langle \overrightarrow{u}^{'},\overrightarrow{v}^{'} \rangle \geq 0 $ due to the monotonicity of the gradient we have $  0 \leq \langle \overrightarrow{u}^{'},\overrightarrow{v}^{'} \rangle \leq  1 $.
 While 
{ \begin{equation}
\begin{split}
 \nabla{f}^{(\iota)}({\bf x}^{*}) - \nabla{f}^{(\iota)}({\bf x}^{(\iota)}) = a_{x^{*},(\iota)} \| {\bf x}^{*} - {\bf x}^{(\iota)} \| \overrightarrow{u}^{'}
\end{split}
\end{equation}}
where $ \| \overrightarrow{u}^{'} \| = 1 $ and $  0 \leq a_{x^{*},(\iota)} \leq L $.
Then
\begin{align*}
 f^{(\iota)}({\bf x}^{*}) - f^{(\iota)}({\bf x}^{(\iota)}) & \leq \langle \nabla{f}^{(\iota)}({\bf x}^{*}) - \nabla{f}^{(\iota)}({\bf x}^{(\iota)}) , {\bf x}^{*} - {\bf x}^{(\iota)} \rangle \\
 & = a_{x^{*},(\iota)} \| {\bf x}^{*} - {\bf x}^{(\iota)} \| ^{2}  b_{x^{*},(\iota)}
\end{align*}
where $ 0 \leq b_{x^{*},(\iota)} \leq 1 $.

Similarly, we also have
{ \begin{equation}\label{res5}
\begin{split}
  \langle \nabla{f}^{(\iota)}({\bf v}_{j}(k)) - \nabla{f}^{(\iota)}({\bf x}^{(\iota)}), & {\bf v}_{j}(k) -{\bf x}^{(\iota)} \rangle \\
  & = a_{v_{j}(k),(\iota)} \| {\bf v}_{j}(k) - {\bf x}^{(\iota)} \| ^{2}  b_{v_{j}(k),(\iota)}
\end{split}
\end{equation}}
From what have preceded we have $ \langle \nabla{f}^{(\iota)}({\bf v}_{j}(k)), {\bf x}^{*} - {\bf x}^{(\iota)} \rangle \leq a_{x^{*},(\iota)} b_{x^{*},(\iota)} \| {\bf x}^{*} - {\bf x}^{(\iota)} \| ^{2} + a_{v_{j}(k),(\iota)} b_{v_{j}(k),(\iota)} \| {\bf v}_{j}(k) - {\bf x}^{(\iota)} \| ^{2} $. Using this upper bound on the fourth term of RHS of \eqref{type2_error_bound_no_subst} then \eqref{type2_error_bound_no_subst} becomes     
 \begin{align*}
     & \sum_{l=1}^{n}\|  {\bf v}_{l}(k+1) - {\bf x}^{*}\|^{2}  \leq  (1-\mu) \sum_{j=1}^{n}\|{\bf v}_{j}(k)- {\bf x}^{*}\|^{2} 
\\ &    + 2 \alpha_{k}\sum_{j=1}^{n}\sum_{(\iota)=1}^{p}\gamma_{(\iota)} \langle {\bf \epsilon}_{j,r_{(\iota)}}(k), {\bf v}_{j} - {\bf x}^{*} \rangle \\
    & + 2 (1-\gamma_{(0)}) \alpha_{k}^{2} \sum_{j=1}^{n}\sum_{(\iota)=1}^{p}\gamma_{(\iota)} \|{\bf \epsilon}_{j,r_{(\iota)}}(k)\|^{2} \\
  &  + 2 \alpha_{k} \sum_{j=1}^{n}\sum_{(\iota)\in I }\gamma_{(\iota)} a_{x^{*},(\iota)} b_{x^{*},(\iota)} \| {\bf x}^{*} - {\bf x}^{(\iota)} \| ^{2} \\
    & + 2 \alpha_{k} \sum_{j=1}^{n}\sum_{(\iota)\in I }\gamma_{(\iota)} a_{v_{j}(k),(\iota)} b_{v_{j}(k),(\iota)} \| {\bf v}_{j}(k) - {\bf x}^{(\iota)} \| ^{2} \\ 
  &  + 2 \alpha_{k} \sum_{j=1}^{n}\sum_{(\iota)=1}^{p}\gamma_{(\iota)} \langle \nabla{f}^{(\iota)}({\bf v}_{j}(k)), {\bf x}^{(\iota)} - {\bf v}_{j}(k) \rangle \\  
    & + 2 (1-\gamma_{(0)}) \alpha_{k}^{2}  \sum_{j=1}^{n}\sum_{(\iota)=1}^{p}\gamma_{(\iota)} \| \nabla{f}^{(\iota)}({\bf v}_{j}(k)) \| ^{2}
\end{align*}

But
{ \begin{equation}\label{des_expaaa}
\begin{split}
\langle \nabla{f}^{(\iota)}({\bf v}_{j}(k)), {\bf x}^{(\iota)} - & {\bf v}_{j}(k) \rangle   = - \langle \nabla{f}^{(\iota)}({\bf v}_{j}(k)), {\bf v}_{j}(k) -{\bf x}^{(\iota)} \rangle \\
& = - \langle \nabla{f}^{(\iota)}({\bf v}_{j}(k)) - \nabla{f}^{(\iota)}({\bf x}^{(\iota)}), {\bf v}_{j}(k) -{\bf x}^{(\iota)} \rangle 
\end{split}
\end{equation}}
Notice that in the second equality we used $ \nabla{f}^{(\iota)}({\bf x}^{(\iota)})=0 $ since $ f^{(\iota)}({\bf x}^{*})=f^{(\iota)}({\bf x}^{(\iota)}) $ for all $(\iota)$.
While 
{ \begin{equation}
\begin{split}
 \nabla{f}^{(\iota)}({\bf v}_{j}(k)) - \nabla{f}^{(\iota)}({\bf x}^{(\iota)}) = a_{v_{j}(k), x^{(\iota)}} \| {\bf v}_{j}(k) - {\bf x}^{(\iota)} \| \overrightarrow{\hat{u}}
\end{split}
\end{equation}}
where $ \| \overrightarrow{\hat{u}} \| = 1 $ and $  0 \leq a_{v_{j}(k), x^{(\iota)}} \leq L $,
and 
{ \begin{equation}
\begin{split}
& {\bf v}_{j}(k) -{\bf x}^{(\iota)} = \| {\bf v}_{j}(k) -{\bf x}^{(\iota)} \| \overrightarrow{\hat{v}}
\end{split}
\end{equation}}
where $ \| \overrightarrow{\hat{v}} \| = 1 $.
Using what preceded we have the expression in \eqref{des_expaaa} equal to
\begin{align*}
\langle \nabla{f}^{(\iota)}({\bf v}_{j}(k)), {\bf x}^{(\iota)} - & {\bf v}_{j}(k) \rangle   = - \langle \nabla{f}^{(\iota)}({\bf v}_{j}(k)) - \nabla{f}^{(\iota)}({\bf x}^{(\iota)}), {\bf v}_{j}(k) -{\bf x}^{(\iota)} \rangle \\
 & = -a_{v_{j}(k), x^{(\iota)}} \| {\bf v}_{j}(k) - {\bf x}^{(\iota)} \| ^{2}  \langle \overrightarrow{\hat{u}},\overrightarrow{\hat{v}} \rangle
\end{align*}
But since $  \langle \overrightarrow{\hat{u}},\overrightarrow{\hat{v}} \rangle \geq 0 $ due to the monotonicity of the gradient we have $  0 \leq  b_{v_{j}(k), x^{(\iota)}} = \langle \overrightarrow{\hat{u}},\overrightarrow{\hat{v}} \rangle \leq  1 $.
Then
 \begin{equation}\label{res1aaa}
\begin{split}
 \langle \nabla{f}^{(\iota)}({\bf v}_{j}(k)), {\bf x}^{(\iota)} - {\bf v}_{j}(k) \rangle   = -a_{v_{j}(k), x^{(\iota)}} \| {\bf v}_{j}(k) - {\bf x}^{(\iota)} \| ^{2}  b_{v_{j}(k), x^{(\iota)}}
\end{split}
\end{equation}
where $ 0 \leq b_{v_{j}(k), x^{(\iota)}} \leq 1 $.


Similarly, 
 \begin{equation}\label{res2aaa}
\begin{split}
\| \nabla{f}^{(\iota)}  ({\bf v}_{j}(k)) \|^{2} & =  \langle \nabla{f}^{(\iota)}({\bf v}_{j}(k)) - \nabla{f}^{(\iota)}({\bf x}^{(\iota)}),\nabla{f}^{(\iota)}({\bf v}_{j}(k)) - \nabla{f}^{(\iota)}({\bf x}^{(\iota)}) \rangle \\
& = a_{v_{j}(k), x^{(\iota)}}^{2} \| {\bf v}_{j}(k) -{\bf x}^{*} \| ^{2} \langle \overrightarrow{\hat{u}},\overrightarrow{\hat{u}} \rangle  
= a_{v_{j}(k), x ^{(\iota)}}^{2} \| {\bf v}_{j}(k) -{\bf x}^{(\iota)} \| ^{2}
\end{split}
\end{equation}

Let 
\begin{equation}\label{barx^i}
\begin{split}
\bar{\bf x}^{(\iota)} = argmax_{x^{(\iota)}} \| {\bf x}^{*} - {\bf x}^{(\iota)} \|^{2}
\end{split}
\end{equation}
and 
\begin{equation}\label{doublebarx^i}
\begin{split}
{\bar{\bar{\bf x}}}^{(\iota)} = argmax_{x^{(\iota)}} \| {\bf v}_{j}(k) - {\bf x}^{(\iota)} \|^{2}
\end{split}
\end{equation}

Then using the above two equalities and \eqref{res1aaa} and \eqref{res2aaa} we get the result.

\subsection{Martingale 1}

{\bf Proof of Lemma~4:}

Assume the following inequality holds a.s. for all $ k \geq k^{*} $
\begin{equation}
\begin{split}
v_{k+1} \leq a_{1}v_{k} + a_{2,k} \max\limits_{k-B \leq \hat{k} \leq k}   v_{k}
\end{split}
\end{equation}
where the variables are as defined in hypothesis.

We can list any consecutive $ B + 1 $ terms in an increasing order. Let us say we choose $ B+1 $ consecutive instants terms $ v_{k} $ beginning from instant $ k_{0} = \bar{k} - B $ until $ \bar {k} $. Take $ \rho =(a_{1} + a_{2,\bar{k}})^{\frac{1}{B+1}}  $ then $ v^{1} \leq v^{2} \leq \ldots \leq v^{B+1} $ where $ v^{l} = v_{\phi(l)} $ such that $ \phi(l) $ is a bijective mapping from $ [1, B+1] \rightarrow [ \bar{k} - B, \bar{k}]$. Then we can bound each $ v^{i} \leq \rho ^{B+1-i+m} V_{0} $ where $ 1 \leq i \leq B+1 $ and $ V_{0} = \max \frac{v^{i}}{\rho ^{i}} $, i.e., $ \rho < 1 $ and $ \rho ^{n} < \rho $ . And $ m $ is an arbitrary constant such that $ m \in \mathbb{Z}^{+} $. Let us choose $ \bar{k} \geq k^{*} - 1 $ where $ \bar {k} \geq B $ (i.e.,if we index from $ k=0 $). For a neat final result we choose $ m = \bar{k} $, that is, $ v^{i} = v_{\phi(\iota)} \leq \rho ^{B+1-i+\bar{k}} V_{0} $, i.e., $ v^{1} \leq \rho ^{\bar{k}+B}V_{0} , \ldots , v^{B+1} \leq \rho ^{\bar{k}} V_{0} $ .

For $ v_{\bar{k}} $ we have 
\begin{equation}
\begin{split}
v_{\bar{k}} \leq \rho ^ {\bar{k} + B - l } V_{0}
\end{split}
\end{equation}
where $ 0 \leq l \leq B $ 
since the terms $ v_{k} $ at instants from $ k_{0} = \bar{k} - B $ till $ \bar{k} $ can be put as defined earlier and $ v_{\bar{k}} $ can be any term in that order.

Take $ \bar{k} + 1 \geq k^{*} $. 

First since $ a_{1} + a_{2,\bar{k}} \leq 1 $ and  $ a_{1}+a_{2,k} \leq a_{1}+a_{2,\bar{k}} $ for $ k \geq \bar{k} $.

then we have $ 1  \leq (a_{1}+a_{2,\bar{k}})^{-\frac{B}{B+1}} $ and $ 1  \leq (a_{1}+a_{2,k})^{-\frac{B}{B+1}} $
which implies for $ k \geq \bar{k} $ that
\begin{align*}
a_{1} +  a_{2,k}  \rho ^ {-B}  & = a_{1} + a_{2,k}(a_{1}+a_{2,\bar{k}})^{-\frac{B}{B+1}} \\
& \leq a_{1}(a_{1}+a_{2,\bar{k}})^{-\frac{B}{B+1}}+ a_{2,k}(a_{1}+a_{2,\bar{k}})^{-\frac{B}{B+1}} \\
& = (a_{1}+a_{2,k})(a_{1}+a_{2,\bar{k}})^{-\frac{B}{B+1}} \\
& \leq (a_{1}+a_{2,\bar{k}})^{\frac{1}{B+1}} = \rho
\end{align*}
That is 
{ \begin{equation}\label{rho}
\begin{split}
a_{1} + & a_{2,k}  \rho ^ {-B}  \leq \rho
\end{split}
\end{equation}}
N.B. We aim to find the tightest upper bound for each case and the bound that holds for all cases, i.e., that takes into consideration the worst case possibility.

For Base case $ \bar{k} + 1 \geq k^{*} $ we have $ v_{\bar{k}} \leq \rho ^ {\bar{k} + B - l } V_{0} $ where $ 0 \leq l \leq B $ and $ \max\limits_{\bar{k}-B \leq \hat{k} \leq \bar{k}} v_{\hat{k}} \leq \rho^{\bar{k}}V_{0} $, then
\begin{align*}
 v_{\bar{k}+1} & \leq a_{1}v_{\bar{k}}+ a_{2,\bar{k}}\max\limits_{\bar{k}-B \leq \hat{k} \leq \bar{k}} v_{\hat{k}} \leq a_{1}\rho ^ {\bar{k}+B-l} V_{0} + a_{2,\bar{k}}\rho ^{\bar{k}} V_{0} \\
 & \leq  \ a_{1}\rho ^ {\bar{k} + B - l - (B-l)} V_{0} + a_{2,\bar{k}}\rho ^{\bar{k}} V_{0} = (a_{1} + a_{2,\bar{k}}) \rho ^ {\bar{k}} V_{0} \\ 
      & \leq \rho^{\bar{k} + B + 1} V_{0} \ \ \ a.s.
\end{align*}
For $ \bar{k} + 2 $ we can have $ \max\limits_{\bar{k}+1-B \leq \hat{k} \leq \bar{k} + 1} v_{\hat{k}} \leq \rho^{\bar{k}+1} V_{0} $ or $ \max\limits_{\bar{k}+1-B \leq \hat{k} \leq \bar{k} + 1} v_{\hat{k}} \leq \rho^{\bar{k}} V_{0} $ if the maximum which is $ \leq \rho^{\bar{k}}V_{0} $ is in $ v_{\bar{k}-B} $ or not, respectively.
Similarly with case $ \bar{k}+2 $ included and for $ \bar{k} +1 + n $ where $ 1 \leq n \leq B $ we can have  
 $ \max\limits_{\bar{k}+n-B \leq \hat{k} \leq \bar{k} + n} v_{\hat{k}} \leq \rho^{\bar{k}+i} V_{0} $ for each $ 0 \leq i \leq n $. And for each of these cases the set $ \{ v_{\bar{k}-B}, v_{\bar{k}-B+1},\ldots,v_{\bar{k}-B+n+1} \} $ contains terms that are $ \{ \rho^{\bar{k}}V_{0}, \rho^{\bar{k}+1}V_{0}, \ldots, \rho^{\bar{k}+i-1}V_{0}\} $ for each $ 0 \leq i \leq n $.

For base case $ k = \bar{k}+1$ we have $ v_{\bar{k}+1} \leq V_{0} \rho^{\bar{k}+B+1} $ as we have proved.
For induction case (i.e., we have $ v _{\bar{k}+n} \leq V_{0} \rho^{\bar{k}+B+1} $ and $ \max\limits_{\bar{k} + n - B \leq \hat{k} \leq \bar{k} + n} v_{\hat{k}} \leq \rho^{\bar{k}+i} V_{0} $ as mentioned in the previous paragraph) $ k = \bar{k}+2 $ until $ k =\bar{k}+B+1 $ we have $ B+1-i \geq 0 $ since $ 0 \leq i \leq n \leq B $. Then for this induction case where $ 1 \leq n \leq B $ and 
\begin{equation}\label{inda1}
\begin{split}
 v_{\bar{k}+1+n} & \leq a_{1}v_{\bar{k}+n}+ a_{2,\bar{k}+n}\max\limits_{\bar{k} + n - B \leq \hat{k} 
 \leq \bar{k} + n} v_{\hat{k}} \\
 & \leq a_{1}\rho ^ {\bar{k}+B+1} V_{0} + a_{2,\bar{k}+n}\rho ^{\bar{k}+i} V_{0} \\
 & \leq a_{1}\rho ^ {\bar{k}+B+1 - (B+1-i)} V_{0} + a_{2,\bar{k}+n}\rho ^{\bar{k}+i} V_{0}  \\
     & = (a_{1} + a_{2,\bar{k}+n}) \rho ^ {\bar{k}+i} V_{0} 
       \leq \rho^{\bar{k} + B + 1 + i} V_{0} \quad \text{almost surely}.
\end{split} 
\end{equation}
But the most relaxed bound that takes into consideration all cases $  0 \leq i \leq n $ for each $ n $ is $ v_{\bar{k}+1+n} \leq V_{0}\rho^{\bar{k}+B+1} $ and $ \max\limits_{\bar{k} + n - B \leq \hat{k} \leq \bar{k} + n} v_{\hat{k}} \leq \rho^{\bar{k}+i} V_{0} $ as mentioned in the previous paragraph. 
\\
Thus, for $ k = \bar{k} +1 $ until $ k = \bar{k} +B +1 $ we have $ v_{k} \leq V_{0} \rho ^{\bar{k}+B + 1} $ which is in turn used in the induction step of \eqref{inda1} and is the subsequent result.
 
For base case $ k = \bar{k} + B + 2 $ we have $ v_{\bar{k}+B +1} \leq \rho ^ {\bar{k}+B+1} V_{0} $ and $ \max\limits_{\bar{k}+1 \leq \hat{k} \leq \bar{k} + B + 1} v_{\hat{k}} \leq \rho^{\bar{k}+ B + 1} V_{0} $ then
\begin{align*}
 v_{\bar{k}+B+2} & \leq a_{1}v_{\bar{k}+B +1}+ a_{2,\bar{k}+B+1}\max\limits_{\bar{k} + 1 \leq \hat{k} \leq \bar{k} + B +1} v_{\hat{k}} \\
 & \leq a_{1}\rho ^ {\bar{k}+B+1} V_{0} + a_{2,\bar{k}+B+1}\rho ^{\bar{k}+B+1-B} V_{0} \\
 & = (a_{1} + a_{2,\bar{k}+B+1}\rho^{-B}) \rho ^ {\bar{k}+B+1} V_{0} \\
 & \leq \rho^{\bar{k} + B + 2} V_{0} \ \ \ a.s.
\end{align*}
For induction case $ k = \bar{k}+B+2+n $ where $ 1 \leq n \leq B + 1 $, that is, from $  k = \bar{k} + B + 3 $ until $ k = \bar{k}+2B+2 $ (i.e., we have $ v_{\bar{k}+B+n+1} \leq \rho^{\bar{k}+B+n+1} V_{0} $, that is, $ v_{k} \leq \rho^{k} V_{0} $ 
and $ \max\limits_{\bar{k}+n+1 \leq \hat{k} \leq \bar{k} + B + n+1} v_{\hat{k}} \leq \rho^{\bar{k}+ B + 1} V_{0} $) we have $ B + 1 - n \geq 0 $ since $ 1 \leq n \leq B + 1 $.
Thus, for $ 1 \leq n \leq B $ we have
\begin{equation}\label{inda2}
\begin{split}
       &  v_{\bar{k}+B+2+n}  \leq a_{1}v_{\bar{k}+B +n+1}+ a_{2,\bar{k}+ B+n+1} \max\limits_{\bar{k} + n + 1 \leq \hat{k} \leq \bar{k} + B + n+1 } v_{\hat{k}} \\
 & \leq a_{1}\rho ^ {\bar{k}+B+n+1} V_{0} + a_{2,\bar{k}+B+n+1}\rho ^{\bar{k}+B+1-(B-n)} V_{0}  \\
     & = (a_{1} + a_{2,\bar{k}+B+n+1}\rho^{-B}) \rho ^ {\bar{k}+B+n+1} V_{0} 
       \leq \rho^{\bar{k} + B + n + 2} V_{0} \ \ \ a.s.
\end{split}
\end{equation}
and $ \max\limits_{\bar{k}+n+1 \leq \hat{k} \leq \bar{k} + B + n + 1} v_{\hat{k}} \leq \rho^{\bar{k}+ B + 1} V_{0} $  with $ v_{k} \leq \rho^{k} V_{0} $ \\
(i.e., $ v_{\bar{k}+B+n+1} \leq \rho^{\bar{k}+B+n+1} V_{0} $) to be used inductively in the induction step of \eqref{inda2} giving the latter term as the final result. 
N.B. Notice in the second inequality of \eqref{indb2} that $ n \leq B $ and $ n \geq 1 $, that is $ B \geq 1 $ since for $ B = 0 $ we have $ n =0 $ and we have only the base case, no induction case for this step.
Therefore, from $ k = \bar{k}+B+2 $ until $ k = \bar{k}+2B+2 $ we have $ v_{k} \leq \rho^{k} V_{0} $.

For $ k \geq \bar{k}+2B+3 $ we have for base case $ k= \bar{k}+2B+3 $ that $ v_{\bar{k}+2B+2} \leq \rho^{\bar{k}+2B+2} V_{0} $ and $  \max\limits_{\bar{k} + B + 2 \leq \hat{k} \leq \bar{k} + 2B + 2} v_{\hat{k}} \leq  \rho^{\bar{k}+B+2} V_{0} $ (i.e., $ v_{k} \leq \rho^{k} V_{0} $ and $ \max\limits_{k-B \leq \hat{k} \leq k} v_{\hat{k}} \leq \rho^{k-B} V_{0} $)  and for induction case $ k \geq \bar{k}+2B+4 $, we also have $ v_{k} \leq \rho^{k} V_{0} $ and $ \max\limits_{k-B \leq \hat{k} \leq k} v_{\hat{k}} \leq \rho^{k-B} V_{0} $, therefore, we can join both cases in one inequality 
\begin{equation}\label{inda3}
\begin{split}
 v_{k+1} & \leq a_{1}v_{k}+ a_{2,k}\max\limits_{k-B \leq \hat{k} \leq k} v_{\hat{k}} \leq a_{1}\rho ^ {k} V_{0} + a_{2,k}\rho ^{k-B} V_{0} \\
 & = (a_{1} + a_{2,k}\rho^{-B}) \rho ^ {k} V_{0} \leq \rho^{k+1} V_{0} \ \ \ a.s.
 \end{split}
\end{equation}
and $ \max\limits_{k-B \leq \hat{k} \leq k} v_{\hat{k}} \leq \rho^{k-B} V_{0} $ and $ v_{k} \leq \rho ^ {k} V_{0}$ (i.e., $ v_{k+1} \leq \rho ^ {k+1} V_{0}$) to be used inductively in \eqref{inda3} induction step and the latter term as the final result.

Therefore, the lemma follows.

\subsection{Martingale 2}

{\bf Proof of Lemma~5:}

Assume the following inequality holds a.s. for all $ k \geq k^{*} $
\begin{equation}
\begin{split}
v_{k+1} \leq a_{1}v_{k} + a_{2,k} \max\limits_{k-B \leq \hat{k} \leq k}   v_{k} + a_{3,k}
\end{split}
\end{equation}
where the variables are as defined in hypothesis.

We can list any consecutive $ B + 1 $ terms in an increasing order. Let us say we choose $ B+1 $ consecutive instants terms $ v_{k} $ beginning from instant $ k_{0} = \bar{k} - B $ until $ \bar {k} $. Take $ \rho =(a_{1} + a_{2,\bar{k}})^{\frac{1}{B+1}}  $ then $ v^{1} \leq v^{2} \leq \ldots \leq v^{B+1} $ where $ v^{l} = v_{\phi(l)} $ such that $ \phi(l) $ is a bijective mapping from $ [1, B+1] \rightarrow [ \bar{k} - B, \bar{k}]$. Then we can bound each $ v^{i} \leq \rho ^{B+1-i+m} V_{0} $ where $ 1 \leq i \leq B+1 $ and $ V_{0} = \max \frac{v^{i}}{\rho ^{i}} $, i.e., $ \rho < 1 $ and $ \rho ^{n} < \rho $ and $ \eta $ as in hypothesis. And $ m $ is an arbitrary constant such that $ m \in \mathbb{Z}^{+} $. Let us choose $ \bar{k} \geq k^{*} - 1 $ where $ \bar {k} \geq B $ (i.e.,if we index from $ k=0 $). For a neat final result we choose $ m = \bar{k} $, that is, $ v^{i} = v_{\phi(\iota)} \leq \rho ^{B+1-i+\bar{k}} V_{0} $, i.e., $ v^{1} \leq \rho ^{\bar{k}+B}V_{0},\ldots, v^{B+1} \leq \rho ^{\bar{k}} V_{0} $.

For $ v_{\bar{k}} $ we have 
\begin{equation}
\begin{split}
v_{\bar{k}} \leq \rho ^ {\bar{k} + B - l } V_{0}
\end{split}
\end{equation}
where $ 0 \leq l \leq B $ 
since the terms $ v_{k} $ at instants from $ k_{0} = \bar{k} - B $ till $ \bar{k} $ can be put as defined earlier and $ v_{\bar{k}} $ can be any term in that order.

Take $ \bar{k} + 1 \geq k^{*} $. 

First since $ a_{1} + a_{2,\bar{k}} \leq 1 $ and  $ a_{1}+a_{2,k} \leq a_{1}+a_{2,\bar{k}} $ for $ k \geq \bar{k} $.

then we have $ 1  \leq (a_{1}+a_{2,\bar{k}})^{-\frac{B}{B+1}} $ and $ 1  \leq (a_{1}+a_{2,k})^{-\frac{B}{B+1}} $
which implies for $ k \geq \bar{k} $ that
\begin{align*}
a_{1} +  a_{2,k}  \rho ^ {-B} & = a_{1} + a_{2,k}(a_{1}+a_{2,\bar{k}})^{-\frac{B}{B+1}} \\
& \leq a_{1}(a_{1}+a_{2,\bar{k}})^{-\frac{B}{B+1}} 
+ a_{2,k}(a_{1}+a_{2,\bar{k}})^{-\frac{B}{B+1}} \\
& = (a_{1}+a_{2,k})(a_{1}+a_{2,\bar{k}})^{-\frac{B}{B+1}} \\
& \leq (a_{1}+a_{2,\bar{k}})^{\frac{1}{B+1}} = \rho
\end{align*}
That is 
{ \begin{equation}\label{rho1}
\begin{split}
a_{1} + & a_{2,k}  \rho ^ {-B}  \leq \rho
\end{split}
\end{equation}} 
N.B. We aim to find the tightest upper bound for each case and the bound that holds for all cases, i.e., that takes into consideration the worst case possibility.

For Base case $ \bar{k} + 1 \geq k^{*} $ we have $ v_{\bar{k}} \leq \rho ^ {\bar{k} + B - l } V_{0} +\eta $ where $ 0 \leq l \leq B $ and $ \max\limits_{\bar{k}-B \leq \hat{k} \leq \bar{k}} v_{\hat{k}} \leq \rho^{\bar{k}}V_{0} +\eta $, then
\begin{align*}
 v_{\bar{k}+1} & \leq a_{1}v_{\bar{k}}+ a_{2,\bar{k}}\max\limits_{\bar{k}-B \leq \hat{k} \leq \bar{k}} v_{\hat{k}} + a_{3,k}  \\
 & \leq a_{1}\rho ^ {\bar{k}+B-l} V_{0} + a_{2,\bar{k}}\rho ^{\bar{k}} V_{0} + a_{3,\bar{k}} \\ 
 & \leq   a_{1}\rho ^ {\bar{k} + B - l - (B-l)} V_{0} + a_{2,\bar{k}}\rho ^{\bar{k}} V_{0} + a_{3,\bar{k}} \\
 & = (a_{1} + a_{2,\bar{k}}) \rho ^ {\bar{k}} V_{0} + b_{\bar{k}+1} \eta
      \leq \rho^{\bar{k} + B + 1} V_{0} + b_{\bar{k}+1} \eta \ \ \ a.s.
\end{align*}
For $ \bar{k} + 2 $ we can have $ \max\limits_{\bar{k}+1-B \leq \hat{k} \leq \bar{k} + 1} v_{\hat{k}} \leq \rho^{\bar{k}+1} V_{0} + b_{\bar{k}+1} \eta $ or $ \max\limits_{\bar{k}+1-B \leq \hat{k} \leq \bar{k} + 1} v_{\hat{k}} \leq \rho^{\bar{k}} V_{0} + b_{\bar{k}+1} \eta $ if the maximum which is $ \leq \rho^{\bar{k}}V_{0} + b_{\bar{k}+1} \eta $ is in $ v_{\bar{k}-B} $ or not, respectively.
Similarly with case $ \bar{k}+2 $ included and for $ \bar{k} +1 + n $ where $ 1 \leq n \leq B $ we can have  
$ \max\limits_{\bar{k}+n-B \leq \hat{k} \leq \bar{k} + n} v_{\hat{k}} \leq \rho^{\bar{k}+i} V_{0} + b_{\bar{k}+1} \eta$ for each $ 0 \leq i \leq n $. And for each of these cases the set $ \{ v_{\bar{k}-B}, v_{\bar{k}-B+1},\ldots,v_{\bar{k}-B+n+1} \} $ contains terms that are $ \{ \rho^{\bar{k}}V_{0} + b_{\bar{k}+1} \eta, \rho^{\bar{k}+1}V_{0} + b_{\bar{k}+1} \eta, \ldots, \rho^{\bar{k}+i-1}V_{0} + b_{\bar{k}+1} \eta \} $ for each $ 0 \leq i \leq n $.

For base case $ k = \bar{k}+1$ we have $ v_{\bar{k}+1} \leq V_{0} \rho^{\bar{k}+B+1} + b_{\bar{k}+1} \eta $ as we have proved.
For induction case (i.e., we have $ v_{\bar{k}+n} \leq V_{0} \rho^{\bar{k}+B+1} + b_{\bar{k}+n} \eta $ and $ \max\limits_{\bar{k} + n - B \leq \hat{k} \leq \bar{k} + n} v_{\hat{k}} \leq \rho^{\bar{k}+i} V_{0} + b_{\bar{k}+1} \eta $ as mentioned in the previous paragraph) $ k = \bar{k}+2 $ until $ k =\bar{k}+B+1 $ we have $ B+1-i \geq 0 $ since $ 0 \leq i \leq n \leq B $. Then for $ 1 \leq n \leq B $ we have
 \begin{equation*}\label{indb1}
\begin{split}
   & v_{\bar{k}+1+n}  \leq a_{1}v_{\bar{k}+n}+ a_{2,\bar{k}+n}\max\limits_{\bar{k} + n - B \leq \hat{k} \leq \bar{k} + n} v_{\hat{k}} + a_{3,\bar{k}+n} \\
 & \leq a_{1}\rho ^ {\bar{k}+B+1} V_{0} + a_{2,\bar{k}+n}\rho ^{\bar{k}+i} V_{0} + a_{1}b_{\bar{k}+n}\eta + a_{2,\bar{k}+n}b_{\bar{k}+1}\eta \\
 & \hspace{3cm} + a_{3,\bar{k}+n} \\
      & \leq a_{1}\rho ^ {\bar{k}+B+1 - (B+1-i)} V_{0} + a_{2,\bar{k}+n}\rho ^{\bar{k}+i} V_{0} + a_{1}b_{\bar{k}+n}\eta \\
   &  \hspace{2cm} + a_{2,\bar{k}+n}b_{\bar{k}+1}\eta + a_{3,\bar{k}+n} \\
      & = (a_{1} + a_{2,\bar{k}+n}) \rho ^ {\bar{k}+i} V_{0} + a_{1}b_{\bar{k}+n}\eta + a_{2,\bar{k}+n}b_{\bar{k}+1}\eta + a_{3,\bar{k}+n} \\
      & \leq \rho^{\bar{k} + B + 1 + i} V_{0} + a_{1}b_{\bar{k}+n}\eta + a_{2,\bar{k}+n}b_{\bar{k}+1}\eta + a_{3,\bar{k}+n} 
      \end{split}
\end{equation*}
      
 \begin{equation}\label{indb1-1}
      \begin{split}
      & \leq \rho^{\bar{k} + B + 1 + i} V_{0} + (1-\mu)b_{\bar{k}+n}\eta + [(1-\frac{1}{l})\frac{1}{(B+2)^{\theta}}+\mu -1]b_{\bar{k}+1}\eta \\
      & \hspace{3cm} + \frac{b_{\bar{k}+n+1}}{l}\eta \\
      & \leq \rho^{\bar{k} + B + 1 + i} V_{0} + (1-\mu)(b_{\bar{k}+n}-b_{\bar{k}+1})\eta + \frac{b_{\bar{k}+1}}{(B+2)^{\theta}}\eta + \frac{b_{\bar{k}+n+1}}{l}\eta \\ & \leq \rho^{\bar{k} + B + 1 + i} V_{0} + \frac{(1-\frac{1}{l})}{(B+2)^{\theta}(\bar{k}+1+a)^{\theta}}\eta + \frac{1}{l(\bar{k}+n+1+a)^{\theta}}\eta \\
      & \leq \rho^{\bar{k} + B + 1 + i} V_{0} + \frac{(1-\frac{1}{l})}{((B+2)(\bar{k}+1+a))^{\theta}}\eta + \frac{1}{l(\bar{k}+n+1+a)^{\theta}}\eta \\
      & \leq \rho^{\bar{k} + B + 1 + i} V_{0} + \frac{(1-\frac{1}{l})}{((n+2)(\bar{k}+1+a))^{\theta}}\eta + \frac{1}{l(\bar{k}+n+1+a)^{\theta}}\eta \\
      & \leq \rho^{\bar{k} + B + 1 + i} V_{0} + \frac{(1-\frac{1}{l})}{(n\bar{k}+n+na+2\bar{k}+2+2a)^{\theta}}\eta + \frac{1}{l(\bar{k}+n+1+a)^{\theta}}\eta \\
      & \leq \rho^{\bar{k} + B + 1 + i} V_{0} + \frac{(1-\frac{1}{l})}{(\bar{k}+n+1+a)^{\theta}}\eta + \frac{1}{l(\bar{k}+n+1+a)^{\theta}}\eta \\
      & \leq \rho^{\bar{k} + B + 1 + i} V_{0} + b_{\bar{k}+n+1}\eta  
     \end{split}
\end{equation}

Notice for induction case we have $ n \geq 1 $ then $ B \geq n \geq 1 $ so for induction case to exist $ B \geq 1 $ however the base case can exist for any $ B \geq 0 $. Then we choose the values of $ a_{2,\bar{k}} $ accordingly, or to be consistent we take the maximum of the bounds, but we choose the first alternative which depends on whether an induction step is needed or not according to the value of $ B $.
But the most relaxed bound that takes into consideration all cases $  0 \leq i \leq n $ for each $ n $ is $ v_{\bar{k}+1+n} \leq V_{0}\rho^{\bar{k}+B+1} +  b_{\bar{k}+n+1} \eta $ and $ \max\limits_{\bar{k} + n - B \leq \hat{k} \leq \bar{k} + n} v_{\hat{k}} \leq \rho^{\bar{k}+i} V_{0} + b_{\bar{k}+1} \eta $ as mentioned in the previous paragraph. 
\\
Thus, for $ k = \bar{k} +1 $ until $ k = \bar{k} +B +1 $ we have $ v_{k} \leq V_{0} \rho ^{\bar{k}+B + 1} +  b_{k} \eta $ which is in turn used in the induction step of \eqref{indb1-1} and is the subsequent result.
 
For base case $ k = \bar{k} + B + 2 $ we have $ v_{\bar{k}+B +1} \leq \rho ^ {\bar{k}+B+1} V_{0} +  b_{\bar{k}+B+1} \eta $ and $ \max\limits_{\bar{k}+1 \leq \hat{k} \leq \bar{k} + B + 1} v_{\hat{k}} \leq \rho^{\bar{k}+ B + 1} V_{0} +  b_{\bar{k}+1} \eta $ then
 \begin{align*}
v_{\bar{k}+B+2} & \leq a_{1}v_{\bar{k}+B +1}+ a_{2,\bar{k}+B+1}\max\limits_{\bar{k} + 1 \leq \hat{k} \leq \bar{k} + B +1} v_{\hat{k}} + a_{3,\bar{k} + B +1} \\
 & \leq a_{1}\rho ^ {\bar{k}+B+1} V_{0} + a_{2,\bar{k}+B+1}\rho ^{\bar{k}+B+1-B} V_{0} \\
 & + a_{1} b_{\bar{k} + B +1}\eta + a_{2,\bar{k}+B+1}b_{\bar{k}+1}\eta + a_{3,\bar{k} + B +1}  \\
 & = (a_{1} + a_{2,\bar{k}+B+1}\rho^{-B}) \rho ^ {\bar{k}+B+1} V_{0} \\
 & \hspace{2cm} + a_{1} b_{\bar{k} + B +1}\eta + a_{2,\bar{k}+B+1}b_{\bar{k}+1}\eta + a_{3,\bar{k} + B +1}  \\
 & \leq \rho^{\bar{k} + B + 2} V_{0} + a_{1} b_{\bar{k} + B +1}\eta + a_{2,\bar{k}+B+1}b_{\bar{k}+1}\eta + a_{3,\bar{k} + B +1} \\
 & \leq \rho^{\bar{k} + B + 2} V_{0} +  (1-\mu)b_{\bar{k} + B +1}\eta + (\frac{1-\frac{1}{l}}{(B+2)^{\theta}} + \mu - 1 ) b_{\bar{k}+1}\eta + \frac{b_{\bar{k} + B +2}}{l} \eta \\
& \leq \rho^{\bar{k} + B + 2} V_{0} +  + (1-\mu)(b_{\bar{k} + B +1}\eta-b_{\bar{k}+1}) + \frac{1-\frac{1}{l}}{(B+2)^{\theta}}b_{\bar{k}+1}\eta + \frac{b_{\bar{k} + B +2}}{l} \eta \\
\end{align*}

 \begin{align*} 
& \leq \rho^{\bar{k} + B + 2} V_{0} + (1-\mu)(b_{\bar{k} + B +1}\eta-b_{\bar{k}+1}) + \frac{1-\frac{1}{l}}{(B+2)^{\theta}(k+1)^{\theta}}\eta + \frac{b_{\bar{k} + B +2}}{l} \eta \\
& \leq \rho^{\bar{k} + B + 2} V_{0} + \frac{1-\frac{1}{l}}{((B+2)(\bar{k}+1+a))^{\theta}}\eta + \frac{1}{l(\bar{k} + B +2+a)^{\theta}} \eta \\
& \leq \rho^{\bar{k} + B + 2} V_{0} + \frac{1-\frac{1}{l}}{((B+2)(\bar{k}+1+a))^{\theta}}\eta+ \frac{1}{l(\bar{k} + B + 2 + a)^{\theta}} \eta \\
& \leq \rho^{\bar{k} + B + 2} V_{0} + \frac{1-\frac{1}{l}}{(B\bar{k}+B+B+2k+2+2a)^{\theta}}\eta+ \frac{1}{l(\bar{k} + B + 2 + a)^{\theta}} \eta \\
& \leq \rho^{\bar{k} + B + 2} V_{0} + \frac{1-\frac{1}{l}}{(\bar{k}+B+2+a)^{\theta}}\eta+ \frac{1}{l(\bar{k} + B + 2 + a)^{\theta}} \eta \\
& \leq \rho^{\bar{k} + B + 2} V_{0} + \frac{1}{(\bar{k}+B+2+a)^{\theta}}\eta \\
& \leq \rho^{\bar{k} + B + 2} V_{0} + b_{\bar{k}+B+2}\eta 
\end{align*}
Thus, for induction case $ k = \bar{k}+B+1+n $ where $ 1 \leq n \leq B + 1 $, that is, from $  k = \bar{k} + B + 2 $ until $ k = \bar{k}+2B+2 $ (i.e., we have $ v_{\bar{k}+B+n} \leq \rho^{\bar{k}+B+n} V_{0} + b_{\bar{k}+B+n} \eta $, that is, $ v_{k} \leq \rho^{k} V_{0} + b_{k} \eta $ or  $ v_{\bar{k}+B+n+1} \leq \rho^{\bar{k}+B+n+1} V_{0} + b_{\bar{k}+B+n+1} \eta $, and $ \max\limits_{\bar{k}+n \leq \hat{k} \leq \bar{k} + B + n} v_{\hat{k}} \leq \rho^{\bar{k}+ B + 1} V_{0} + b_{\bar{k}+n} \eta $) we have $ B + 1 - n \geq 0 $ since $ 1 \leq n \leq B + 1 $
 \begin{equation}\label{indb2}
\begin{split}
    & v_{\bar{k}+B+2+n}  \leq a_{1}v_{\bar{k}+B +n + 1}+ a_{2,\bar{k}+ B+n+1}\max\limits_{\bar{k} + n +1\leq \hat{k} \leq \bar{k} + B + n + 1} v_{\hat{k}} \\
    & \hspace{3cm} + a_{3,\bar{k}+B +n + 1} \\ 
    & \leq a_{1}\rho ^ {\bar{k}+B+n+1} V_{0} + a_{2,\bar{k}+B+n+1}\rho ^{\bar{k}+B+1-(B-n)} V_{0}  + a_{1}b_{\bar{k}+B +n+1}\eta \\
    & \hspace{2cm} + a_{2,\bar{k}+B+n+1}b_{\bar{k}+n+1}\eta + a_{3,\bar{k}+B +n+1}  \\
     & = (a_{1} + a_{2,\bar{k}+B+n+1}\rho^{-B}) \rho ^ {\bar{k}+B+n+1} V_{0} a_{1}b_{\bar{k}+B +n+1}\eta \\
     & \hspace{2cm} + a_{2,\bar{k}+B+n+1}b_{\bar{k}+n+1}\eta + a_{3,\bar{k}+B +n+1} \\
     & \leq \rho^{\bar{k} + B + n + 2} V_{0} a_{1}b_{\bar{k}+B +n+1}\eta + a_{2,\bar{k}+B+n+1}b_{\bar{k}+n+1}\eta + a_{3,\bar{k}+B +n+1} 
\end{split}    
\end{equation}

\vspace{-0.75cm}
 \begin{equation}\label{indb2-1}
\begin{split}
& \leq \rho^{\bar{k} + B + n + 2} V_{0} + (1-\mu)b_{\bar{k}+B +n+1} \eta + (\frac{1-\frac{1}{l}}{(B+2)^{\theta}}+\mu-1)b_{\bar{k}+n+1} \eta \\
& \hspace{3cm} + \frac{b_{\bar{k}+B +n+2}}{l} \eta  \\
& \leq \rho^{\bar{k} + B + n + 2} V_{0} + (1-\mu)(b_{\bar{k}+B +n+1} - b_{\bar{k}+n+1}) \eta +\frac{1-\frac{1}{l}}{(B+2)^{\theta}}b_{\bar{k}+n+1} \eta \\
& \hspace{3cm} + \frac{b_{\bar{k}+B +n+2}}{l} \eta  \\
& \leq \rho^{\bar{k} + B + n + 2} V_{0} + \frac{1-\frac{1}{l}}{(B+2)^{\theta}(\bar{k}+n+1+a)^{\theta}} \eta + \frac{1}{l(\bar{k}+B +n+2+a)^{\theta}} \eta  \\
& \leq \rho^{\bar{k} + B + n + 2} V_{0} + \frac{1-\frac{1}{l}}{((B+2)(\bar{k}+n+1+a))^{\theta}} \eta + \frac{1}{l(\bar{k}+B +n+2+a)^{\theta}} \eta  \\
& \leq \rho^{\bar{k} + B + n + 2} V_{0} + \frac{1-\frac{1}{l}}{(B\bar{k}+Bn+B+Ba+2\bar{k}+2n+2+2a)^{\theta}} \eta \\
& \hspace{3cm} + \frac{1}{l(\bar{k}+B +n+2+a)^{\theta}} \eta  \\
& \leq \rho^{\bar{k} + B + n + 2} V_{0} + \frac{1-\frac{1}{l}}{(\bar{k}+B+n+2+a)^{\theta}} \eta + \frac{1}{l(\bar{k}+B +n+2+a)^{\theta}} \eta  \\
& \leq \rho^{\bar{k} + B + n + 2} V_{0} + \frac{1}{(\bar{k}+B+n+2+a)^{\theta}} \eta \\
& \leq \rho^{\bar{k} + B + n + 2} V_{0} + b_{\bar{k}+B+n+2} \eta 
\end{split}    
\end{equation}
and $ \max\limits_{\bar{k}+n+1 \leq \hat{k} \leq \bar{k} + B + n+1} v_{\hat{k}} \leq \rho^{\bar{k}+ B + 1} V_{0} + b_{\bar{k}+n+1} \eta $  with $ v_{k} \leq \rho^{k} V_{0} + b_{k} \eta $ \\ 
(i.e., $ v_{\bar{k}+B+n+1} \leq \rho^{\bar{k}+B+n+1} V_{0} + b_{\bar{k}+B+n+1} \eta $) to be used inductively in the induction step of \eqref{indb2-1} giving the latter term as the final result. 
\\
Therefore, from $ k = \bar{k}+B+2 $ until $ k = \bar{k}+2B+2 $ we have $ v_{k} \leq \rho^{k} V_{0} + b_{k} \eta $.

For $ k \geq \bar{k}+2B+3 $ we have for base case $ k= \bar{k}+2B+3 $ that $ v_{\bar{k}+2B+2} \leq \rho^{\bar{k}+2B+2} V_{0} + b_{\bar{k}+2B+2} \eta $ and $  \max\limits_{\bar{k} + B + 2 \leq \hat{k} \leq \bar{k} + 2B + 2} v_{\hat{k}} \leq  \rho^{\bar{k}+B+2} V_{0} + b_{\bar{k} + B + 2} \eta $ (i.e., $ v_{k} \leq \rho^{k} V_{0} + b_{k} \eta $ and $ \max\limits_{k-B \leq \hat{k} \leq k} v_{\hat{k}} \leq \rho^{k-B} V_{0} + b_{k-B} \eta $) and for induction case $ k \geq \bar{k}+2B+4 $, we also have $ v_{k} \leq \rho^{k} V_{0} + b_{k} \eta $ and $ \max\limits_{k-B \leq \hat{k} \leq k} v_{\hat{k}} \leq \rho^{k-B} V_{0} + b_{k-B} \eta $, therefore, we can join both cases in one inequality 
 \begin{equation}\label{indb3}
\begin{split}
  v_{k+1} & \leq a_{1}v_{k}+ a_{2,k}\max\limits_{k-B 
 \leq \hat{k} \leq k} v_{\hat{k}} + a_{3,k} \\
 & \leq a_{1}\rho ^ {k} V_{0} + a_{2,k}\rho ^{k-B} V_{0} + a_{1}b_{k}\eta + a_{2,k}b_{k-B}\eta + a_{3,k} \\
 & = (a_{1} + a_{2,k}\rho^{-B}) \rho ^ {k} V_{0} + a_{1}b_{k}\eta + a_{2,k}b_{k-B}\eta + a_{3,k} \\
 & \leq \rho^{k+1} V_{0} + a_{1}b_{k}\eta + a_{2,k}b_{k-B}\eta + a_{3,k} 
\end{split}    
\end{equation}
 \begin{equation}\label{indb3-1}
\begin{split}
& \leq \rho^{k+1} V_{0} + (1-\mu)b_{k}\eta + (\frac{1-\frac{1}{l}}{(B+2)^{\theta}}+\mu -1)b_{k-B} \eta + \frac{b_{k+1}}{l} \eta \\
& \leq \rho^{k+1} V_{0} + (1-\mu)(b_{k}-b_{k-B}) \eta + \frac{1-\frac{1}{l}}{(B+2)^{\theta}}b_{k-B} \eta + \frac{b_{k+1}}{l} \eta \\
& \leq \rho^{k+1} V_{0} + \frac{1-\frac{1}{l}}{(B+2)^{\theta}(k-B+a)^{\theta}} \eta + \frac{1}{l(k+1+a)^{\theta}} \eta \\
& \leq \rho^{k+1} V_{0} + \frac{1-\frac{1}{l}}{((B+2)(k-B+a))^{\theta}} \eta + \frac{1}{l(k+1+a)^{\theta}} \eta \\
& \leq \rho^{k+1} V_{0} + \frac{1-\frac{1}{l}}{(Bk-B^{2}+Ba+2k-2B+2a)^{\theta}} \eta + \frac{1}{l(k+1+a)^{\theta}} \eta \\
& \leq \rho^{k+1} V_{0} + \frac{1-\frac{1}{l}}{(k(B+1)-B^{2}+Ba+k-2B+2a)^{\theta}} \eta \\
& \hspace{3cm} + \frac{1}{l(k+1+a)^{\theta}} \eta \\
& \leq \rho^{k+1} V_{0} + \frac{1-\frac{1}{l}}{((2B+2)(B+1)-B^{2}+Ba+k-2B+2a)^{\theta}} \eta \\
& \hspace{3cm}+ \frac{1}{l(k+1+a)^{\theta}} \eta \\
& \leq \rho^{k+1} V_{0} + \frac{1-\frac{1}{l}}{(B^{2}+4B+2+Ba+k-2B+2a)^{\theta}} \eta + \frac{1}{l(k+1+a)^{\theta}} \eta \\
& \leq \rho^{k+1} V_{0} + \frac{1-\frac{1}{l}}{(k+a+1)^{\theta}} \eta + \frac{1}{l(k+1+a)^{\theta}} \eta \\
& \leq \rho^{k+1} V_{0} + \frac{1}{(k+a+1)^{\theta}} \eta \\
& \leq \rho^{k+1} V_{0} + b_{k+1} \eta  
\end{split}    
\end{equation}
since we have used before the last inequality that $ Bk- B^{2} - B \geq 0 $ since $ k \geq \bar{k} +2B+2 $ where $ \bar{k} \geq 0 $ and $ B \geq 0 $.

And $ \max\limits_{k+1-B \leq \hat{k} \leq k+1} v_{\hat{k}} \leq \rho^{k+1-B} V_{0} + b_{k+1-B} \eta $ and $ v_{k} \leq \rho ^ {k} V_{0} + b_{k} \eta $ (i.e., $ v_{k+1} \leq \rho ^ {k+1} V_{0} + b_{k+1} \eta $) to be used inductively in \eqref{indb3} induction step and the latter term as the final result.

Therefore, the lemma follows.

\subsection{$ \rho_{3} $ and $ \eta_{3} $}

 \begin{equation}\label{ConvexConvergenceRateRho}
\begin{split}
\rho_{3}= & ((1-\mu) + 4 p L \alpha_{\bar{k}_{3}} \| {\bf A}^{r_{(\iota)}} \| _{\infty} \|{\bf B}^{r_{(\iota)}}\| _{2,\infty} \\
& + 8 p (1-\gamma_{(0)}) L^{2} \alpha_{\bar{k}_{3}}^{2}  \| {\bf A}^{r_{(\iota)}} \| ^{2} _{\infty} \|{\bf B}^{r_{(\iota)}}\| ^{2} _{2,\infty} \\
& + 8 (1-\gamma_{(0)}) \alpha_{\bar{k}_{3}}^{2}  p \gamma_{max}L ^{2} + \alpha_{k}( 1 + 4 ( 1 - \gamma_{(0)}) \alpha_{k} L)p L)^{\frac{1}{H+1}}
\end{split}
\end{equation}

\begin{equation}\label{ConvexConvergenceRateeta}
\begin{split}
\hspace{-1.5cm} \eta_{3} = \eta_{3,1}/\eta_{3,2}
\end{split}
\end{equation}

\begin{equation}
\begin{split}
\eta_{3,1}= & 2 \alpha_{\bar{k}_{3}} \sum_{j=1}^{n}(|I|+2p)\gamma_{max} L \| {\bf x}^{*} - \bar{\bf x}^{(\iota)} \| ^{2} \\
 & +  \frac{1}{1-\mu} \alpha_{k}n (1 + 4 (1-\gamma_{(0)}) \alpha_{k} L \max\limits_{(\iota)}\| {\bf A}^{r_{(\iota)}} \| _{\infty} \|{\bf B}^{r_{(\iota)}} \| _{2,\infty}  ) \\
 & \times p L  \max\limits_{(\iota)} \| {\bf A}^{r_{(\iota)}} \| _{\infty} \| {\bf B}^{r_{(\iota)}} \| _{2,\infty} \| {\bf x}^{*}- {\bf x}^{\lambda,r_{(\iota)}} \| ^{2} 
\end{split}
\end{equation} 

\begin{equation}
\begin{split} 
\eta_{3,2} = & \mu - 4 p L \alpha_{\bar{k}_{3}} \max\limits_{(\iota)} \| {\bf A}^{r_{(\iota)}} \| _{\infty} \|{\bf B}^{r_{(\iota)}}\| _{2,\infty} \\
& - 8 p (1-\gamma_{(0)}) L^{2} \alpha_{\bar{k}_{3}}^{2}  \max\limits_{(\iota)} \| {\bf A}^{r_{(\iota)}} \| ^{2} _{\infty} \|{\bf B}^{r_{(\iota)}}\| ^{2} _{2,\infty} 
 \\
 & - 8 (1-\gamma_{(0)}) \alpha_{\bar{k}_{3}}^{2}  p \gamma_{max}L ^{2} + \alpha_{k}( 1 + 4 ( 1 - \gamma_{(0)}) \alpha_{k} L)p L 
\end{split}
\end{equation}

\subsection{Gradient Coding Algorithms}

The following are the gradient decoding and encoding algorithms, respectively.

\begin{algorithm}[H]
\caption{Algorithm to compute {\bf A}}
\label{alg:algorithm-A}
\begin{algorithmic}[1]
\State{\bf Input:} \ {\bf B} satisfying Condition~4.1, $ s (< n) $  
\State{ {\bf Output:} \ {\bf A} such that {\bf A}{\bf B} = ${\bf 1}_{{n\choose{s}} \times n} $ }  
\State {f = binom(n, s);}  
\State A = zeros(f, n); 
\State {\bf foreach} \ $ I \subset [n] $ \ s.t. \ $ |I| = (n-s) $ \ {\bf do} 
\State a = zeros(1, k);  
\State x = ones(1, k)/B(I, :);
\State $a(I)$ = x;  
\State A = [A; a];  
\State {\bf end}
\end{algorithmic}
\end{algorithm}

\begin{algorithm}[H]
\caption{Algorithm to construct $ {\bf B}= {\bf B}_{cyc}$}
\label{alg:algorithm-B}
\begin{algorithmic}[1]
\State {\bf Input:} \ $ n, s( < n) $
\State {\bf Output:} \ $ {\bf B} \in \mathbf{R}^{n \times n} $
with $ (s+1) $ non-zeros in each row
\State $H = randn(s, n)$;
\State $H(:, n) =-sum(H(:,1:n-1),2)$;
\State B = zeros(n);
\State {\bf for} \  $i = 1 : n$ \ {\bf do}
\State j = mod(i-1:s+i-1,n)+1;
\State B(i, j) = $ [1; -H(:,j(2:s+1)) \backslash H(:,j(1))] $;
\State {\bf end}
\end{algorithmic}
\end{algorithm}

\bibliographystyle{c}
\bibliography{main}

\end{document}